\documentclass[leqno,12pt]{amsart}

\usepackage{amsmath,amsthm,amssymb,mathrsfs,mathtools}
\usepackage{graphicx,tikz}

\DeclareRobustCommand{\gobblefive}[5]{}
\newcommand*{\SkipTocEntry}{\addtocontents{toc}{\gobblefive}}
\usepackage{CJKutf8} 
\newcommand{\n}{{\nabla}}
\newcommand{\FF}{{E_\nabla}}
\newcommand{\FFsr}{{E_\nabla^{\sharp_r}}}
\newcommand{\FFs}{{E_\nabla^\sharp}}
\newcommand{\FFtw}{{E^2_\nabla}}
\newcommand{\FFth}{{E^3_\nabla}}

\newcommand{\la}{\langle}
\newcommand{\ra}{\rangle}

\newcommand{\bb}{{\beta^\sharp}}

\newcommand{\bbt}{{\beta^{\sharp^{g_t}}}}

\newcommand{\gi}{{G^{-1}_\beta}}

\usepackage[T1]{fontenc}
\usepackage[english]{babel}
\usepackage[margin=1in]{geometry}
\usepackage[onehalfspacing]{setspace}
\usepackage[colorlinks,pagebackref, hypertexnames=false, bookmarks=false]{hyperref}
\usepackage[alphabetic,backrefs]{amsrefs}
\usepackage{xcolor}
\hypersetup{
    allcolors = blue
}
\newcounter{dummy}
\usepackage{enumitem}
\makeatletter
\newcommand\myitem[1][]{\item[#1]\refstepcounter{dummy}\def\@currentlabel{#1}}
\makeatother
\usepackage{ae,aecompl}
\usepackage{multicol}
\usepackage[normalem]{ulem}
\usepackage{tabularx}
\numberwithin{equation}{section}							% must call it before cleveref
\usepackage[nameinlink,noabbrev]{cleveref}
\usepackage{autonum}										% must call after cleveref

\let\originalleft\left
\let\originalright\right
\renewcommand{\left}{\mathopen{}\mathclose\bgroup\originalleft}
\renewcommand{\right}{\aftergroup\egroup\originalright}

\makeatletter
\renewcommand*{\eqref}[1]{\hyperref[{#1}]{\textup{\tagform@{\ref*{#1}}}}}		% eqref now links parentheses as well
\makeatother

\newcommand\rl{\rm{l}}

\newcommand{\cA}{\mathcal{A}}

\newcommand{\cF}{\mathcal{F}}

\newcommand{\cM}{\mathcal{M}}

\newcommand{\fg}{{\mathfrak g}}

\newcommand{\Spin}{\rm{Spin}}
\newcommand{\GL}{\mathrm{GL}}

\renewcommand{\det}{\mathop\mathrm{det}\nolimits}

\newcommand{\eps}{\epsilon}

\newcommand{\Ric}{\rm{Ric}}

\newcommand{\id}{\mathrm{id}}

\newcommand{\tr}{\mathop{\mathrm{tr}}\nolimits}

\def\cx{\mathbb{C}}

\def\rl{\mathbb{R}}

\def\Z{\mathbb{Z}}

\def\GL{\mathrm{GL}}

\def\Ric_g{\mathrm{Ric}}

\def\Spin{\mathrm{Spin}}

\def\tr{\mathrm{tr}}

\def\vol{\mathrm{vol}}

\def\<{\mathopen{}\left<}
\def\>{\right>\mathclose{}}
\def\({\mathopen{}\left(}
\def\){\right)\mathclose{}}

\renewcommand{\epsilon}{\varepsilon}

\renewcommand{\i}{{\sqrt{-1}}}

\newtheorem{theorem}{Theorem}[section]

\newtheorem{corollary}[theorem]{Corollary}

\newtheorem{lemma}[theorem]{Lemma}
\newtheorem{proposition}[theorem]{Proposition}
\theoremstyle{definition} \newtheorem{definition}[theorem]{Definition}

\newtheorem{remark}[theorem]{Remark}

\crefname{theorem}{Theorem}{Theorems}						% label for Theorems
\creflabelformat{theorem}{#2{#1}#3}							% label format for 'theorem'
\crefname{Mtheorem}{Main Theorem}{Main Theorems}			% label for the Main Theorems
\creflabelformat{Mtheorem}{#2{#1}#3}						% label format for 'Mtheorem'
\crefname{lemma}{Lemma}{Lemmata}							% label for Lemmata
\creflabelformat{lemma}{#2{#1}#3}							% label format for 'lemma'
\crefname{corollary}{Corollary}{Corollaries}				% label for Corollaries
\creflabelformat{corollary}{#2{#1}#3}						% label format for 'corollary'
\crefname{proposition}{Proposition}{Propositions}			% label for Propositions
\creflabelformat{proposition}{#2{#1}#3}						% label format for 'proposition'
\crefname{ineq}{inequality}{inequalities}					% label for inequalities
\creflabelformat{ineq}{#2{\upshape(#1)}#3}					% label format for 'ineq'
\crefname{cond}{condition}{conditions}						% label for conditions
\creflabelformat{cond}{#2{\upshape(#1)}#3}					% label format for 'cond'
\crefname{hypoth}{Hypothesis}{Hypotheses}					% label for Hypotheses
\creflabelformat{hypoth}{#2{#1}#3}							% label format for 'hypoth'
\crefname{def}{Definition}{Definitions}						% label for Definitions
\creflabelformat{def}{#2{#1}#3}								% label format for 'def'
\crefname{appsec}{Appendix}{Appendices}

\crefname{sec}{Section}{Sections}

\begin{document}
%\begin{CJK}{UTF8}{min}
%Japanese

\author{Kotaro Kawai}
\address{Beijing Institute of Mathematical Sciences and Applications, 
No. 544, Hefangkou Village, Huaibei Town, Huairou District, Beijing, 101408, 
China}
\address{
Department of Mathematics, Osaka Metropolitan University, 3-3-138, Sugimoto, Sumiyoshi-ku, Osaka, 558-8585, Japan}
\email{\href{mailto:kkawai@bimsa.cn}{kkawai@bimsa.cn}}

%\date{\today}
\keywords{mirror symmetry, gauge theory, monotonicity formula, 
minimal connections, deformed Donaldson--Thomas connections}
\subjclass[2020]{53C07, 58E15, 53D37}
%\maketitle
%53C07 Special connections and metrics on vector bundles (Hermite-Einstein, Yang--Mills)  
%58E15 Variational problems concerning extremal problems in several variables; Yang--Mills  functionals [See also 81T13], etc.
%53D37 Symplectic aspects of mirror symmetry, homological mirror symmetry, and Fukaya category

\title{A monotonicity formula for minimal connections}
\thanks{
This work was supported by JSPS KAKENHI Grant Number JP21K03231 
and 
MEXT Promotion of Distinctive Joint Research Center Program JPMXP0723833165. 
}

\maketitle

\begin{abstract}
For Hermitian connections on a Hermitian complex line bundle 
over a Riemannian manifold $(X,g)$, 
we can define the ``volume", which can be considered to be 
the ``mirror" of the standard volume for submanifolds. 
We call the critical points minimal connections. 

In this paper, 
(1) we prove monotonicity formulas for minimal connections 
with respect to some versions of volume functionals 
under certain conditions on $\dim X$ and the curvature of $g$. 
These formulas would be important in bubbling analysis. 
As a corollary, we obtain the vanishing theorem for minimal connections 
on the odd dimensional Euclidean space. 

(2)
We see that 
the formal ``large radius limit" of the defining equation of minimal connections 
is that of Yang--Mills connections. 
Then the existence theorem of minimal connections is proved 
for a ``sufficiently large" metric. 

(3) 
We can consider deformed Donaldson--Thomas (dDT) connections on $G_2$-manifolds 
as ``mirrors" of calibrated (associative) submanifolds. 
We show that dDT connections are minimal connections, 
just as calibrated submanifolds are minimal submanifolds. 
By the argument specific to dDT connections, 
we obtain the stronger monotonicity formulas and vanishing theorem 
for dDT connections than in (1). 
\end{abstract}

\tableofcontents

%%%%%%%%%%%%%%%%%%%%%%%%%%%%%%%%%%%%%%%%%%%%%%%%%%%%%%%%%%%%%%%%%%%%

\section{Introduction} 
Let $(X,g)$ be an oriented $n$-dimensional Riemannian manifold 
and $L \to X$ be a smooth complex line bundle with a Hermitian metric $h$. 
Throughout this paper, we assume $n \geq 2$. 
Let $\cA_0$ be the space of Hermitian connections of $(L, h)$. 
The curvature 2-form for $\n \in \cA_0$ is identified with a 2-form on $X$, 
which we denote by $\FF \in \Omega^2$. 
Taking the contraction of $\FF$ and $g$, we define 
$\FFs \in \Gamma(X, {\rm End}(TX))$. 
Then the \textbf{volume functional} $V:\cA_0 \to [0,\infty]$ is defined by 
\begin{align} \label{eq:def vol intro}
V (\n) = \int_X v(\n) \vol_g, 
\qquad
v(\n) = v_g (\n) =\sqrt{\det \left(\id_{TX} + \FFs \right)}, 
\end{align}
where $\vol_g$ is the volume form defined by the Riemannian metric $g$. 
Note that $v(\n) \geq 1$ since $\FFs$ is skew-symmetric. 
See Sections \ref{sec:def vol} and \ref{sec:min conn} for details. 
Thus when $X$ is noncompact of infinite volume, it is useful to define 
the \textbf{normalized volume functional} $V^0:\cA_0 \to \rl$ given by 
\begin{align} \label{eq:def nor vol intro}
V^0 (\n) = \int_X (v(\n)-1) \vol_g. 
\end{align}
Note that 
the functional $V$ is considered as the ``mirror" of the volume functional for submanifolds. 
That is, when $X$ is the total space of a torus fiber bundle,
the functional $V$ corresponds to the volume functional for submanifolds via the real Fourier--Mukai transform by the proof of \cite[Lemma 4.3]{kawai2021FM}.
This is also called the Dirac-Born-Infeld (DBI) action in theoretical physics \cite{mmms2000DBI}. 
A similar volume functional is defined in \cite[Definition 3.1]{jacob2017}. 
See \cite[Remark 6.12]{kawai2021mirror} for the relation between these functionals. 

We call critical points of $V$ or $V^0$ \textbf{minimal connections}. 
Minimal connections are considered as ``mirrors" of minimal submanifolds 
in the sense that minimal submanifolds correspond to minimal connections via 
the real Fourier--Mukai transform as shown in Lemma \ref{lem:FM min}.

More generally, we define the volume functional 
$V=V_g: \Omega^2 \to [0,\infty]$ and the normalized volume functional 
$V^0=V^0_g: \Omega^2 \to [0,\infty]$ as 
\begin{align}
V_g (\beta)= \int_X v_g(\beta) \vol_g,\qquad
V^0_g (\beta)= \int_X (v_g(\beta)-1) \vol_g 
\end{align}
by replacing $\FF$ with a 2-form $\beta$ in \eqref{eq:def vol intro} and 
\eqref{eq:def nor vol intro}, respectively. 
By the variation of $V^0_g$ with respect to the metric $g$, 
we can define the \textbf{stress-energy tensor} for a 2-form, which is a symmetric 2-tensor. 
A 2-form is said to satisfy a \textbf{conservation law} if the stress-energy tensor is divergence free. 
We show that the curvature 2-form of a minimal connection satisfies a conservation law 
in Theorem \ref{thm:min imply cons}. 
This can be understood as a consequence of 
Noether's theorem and the diffeomorphism invariance of $V^0_g (\n)$. 
See Remarks \ref{rem:variational SET} and \ref{rem:variational minimal} 
for details.

The stress-energy tensor was first introduced by Baird and Eells in \cite{Baird1980} in the context of harmonic maps, and there are generalizations in various settings. 
In order to consider these generalizations in a uniform way, 
Ara \cite{Ara1999geom} introduced 
the $F$-harmonic map, 
which is a generalization of the harmonic map by a certain $C^2$ function $F: [0,\infty) \to [0,\infty)$, 
and its associated stress-energy tensor. 
In the same spirit, 
Dong and Wei \cite{Dong2011vanishing} introduced the $F$-Yang--Mills connections 
and its associated stress-energy tensor. 
Then they proved some monotonicity formulas and vanishing theorems.
Wei \cite{Wei2022} extended the results in \cites{Dong2011vanishing, wei2021dual} 
to (normalized) exponential Yang--Mills connections.

\SkipTocEntry \subsection*{Motivation}
We envision that the study of minimal connections would be useful 
for the study of deformed Donaldson--Thomas (dDT) connections on $G_2$-manifolds. 
They were introduced by \cite{lee2009geometric} 
as ``mirrors" of calibrated (associative) submanifolds. 
When a $G_2$-manifold is compact without boundary, 
a dDT connection is a global minimizer of $V$ 
as in the case of calibrated submanifolds by \cite[Corollary 5.4]{kawai2021mirror}, 
and hence it is a minimal connection. 
When a $G_2$-manifold is noncompact, 
we show that 
dDT connections are minimal connections as well in Proposition \ref{prop:dDTismin}. 
Thus dDT connections have properties that minimal connections have.

One of the challenging problems in $G_2$-geometry 
is to define an enumerative invariant of $G_2$-manifolds 
by counting associative submanifolds, 
which are considered to be analogous to pseudoholomorphic curves 
in symplectic manifolds, 
or 
$G_2$-instantons, which are considered as higher dimensional analogues of 
ASD connections on 4-manifolds. 
It is expected that 
dDT connections, associative submanifolds and $G_2$-instantons 
will behave similarly
because dDT connections are ``mirrors" of associative submanifolds,  
as stated above, and 
the formal ``large radius limit" of the defining equation of dDT connections 
is that of $G_2$-instantons (\cite{kawai2023notes}). 
We proved in \cites{kawai2020deformation, kawai2021mirror} 
that the moduli space of dDT connections indeed has properties similar to 
those of associative submanifolds and $G_2$-instantons. 
Hence we might hope to define an enumerative invariant 
by counting dDT connections. 
The dDT equation is more nonlinear than the $G_2$-instanton equation, 
but we consider on line bundles 
which we hope would make the analysis of the dDT equation tractable. 
Moreover, note that while there is a mirror symmetric interpretation 
of the dDT equation on line bundles, 
there is no such interpretation in the higher rank case. 
See also \cite[p.3]{kawai2021mirror}.

We will need to consider the compactification of the moduli space to define an enumerative invariant. 
To do so, we should know the behaviour of a sequence of dDT connections. 
Hence we will need to establish the compactness theorem. 
As stated above, dDT connections are minimal connections, 
just as $G_2$-instantons are Yang--Mills connections. 
We can show that minimal connections will behave similarly to 
Yang--Mills connections in the sense that 
the formal ``large radius limit" of the defining equation of minimal connections 
is that of Yang--Mills connections 
(Remark \ref{rem:large radius limit of min conn}). 
In addition, the compactness theorem is already established for 
Yang--Mills connections in 
\cites{uhlenbeck1982removable, price1983monotonicity, nakajima1988compactness, tian2000gauge}.  
Thus to consider the compactness theorem for dDT connections, 
it would be natural to consider the compactness theorem of minimal connections 
following the Yang--Mills case. 
In (higher dimensional) gauge theory, 
the compactness theorem for Yang--Mills connections was obtained from 
(i) Price's monotonicity formula \cite{price1983monotonicity} 
and 
(ii) $\varepsilon$-regularity theorem of Uhlenbeck \cite{uhlenbeck1982removable} 
and Nakajima \cite{nakajima1988compactness}. 
Using these, Tian \cite{tian2000gauge} 
studied the structure of the ``energy concentration set" $S$ in detail 
for ASD, HYM, $G_2, {\rm Spin}(7)$-instantons, 
and related $S$ to calibrated geometry. 
In this paper, we especially study monotonicity formulas for minimal connections.

\SkipTocEntry \subsection*{Main Results}

Minimal connections can be studied more generally in terms of conservation laws, 
so we formulate our results in that light. 
We show monotonicity formulas for some versions of volume functionals.

\begin{theorem} \label{mainthm 1}
Let $(X,g)$ be an oriented $n$-dimensional Riemannian manifold. 
Fix $p \in X$. 
Denote by ${\rm inj}_g (p)$ and $B_{\rho} (p)$
the injectivity radius of $(X,g)$ at $p$ and 
the open geodesic ball of radius $\rho$ centered at $p \in X$, respectively.  

For $k=1, \cdots, 4$, assume $(A_k)$ as described below. 
Then there exist $0<r_p< {\rm inj}_g (p)$ 
and a constant $a=a(n,p,g) \geq 0$ such that 
the following conclusion $(C_k)$ holds 
for any 2-form $\beta \in \Omega^2$ satisfying a conservation law. 

\begin{enumerate}
\item (Theorem \ref{thm:mono} and Corollary \ref{cor:mono sh})
$(A_1)$: No assumptions. 
$(C_1)$:
The function 
$$
(0,r_p] \to \rl, \qquad 
\rho \mapsto \frac{e^{a \rho^2}}{\rho} \int_{B_\rho (p)} 
\tr (G_\beta^{-1}) v_g(\beta)  \vol_g
$$
is non-decreasing, where 
$G_\beta = \id_{TX} - \beta^\sharp \circ \beta^\sharp 
\in \Gamma(X, {\rm End}(TX))$ 
as defined in \eqref{eq:Gn}.

\item (Theorem \ref{thm:mono vol} and Corollary \ref{cor:mono vol odd dim})
$(A_2)$: $X$ is odd dimensional. $(C_2)$: The function  
$$
(0,r_p] \to \rl, \qquad 
\rho \mapsto \frac{e^{a \rho^2}}{\rho} \int_{B_\rho (p)} v_g(\beta)  \vol_g
$$
is non-decreasing.

\item (Theorems \ref{thm:mono nvol 1} and \ref{thm:mono nvol 3}) 
$(A_3)$: $X$ is odd dimensional and one of the following conditions is satisfied. 
\begin{enumerate} 
    \item ${\rm Scal}_p >0$. 
    \item ${\rm Scal}_p = 0$ and $3|R_p|^2-8|{\rm Ric}_p|^2 - 18 (\Delta {\rm Scal})_p >0$, 
    where $\Delta {\rm Scal} = d^* d {\rm Scal}$.  
    \item ${\rm Ric} \geq 0$ on $X$.   
\end{enumerate}
Here, 
we denote by $R$, ${\rm Ric}$, ${\rm Scal}$ 
the curvature tensor, Ricci curvature, scalar curvature, respectively, 
and $T_p$ is the value of a tensor $T$ at $p$. 
$(C_3)$: The function 
$$
(0,r_p] \to \rl, \qquad 
\rho \mapsto \frac{e^{a \rho^2}}{\rho} \int_{B_\rho (p)} (v_g (\beta)-1) \vol_g 
+ 2a \Theta (\rho) 
$$
is non-decreasing for a function $\Theta: [0, \infty) \to [0, \infty)$ given 
by \eqref{eq:Theta monot nvol}.

\item (Theorem \ref{thm:monotnicity 2dim}) 
$(A_4)$: $(X, g)$ is a 2-dimensional real analytic Riemannian manifold 
and the 2-form $\beta$ is real analytic. 
$(C_4)$: The functions 
$$
(0,r_p] \to \rl, \qquad 
\rho \mapsto \frac{e^{a \rho^2}}{\rho} \int_{B_\rho (p)} 
v_g (\beta) \vol_g, 
\qquad
\mbox{and}
\qquad
\rho \mapsto \frac{e^{a \rho^2}}{\rho} \int_{B_\rho (p)} 
(v_g (\beta)-1) \vol_g, 
$$
are non-decreasing. 
\end{enumerate}
\end{theorem}

The more detailed statements are given in 
Theorems \ref{thm:mono}, \ref{thm:mono vol}, \ref{thm:mono nvol 1}, \ref{thm:mono nvol 3}, 
\ref{thm:monotnicity 2dim} 
and Corollaries \ref{cor:mono sh}, \ref{cor:mono vol odd dim}. 
Theorem \ref{mainthm 1} can be considered 
to be analogous to Price's monotonicity formula for Yang--Mills connections (\cite{price1983monotonicity}) 
and the monotonicity formula for minimal submanifolds (\cite[Propositions 1.12 and 1.15]{colding2011minimal}).
Theorem \ref{mainthm 1} allows estimates of these versions of volume functionals 
on balls to pass down to smaller balls. 
This is an important property because 
the $\varepsilon$-regularity theorem is proved by combining this property 
with the mean value inequality in the Yang--Mills case. 
We hope that Theorem \ref{mainthm 1} is also useful 
to prove the $\varepsilon$-regularity theorem for our case, 
but there are some problems as in Remark \ref{rem:future problem}.\\

The rough idea of the proof is as follows. 
For the proof, 
we use the Laplacian-like operator $\Delta_\beta$ 
for a 2-form $\beta$ defined in \eqref{eq:lapb}, 
which would be the most suitable operator for our study. 
We first show the ``integration by parts" formula for $\Delta_\beta$ 
when a 2-form $\beta$ satisfies a conservation law in Corollary \ref{cor:ibp}. 
This formula can be considered as the ``mirror" of 
the ``integration by parts" formula on submanifolds. 
See Remark \ref{rem:int by parts mirror interp} and Proposition \ref{prop:FM codiff}. 
Using this formula, we compute $\Delta_\beta \xi$ for a suitable cutoff function $\xi$. 
Then we can discuss as in the Yang--Mills case 
(\cite[Theorem 1]{price1983monotonicity} or \cite[Theorem 2.1.1]{tian2000gauge}) 
to compute the difference in 
(radius-normalized) ``modified volume" on geodesic balls of different radii 
and obtain Theorem \ref{mainthm 1} (1). 
The factor $\tr (G_\beta^{-1})$ appears because we use $\Delta_\beta$ 
whose symbol gives a quadratic form with trace $\tr (G_\beta^{-1})$. 
In this Yang--Mills case, the (standard) Hodge Laplacian is used 
and its symbol gives a quadratic form with constant trace, so 
factors like $\tr (G_\beta^{-1})$ do not appear.

Using this proof, we also compute the differences 
for the volume and the normalized volume in Theorems \ref{thm:mono vol} and \ref{thm:mono nvol 1}, respectively. 
When $X$ is odd dimensional, we can show 
an algebraic estimate in Lemma \ref{lem:ineq tr}. 
This together with Theorem \ref{thm:mono vol} implies 
the monotonicity for the volume (Theorem \ref{mainthm 1} (2)). 

The formula in Theorem \ref{thm:mono nvol 1} gives 
some conditions to obtain the monotonicity for the normalized volume 
(Corollary \ref{cor:mono nvol sh}): 
geometric conditions (the first two assumptions of Corollary \ref{cor:mono nvol sh})
and an algebraic condition 
(the third assumption of Corollary \ref{cor:mono nvol sh}). 
We show that 
the geometric conditions are satisfied 
under the curvature assumptions of (a), (b), (c) in Theorem \ref{mainthm 1} (3), 
which are used to control the volume of geodesic balls 
centered at $p$ (Lemma \ref{lem:cond assump}), 
and the algebraic condition is satisfied when $X$ is odd dimensional 
to obtain Theorem \ref{mainthm 1} (3).
The function $\Theta$ could be understood as an error term.

The two dimensional case is rather special. 
In this case, $V^0_g$ agrees with a simpler functional 
(Remark \ref{rem:2dim cons law}), 
and we can show the strong result that 
the 2-form $\beta$ is parallel under the assumption ($A_4$). Using this, we obtain Theorem \ref{mainthm 1} (4). \\

We can take $r_p= \infty$ and $a=0$ when $X=\rl^n$ with the standard flat metric $g_0$. 
Then we obtain the following vanishing theorem. 

\begin{corollary}[Corollary \ref{cor:vanish}] \label{maincor vanishing}
Let $\beta \in \Omega^2(\rl^{2m+1})$ be a 2-form 
satisfying a conservation law. 
Suppose that for some $p \in \rl^{2m+1}$, we have 
$$
\int_{B_r(p)} (v_{g_0} (\beta)-1) \vol_{g_0} = o(r) 
\qquad \mbox{as} \qquad r \to \infty. 
$$
Then $\beta=0$. 

In particular, if $\beta \in \Omega^2(\rl^{2m+1})$ is a 2-form 
satisfying a conservation law and 
$
\int_{\rl^{2m+1}} (v_{g_0} (\beta)-1) \vol_{g_0} < \infty,  
$
then $\beta=0$. 
\end{corollary}
We also have the vanishing theorem in the two dimensional case as noted in Remark \ref{rem:2dim cons law}. \\

Now we go back to setting of connections. 
Let $(X,g)$ be an oriented $n$-dimensional Riemannian manifold 
and $L \to X$ be a smooth complex line bundle with a Hermitian metric $h$. 
We first show the following. 

\begin{theorem}[Theorem \ref{thm:min imply cons}] \label{mainthm min imply cons}

The curvature 2-form of a minimal connection satisfies a conservation law. 
\end{theorem}

As stated at the beginning of the introduction, 
this can be understood as a consequence of 
Noether's theorem and the diffeomorphism invariance of $V^0_g (\n)$ 
(Remarks \ref{rem:variational SET} and \ref{rem:variational minimal}). 
Hence the same results as in Theorem \ref{mainthm 1} hold for minimal connections. 
In particular, the vanishing theorem as in Corollary \ref{maincor vanishing} holds 
for minimal connections (Corollary \ref{cor:vanish min}). 

In Proposition \ref{prop:fistvar}, we can characterize that 
a Hermitian connection $\n$ is minimal if and only if $\delta_\n \FF=0$, 
where $\delta_\n$ is the codifferential-like operator depending on $\n$ 
given in \eqref{eq:del nabla}. 
Using this, we see that 
the formal ``large radius limit" of the defining equation of minimal connections 
is that of Yang--Mills connections in Remark \ref{rem:large radius limit of min conn}. 
Thus it is natural to expect that minimal connections for a ``sufficiently large" metric  
will behave like Yang--Mills connections. 
As an application, we show the following existence theorem. 

\begin{theorem}[Theorem \ref{thm:exist minimal conn}] \label{mainthm existence}
Let $(X,g)$ be a compact oriented $n$-dimensional Riemannian manifold 
and $L \to X$ be a smooth complex line bundle with a Hermitian metric $h$. 
When the metric $g$ is ``sufficiently large", there exists a minimal connection. 
\end{theorem}

\vspace{1em}
Next, we consider deformed Donaldson--Thomas (dDT) connections. 
There is also the notion of dDT connections on a manifold with 
a ${\rm Spin}(7)$-structure. 
For distinction, 
we call dDT connections on a manifold with 
a $G_2$- and ${\rm Spin}(7)$-structure 
$G_2$- and ${\rm Spin}(7)$-dDT connections, respectively. 
See Definitions \ref{def:G2dDT}--\ref{def:dHYM} for the definition.

\begin{theorem}[Propositions \ref{prop:dDTismin}, \ref{prop:mono nvol G2dDT} 
and Corollary \ref{cor:mono vol G2dDT}] \label{mainthm dDT}

\begin{enumerate}
    \item 
$G_2$-dDT connections on a manifold with a closed $G_2$-structure $\varphi$ ($d \varphi=0$), 
${\rm Spin}(7)$-dDT connections on a (torsion-free) ${\rm Spin}(7)$-manifold 
and deformed Hermitian Yang--Mills (dHYM) connections 
on a K\"ahler manifold are minimal. 

\item 
Let $X$ be a (torsion-free) $G_2$-manifold 
and $L \to X$ be a smooth complex line bundle with a Hermitian metric $h$. 
Fix $p \in X$. Then there exist $0<r_p< {\rm inj}_g (p)$ 
and a constant $a=a(n,p,g) \geq 0$ such that 
for a $G_2$-dDT connection $\n$ of $(L,h)$, 
$$
(0,r_p] \to \rl, \qquad 
\rho \mapsto \frac{e^{a \rho^2}}{\rho^{5/2}} \int_{B_\rho (p)} v (\n)  \vol_g
$$
and 
$$
(0,r_p] \to \rl, \qquad 
\rho \mapsto \frac{e^{a \rho^2}}{\rho^{13/7}} \int_{B_\rho (p)} (v_g (\beta)-1) \vol_g 
+ 2a \Theta (\rho) 
$$
are non-decreasing for a function $\Theta: [0, \infty) \to [0, \infty)$. 
\end{enumerate}
\end{theorem}

(1) is the ``mirror" version of the fact that 
calibrated submanifolds in a $G_2$-, ${\rm Spin}(7)$-, or 
K\"ahler manifold are minimal submanifolds. 
It is already proved in \cite[Corollaries 4.5, 5.4, 6.5]{kawai2021mirror} 
when a $G_2$-, ${\rm Spin}(7)$-, or a 3,4-dimensional K\"ahler manifold is compact without boundary. 
We show that these connections are also minimal  
when a manifold is noncompact (and of any dimension in the K\"ahler case). 
For more details, see Section \ref{sec:ddt dhym min}.

(2) says that we obtain the stronger monotonicity formula in the $G_2$-dDT case 
than in Theorem \ref{mainthm 1} in the sense that the power of $\rho$ 
in the denominator can be taken larger. 
We obtain this using some algebraic equations that $G_2$-dDT connections must satisfy given 
in \cite{kawai2020deformation}. 
Hence we obtain the following stronger vanishing theorem in the $G_2$-dDT case 
than in Corollary \ref{maincor vanishing}. 

\begin{corollary}[Corollary \ref{cor:vanish G2dDT}] \label{maincor vanish G2dDT}
Let $(L, h) \to \rl^7$ be a (necessarily trivial) smooth complex Hermitian line bundle 
over $\rl^7$ with the standard $G_2$-structure inducing the standard flat metric $g_0$. 
If $\n$ is a $G_2$-dDT connection satisfying 
$$
\int_{B_r(p)} (v (\n)-1) \vol_{g_0} = o(r^{13/7}) 
\qquad \mbox{as} \qquad r \to \infty 
$$
for some $p \in \rl^7$, 
then $\n$ is flat. 

In particular, if $\n$ is a $G_2$-dDT connection and 
$
\int_{\rl^7} (v(\n)-1) \vol_{g_0} < \infty,  
$
then $\n$ is flat. 
\end{corollary}

\begin{remark} \label{rem:future problem}
In Theorems \ref{mainthm 1} and \ref{mainthm dDT}, 
we state that 
a function of the form 
$\frac{e^{a \rho^2}}{\rho^\kappa} \int_{B_\rho (p)} 
\widetilde v_g (\beta)  \vol_g
$
is non-decreasing for some $\kappa \geq 1$. 
We are not sure whether the $\kappa$ is optimal or not. 
That is, it might be possible that we can show the monotonicity for a larger $\kappa$ 
and such a stronger statement might be required for the bubbling analysis. 
We have no idea about these now. 
The reasons why it is difficult to determine which $\kappa$ is optimal 
include the following: 
(a) As stated below Theorem \ref{mainthm 1}, 
the symbol of $\Delta_\beta$ gives a quadratic form with trace $\tr (G_\beta^{-1})$, 
which is nonconstant in general. 
This causes a problem in the choice of $\kappa$ 
so that $\kappa$ cannot be taken very large. 
The symbol of the Hodge Laplacian used in the Yang--Mills case 
gives a quadratic form with constant trace, 
so $\kappa$ can be taken sufficiently large. 
(b) In the Yang--Mills case, $\kappa$ is taken so that the radius-normalized 
Yang--Mills functional is scaling invariant 
(cf. \cite[Remark 3.31]{fadel2019gauge} or \cite[Remark 3.25]{fadel2016blow}).
We cannot find such a canonical scaling invariance for $V_g$ or $V^0_g$. 
(c) In the minimal submanifold case, 
$\kappa$ is the dimension of a submanifold 
(cf. \cite[Propositions 1.12 and 1.15]{colding2011minimal}). 
As stated at the beginning of the introduction, 
$V$ corresponds to the volume functional for submanifolds 
via the real Fourier--Mukai transform. 
However, the information of the dimension of submanifolds is lost 
via the real Fourier--Mukai transform.

Next, we would like to prove the $\varepsilon$-regularity theorem for the compactness theorem, but there are some problems. 
In the Yang--Mills case, 
the mean value inequality for the Hodge Laplacian is used to prove the $\varepsilon$-regularity theorem. 
See \cite{tian2000gauge} or 
\cites{fadel2019gauge, fadel2016blow}. 
We will have to use $\Delta_\beta$ instead of the Hodge Laplacian. 
However, the eigenvalues of the quadratic form given by the symbol of $\Delta_\beta$ 
are not uniformly bounded from below, 
so we will not be able to use the mean value inequality for $\Delta_\beta$ 
as in \cite[Theorem 9.20]{gilbarg2001elliptic} 
to estimate the value of a point by the uniform constant multiple of an integral. 

There is a mean value inequality in the minimal submanifold case. 
The eigenvalues will not be uniformly bounded in this case, but the Lebesgue differentiation theorem 
is used in the proof. See for example \cite[Proposition 1.15 and Corollary 1.16]{colding2011minimal}. 
For the application of the Lebesgue differentiation theorem, 
we will have to show the monotonicity for $\kappa=n= \dim X$, which seems to be difficult to prove.  
Corollary \ref{cor:mono vol sh} gives a condition for this about the volume functional, 
but $\beta=0$ must be satisfied to obtain the monotonicity for $\kappa=n$. 
%Thus we will have to find another method 
We will need to solve these problems somehow 
to obtain the $\varepsilon$-regularity theorem. 
\end{remark}

\SkipTocEntry \subsection*{Organization of this paper}
This paper is organized as follows. 
In Section \ref{sec:def vol}, 
we give a precise definition of the volume functional for 2-forms  
and introduce a tensor $G_\beta$ depending on a 2-form $\beta$ that is 
frequently used in this paper. 

In Section \ref{sec:SET cons law}, 
we define the stress-energy tensor for a 2-form by the variation of $V^0_g$ 
with respect to the metric $g$, 
and we explicitly describe the conditions for a 2-form to satisfy a conservation law. 
Then we deduce some important formulas including the ``integration by parts" formula 
for $\Delta_\beta$ when a 2-form $\beta$ satisfies a conservation law. 

In Section \ref{sec:monot formula}, we prove the monotonicity formulas 
for some versions of volume functionals. 
First, we consider the ``modified volume" case, which corresponds 
to Theorem \ref{mainthm 1} (1). 
Then we consider the volume and normalized volume cases 
and give conditions to obtain the monotonicity. 
Then we show that some of the conditions are satisfied in the case of odd and two dimensions 
to obtain Theorem \ref{mainthm 1} (2)--(4) and Corollary \ref{maincor vanishing}. 

In Section \ref{sec:min conn}, we compute the first variation of 
the volume functional for Hermitian connections and give a simpler description than 
in \cite[Proposition 3.2]{kawai2021mirror}. 
Using this, we show Theorem \ref{mainthm min imply cons}. 
Then we prove the existence theorem (Theorem \ref{mainthm existence}) 
and Theorem \ref{mainthm dDT}, Corollary \ref{maincor vanish G2dDT} for dDT connections. 

In Appendix \ref{sec:FM}, 
we give the ``mirror" correspondence between geometric objects 
by the real Fourier--Mukai transform.

\vspace{1em}
\noindent{{\bf Acknowledgements}}: 
The author would like to thank Hikaru Yamamoto for a number of useful discussions 
and Daniel Fadel for helpful discussions on early versions of this paper. 
He also would like to thank anonymous referees for the careful reading of an earlier version of this paper and many useful comments which 
significantly improved the quality of the paper.

\subsection{Notation}\label{sec:cons not}
%%%%%%%%%%%%%%%%%%%%%%%%%%%%%%%%%%%%%%%%

We summarize the notation used in this paper. 

We omit the summation symbol $\Sigma$ if it does not cause confusion. 
Let $(X,g)$ be an oriented $n$-dimensional Riemannian manifold. 
We assume that $X$ has no boundary throughout this paper. 
Denote by $\Omega^k=\Omega^k(X)$ the space of $k$-forms on $X$. 
For $v \in TX$ and $\alpha \in T^*X$, 
define $v^\flat \in T^*X$ and $\alpha^\sharp \in TX$ by 
$$
v^\flat = g(v, \cdot), \qquad \alpha=g(\alpha^\sharp, \cdot), 
$$
respectively. 
For an endomorphism $K \in \Gamma(X, {\rm End}(TX))$, 
$K^{\flat} \in \Gamma(X, T^*X \otimes T^*X)$ is a 2-tensor
defined by 
\begin{align} \label{eq:Kflat}
K^{\flat}(u_1, u_2)=g(K(u_1), u_2)    
\end{align}
for $u_1, u_2 \in TX$. 
For a 2-tensor $T \in \Gamma(X, T^*X \otimes T^*X)$, 
define an endomorphism 
$T^\sharp \in \Gamma (X, {\rm End}(TX))$ by 
\begin{align} \label{eq:def sharp}
g (T^\sharp (u_1),u_2) = T (u_1,u_2)    
\end{align}
for $u_1, u_2 \in TX$. 
Then we have 
$(K^{\flat})^\sharp=K, (T^\sharp)^\flat=T$, and 
\begin{align}
K^\flat (u, \cdot)=(K(u))^\flat, \qquad
T^\sharp (u)= (T(u,\cdot))^\sharp,    
\end{align}
for any $u \in TX$, endomorphism $K$ and 2-tensor $T$.  
Define the metric for 2-tensors by 
\begin{align} \label{eq:ipdef}
\la \alpha_1 \otimes \alpha_2, \alpha_3 \otimes \alpha_4 \ra
= 
\frac{1}{2} \la \alpha_1, \alpha_3 \ra \la \alpha_2, \alpha_4 \ra, 
\end{align}
where $\alpha_j \in \Omega^1$. 
This definition is compatible with the standard metric 
for 2-forms. That is, we have 
$
\la \alpha_1 \wedge \alpha_2, \alpha_3 \wedge \alpha_4 \ra
= 
\la \alpha_1, \alpha_3 \ra \la \alpha_2, \alpha_4 \ra
- \la \alpha_1, \alpha_4 \ra \la \alpha_2, \alpha_3 \ra, 
$
where $\alpha_1 \wedge \alpha_2 
= \alpha_1 \otimes \alpha_2 - \alpha_2 \otimes \alpha_1$. 
For 2-forms $\beta_1, \beta_2$ 
and symmetric 2-tensors $h_1, h_2$, we have 
\begin{align} \label{eq:tr met}
{\rm tr} \left(\beta_1^{\sharp} \circ \beta_2^\sharp\right)
=-2\left\la \beta_1, \beta_2 \right\ra, \qquad 
{\rm tr} \left(h_1^{\sharp} \circ h_2^\sharp\right)
=2\left\la h_1, h_2 \right\ra, \qquad
{\rm tr} \left(h^{\sharp} \right)
=2\left\la h, g \right\ra. 
\end{align}

We denote by $D$ the Levi-Civita connection with respect to $g$.  
Note that for any vector field $u$, 
2-form $\beta$ and symmetric 2-tensor $h$, 
$D_u \bb=(D_u \beta)^\sharp$ is skew-symmetric and $D_u h$ is symmetric. 

For a local orthonormal frame $\{ e_i \}$ with its dual $\{ e^i=e_i^\flat \}$, set 
\begin{align} \label{eq:def deri}
D_i T:= D_{e_i} T
\end{align}
for any tensor $T$. 
%%%%%%%%%%%%%%%%%%%%%%%%%%%%%%%%%%%%%%%%

\section{The (normalized) volume functional} \label{sec:def vol}
Let $(X,g)$ be an oriented $n$-dimensional Riemannian manifold. 
Define $V=V_g :\Omega^2 \to [0, \infty]$ by 
\begin{align} \label{eq:def vol func}
V_g (\beta)= \int_X v_g(\beta) \vol_g,\qquad \mbox{where} \quad 
v_g (\beta)= \sqrt{\det (\id_{TX} + \beta^{\sharp_g})}. 
\end{align}
Here, 
$\vol_g$ is the volume form defined by the Riemannian metric $g$ 
and 
$\beta^\sharp = \beta^{\sharp_g} \in \Gamma (X, {\rm End} TX)$ 
is defined as in \eqref{eq:def sharp}.  
Note that $v_g (\beta) \geq 1$ since $\beta^{\sharp_g}$ is skew-symmetric. 
See also the proof of Lemma \ref{lem:ineq tr}. 
We call $V$ the {\bf volume functional}. 
It is because 
$V$ is considered as the ``mirror'' of the standard volume for submanifolds 
via the real Fourier--Mukai transform 
when $\beta$ is the ($-\i$ times) curvature 
of a Hermitian connection of a Hermitian line bundle. 
See the proof of \cite[Lemma 4.3]{kawai2021FM}.

If $X$ is noncompact of infinite volume, 
we always have $V_g(\beta)=\infty$ as $v_g (\beta) \geq 1$ for any $\beta \in \Omega^2$. 
To obtain some information from the volume functional, 
we define the {\bf normalized volume functional} 
$V^0=V^0_g :\Omega^2 \to [0, \infty]$ by 
\begin{align} \label{eq:def nor vol func}
V^0_g (\beta)= \int_X \left( v_g(\beta)-1 \right) \vol_g 
\end{align}
so that it can take finite values (for some $\beta$). 
A similar definition is found in \cite[Definition 3.1]{Wei2022}. 
The functional $V^0_g$ has a simple and useful property as shown in \eqref{eq:norvol 00}.

\begin{remark} \label{rem:det Gb}
It is useful to introduce $G_\beta \in \Gamma (X, {\rm End} TX)$ for $\beta \in \Omega^2$ 
defined by 
\begin{align} \label{eq:Gn}
\begin{split}
G_\beta:=&\id_{TX} - \beta^\sharp \circ \beta^\sharp \\
=& \left(\id_{TX} - \bb \right) \circ \left(\id_{TX} + \bb \right)
= {}^t\! \left(\id_{TX} + \bb \right) \circ \left(\id_{TX} + \bb \right), 
\end{split}
\end{align}
where we denote by ${}^t T$ the transpose of $T \in \Gamma (X, {\rm End} TX)$ 
with respect to $g$. 
We see that $G_\beta$ is symmetric and positive definite. 
Since 
$\bb \circ G_\beta = G_\beta \circ \bb$, 
we have 
\begin{align} \label{eq:bG comm}
\bb \circ G_\beta^{-1} = G_\beta^{-1}\circ \bb, 
\end{align}
which we use frequently. 
We see that $G_\beta^{-1}\circ \bb$ is skew-symmetric by \eqref{eq:bG comm}. 
This identity also implies that 
\begin{align} \label{eq:bGb}
\bb \circ G_\beta^{-1} \circ \bb
= - \id_{TX} + G_\beta^{-1}. 
\end{align}
We see that $v_g (\beta)$ is rewritten as 
\[
v_g (\beta) = \left(\det G_\beta \right)^{1/4}. 
\]
\end{remark}

Since 
$G_\beta=\id_{TX} + {}^t \beta^\sharp \circ \beta^\sharp$, 
eigenvalues of $G_\beta$ are greater than or equal to 1. 
So eigenvalues of $G_\beta^{-1}$ are less than or equal to 1, 
and we have the following. 
\begin{lemma} \label{lem:Gn-1 eq}
For any unit vector $v \in TX$, we have 
$$
g(G_\beta^{-1}(v),v) \leq 1 \qquad \mbox{and} \qquad
\tr \left(G_\beta^{-1} \right) \leq n. 
$$
\end{lemma}

\begin{remark} \label{rem:norvol 00}
By the proof of \cite[Lemma A.2]{kawai2021mirror}, we see that 
$$
v_g (\beta)
=
\left| e^\beta \right|
=
\sqrt{\sum_{k=0}^{[n/2]} \left| \frac{\beta^k}{k!} \right|^2}
= 
\sqrt{1+ \left| \beta \right|^2 
+ 
\left| \frac{\beta^2}{2!} \right|^2
+ 
\left| \frac{\beta^3}{3!} \right|^2
+ 
\cdots
}. 
$$
There are terms containing $\beta^2$ or $\beta^3$, 
so the (normalized) volume functional is not studied in the framework of 
$F$-energy functionals as in \cite[Section 2]{Dong2011vanishing} when $\dim X \geq 4$. 
When $\dim X=2$, 
$V_g (\beta)$ is considered as an $F$-energy functional 
for a certain function $F$ and there are some special properties 
as we see Section \ref{sec:2dim}. 
By this description of $v_g (\beta)$, we easily see that 
\begin{align}\label{eq:norvol 00}
V^0_g (\beta)=0 
\qquad \Longleftrightarrow \qquad 
\beta=0.
\end{align}

\end{remark}

\section{The stress-energy tensor and a conservation law} \label{sec:SET cons law}

In this section, 
we compute the variation of $V^0_g (\beta)$ 
with respect to $g$ for a fixed $\beta \in \Omega^2$, 
and define the stress-energy tensor $S_{g,\beta}$ as its $L^2$ gradient. 
Then we say that 
a 2-form satisfies a conservation law 
if the stress-energy tensor is divergence free. 
These notions are defined in the same way 
in the cases of ($F$-)harmonic maps and ($F$-)Yang--Mills connections 
\cites{Baird1980, Dong2011vanishing, Wei2022}. 
The reason why we study these is due to Noether's theorem 
and the diffeomorphism invariance of $V^0_g (\beta)$. 
That is, when we consider the variation of $V^0_g (\beta)$ with respect to $\beta$, 
the critical points satisfy a conservation law, 
which is explained in Remark \ref{rem:variational SET}.

\begin{proposition} \label{prop:fistvar set}
Suppose that $V^0_g(\beta)< \infty$. 
Fix a compactly supported symmetric 2-tensor $h$ and set $g_t=g+t h$ for $|t| < \eps$. 
Then we have 
\[
\left. \frac{d}{dt} V^0_{g_t} (\beta) \right|_{t=0} 
= \la h, S_{g, \beta} \ra_{L^2}. 
\]
Here, $\la \,\cdot\,, \,\cdot\, \ra_{L^2}$ 
is the $L^2$ inner product with respect to the metric $g$, 
and the symmetric 2-tensor $S_{g, \beta} \in \Gamma(X, S^2 T^*X)$ is given by 
\begin{align} \label{eq:set}
S_{g, \beta}
= -g + v_g (\beta) \cdot (G_\beta^{-1})^\flat. 
\end{align}
\end{proposition}

\begin{definition}
We call $S_{g, \beta}$ the {\bf stress-energy tensor} associated with $V^0_g$. 
\end{definition}
Note that $S_{g, \beta}$ is also defined when $V^0_g (\beta)=\infty$. 

\begin{proof}
Since $V^0_g(\beta)< \infty$, we have 
$$
\left. \frac{d}{dt} V^0_{g_t} (\beta) \right|_{t=0}
= 
\int_X 
\left. \frac{d}{d t} (v_{g_t} (\beta)-1) \vol_{g_t}
\right|_{t=0} 
$$
by  Lebesgue's dominated convergence theorem, 
where $\vol_{g_t}$ is the volume form defined by $g_t$. 
So we first compute the following. 
\begin{lemma} \label{lem:der volb}
\begin{align}
\left. \frac{d}{d t} \vol_{g_t} \right|_{t=0} 
= \la h, g \ra \vol_g, \qquad 
\left. \frac{d}{d t} \bbt \right|_{t=0} 
= -h^\sharp \circ \bb,  
\end{align}
where $\la \cdot, \cdot \ra$ is the induced metric from $g$ in \eqref{eq:ipdef} 
and $\bbt$ is defined as in \eqref{eq:def sharp} for $g_t$. 
\end{lemma}

\begin{proof}
Let $\{ e_i(t) \}_{i=1}^n$ be a local orthonormal frame 
with respect to $g_t$. Denote by $\{ e^i(t) \}_{i=1}^n$ its dual. 
Since 
$\vol_{g_t} = e^1(t) \wedge \cdots \wedge e^n(t)$, we compute 
\begin{align}
\left. \frac{d}{d t} \vol_{g_t} \right|_{t=0} 
= 
\left. \left \la \frac{d}{d t} e^i(t), e^i \right \ra \right|_{t=0}\vol_g
=
\frac{1}{2} 
\left. \frac{d}{d t} \left \la e^i(t), e^i(t) \right \ra \right|_{t=0} \vol_g. 
\end{align}
By \eqref{eq:ipdef}, we have 
$$
2 \la g_t, g \ra 
= 
2 \la e^i(t) \otimes e^i(t), e^j \otimes e^j \ra
= 
\la e^i(t), e^j \ra \la e^i(t), e^j \ra 
= \left \la e^i(t), e^i(t) \right \ra, 
$$
which implies the first equation. 
Differentiating 
$g_t( \bbt (\cdot), \cdot)= \beta$, we have 
$$
h (\bb(\cdot), \cdot) 
+ g \left( \left. \frac{d}{d t} \bbt \right|_{t=0} (\cdot), \cdot \right)=0.  
$$
Since 
$h (\bb(\cdot), \cdot) 
= g ((h^\sharp \circ \bb) (\cdot), \cdot)$, 
we obtain the second equation. 
\end{proof}

For simplicity, we set 
$\beta^\sharp = \beta^{\sharp_g}$ and 
$G_t = \id_{TX} - \bbt \circ \bbt$. 
Then $G_0=G_\beta$ and we compute 
\begin{align}
\left. \frac{d}{d t} (v_{g_t} (\beta)-1) \right|_{t=0}
= 
\left. \frac{d}{d t} (\det G_t)^{1/4} \right|_{t=0} 
= 
\frac{(\det G_\beta)^{1/4}}{4} \cdot 
{\rm tr} 
\left( \left. \frac{d}{d t} G_t \right|_{t=0} \circ G_\beta^{-1}
\right). 
\end{align}
By Lemma \ref{lem:der volb}, we have 
\begin{align}
\left. \frac{d}{d t} G_t \right|_{t=0} 
= h^\sharp \circ (\bb)^2 + \bb \circ h^\sharp \circ \bb. 
\end{align}
Then by \eqref{eq:bG comm}, $(\bb)^2=-G_\beta + \id_{TX}$ and \eqref{eq:tr met}, 
it follows that  
\begin{align} \label{eq:diff vgt}
\begin{split}
\left. \frac{d}{d t} (v_{g_t} (\beta)-1) \right|_{t=0}
=&
\frac{v_g(\beta)}{2} \cdot 
{\rm tr} 
\left( h^\sharp \circ (\bb)^2 \circ G_\beta^{-1}
\right)\\
=&
\frac{v_g(\beta)}{2} \cdot 
{\rm tr} \left( - h^\sharp + h^\sharp \circ G_\beta^{-1} \right) 
= 
v_g(\beta) \left \la h, -g + (G_\beta^{-1})^\flat \right \ra. 
\end{split}
\end{align}
Then by Lemma \ref{lem:der volb} and \eqref{eq:diff vgt}, we obtain 
\begin{align}
\left. \frac{d}{d t} V^0_{g_t} (\beta) \right|_{t=0} 
=& 
\int_X 
\left. \frac{d}{d t} (v_{g_t} (\beta)-1) \right|_{t=0} \vol_g 
+ 
(v_g (\beta)-1) \left. \frac{d}{d t} \vol_{g_t} \right|_{t=0} \\
=& 
\int_X 
v_g (\beta) \left \la h, -g + (G_\beta^{-1})^\flat \right \ra \vol_g 
+ 
(v_g (\beta)-1) \la h, g \ra \vol_g \\
=& 
\int_X \left \la h, -g + v_g(\beta) \cdot (G_\beta^{-1})^\flat \right \ra \vol_g  
\end{align}
and the proof is completed. 
\end{proof}

Next, we compute the divergence of $S_{g, \beta}$. 
Recall that the divergence ${\rm div} S$ of a symmetric 2-tensor $S$ 
is defined by 
$$
{\rm div} S = (D_i S) (e_i, \cdot) \in \Omega^1, 
$$
where $\{ e_i \}$ is a local orthonormal frame and 
we use the notation \eqref{eq:def deri}. 
The divergence for higher order symmetric tensors is also defined. 
See \cite[Section 1.59]{besse2008einstein}. 
\begin{definition} \label{def:conservation law}
A 2-form $\beta \in \Omega^2$ is said to satisfy a {\bf conservation law} 
(with respect to $V^0_g$)
if ${\rm div} S_{g, \beta} =0$. 
\end{definition}

\begin{remark} \label{rem:variational SET}
We explain why it is important to study 2-forms satisfying a conservation law 
following \cite[Section 3.3]{hawking1973}. 
See also \cite[Section 1]{Baird1980}. 

It is because when we consider the variation of $V^0_g (\beta)$ 
with respect to $\beta$, 
the critical points satisfy a conservation law. 
Indeed, take any 2-form $\beta \in \Omega^2$ with $V^0_g(\beta)< \infty$ and diffeomorphism $\eta$ of $X$. 
Define $\sharp_{\eta^* g}$ as in \eqref{eq:def sharp} for the metric $\eta^* g$. 
Then 
$$
g \left((\beta^\sharp \circ \eta_*) (\cdot), \eta_* (\cdot) \right) 
= 
\beta (\eta_* (\cdot), \eta_* (\cdot)) 
= 
\eta^* \beta
= 
(\eta^* g) \left((\eta^* \beta)^{\sharp_{\eta^* g}} (\cdot), \cdot \right)
= 
g \left(\left((\eta^* \beta)^{\sharp_{\eta^* g}} \circ \eta_* \right) (\cdot), \eta_* (\cdot) \right). 
$$
Thus we have  
$$
(\eta^* \beta)^{\sharp_{\eta^* g}} = \beta^\sharp,  
$$
which implies that 
$v_g (\beta) = v_{\eta^* g}(\eta^* \beta)$. Hence 
\begin{align} \label{eq:nor vol diff inv}
V^0_g (\beta) = V^0_{\eta^* g} (\eta^* \beta).     
\end{align}
Now take any compactly supported vector field $u$ and denote by $\{ \eta_t \}$ the flow of $u$. 
By \eqref{eq:nor vol diff inv}, it follows that 
\begin{align}
0= \left. \frac{d}{dt} V^0_{\eta_t^* g} (\eta_t^* \beta) \right|_{t=0}
= 
\left. \frac{d}{dt} V^0_{\eta_t^* g} (\beta) \right|_{t=0}
+ 
\left. \frac{d}{dt} V^0_g (\eta_t^* \beta) \right|_{t=0}. 
\end{align}
Since $\left. \frac{d}{dt}\eta_t^* g \right|_{t=0} = L_u g$, 
where $L_u$ is the Lie derivative along $u$, the first term is given by 
$
\la L_u g, S_{g, \beta} \ra_{L^2}
$
by Proposition \ref{prop:fistvar set}. 
When $\beta$ is a critical point for 
the variation of $V^0_g (\beta)$ with respect to $\beta$, 
the second term vanishes. 
Hence 
$$
0
= \left. \frac{d}{dt} V^0_{\eta_t^* g} (\eta_t^* \beta) \right|_{t=0} 
= \la L_u g, S_{g, \beta} \ra_{L^2}, 
$$
which is equivalent to ${\rm div} S_{g, \beta}=0$ by \cite[Lemma 1.60]{besse2008einstein}. 
\end{remark}

For the computation of ${\rm div} S_{g, \beta}$, 
we define a differential operator $\delta_\beta: \Omega^k \to \Omega^{k-1}$ 
for a 2-form $\beta \in \Omega^2$ by 
\begin{align} \label{eq:delb}
\delta_\beta = -i(G_\beta^{-1} (e_i)) D_i, 
\end{align}
where $i(\cdot)$ is the interior product, 
$\{ e_i \}$ is a local orthonormal frame and 
we use the notation \eqref{eq:def deri}. 
Note that $\delta_\beta$ agrees with the standard codifferential $d^*$ 
if $\beta=0$. 
We can define a second order operator $\Delta_\beta:\Omega^k \to \Omega^k$ by
\begin{align} \label{eq:lapb}
\Delta_\beta = d \delta_\beta + \delta_\beta d. 
\end{align}
We easily see that $\Delta_\beta$ is elliptic. \\

\begin{proposition} \label{prop:div set}
For a 2-form $\beta \in \Omega^2$, we have 
\begin{align}
{\rm div} S_{g, \beta} 
= 
v_g (\beta)
\left \{
\left \la i(e_k) d\beta, \left( G_\beta^{-1} \circ \bb \right)^\flat 
\right \ra 
G_\beta^{-1}(e_k) 
- 
(G_\beta^{-1} \circ \bb) 
\left((\delta_\beta \beta)^\sharp \right) 
\right \}^\flat \in \Omega^1, 
\end{align}
where $\{ e_i \}$ is a local orthonormal frame with its dual $\{ e^i \}$. 
\end{proposition}

In particular, 
$\beta$ satisfies a conservation law 
if $\beta$ is closed and $\delta_\beta \beta=0$. 
Note that Definition \ref{def:conservation law} makes sense 
also when $V^0_g (\beta)=\infty$. 
We first prove the following.

\begin{lemma} \label{lem:cl diff}
For any $u \in TX$, we have 
\begin{align}
u \left(v_g (\beta) \right) &= v_g (\beta) 
\left \la \left(G_\beta^{-1} \circ \bb \right)^\flat, D_u \beta \right \ra, \\
\left( (D_i \bb) \circ G_\beta^{-1} \circ \bb \right)(e_i)
&=
\left \la \left(G_\beta^{-1} \circ \bb \right)^\flat, 
i(e_k) d \beta - D_k \beta \right \ra e_k. 
\end{align}
\end{lemma}

\begin{proof}
By $v_g (\beta)=\left(\det G_\beta \right)^{1/4}$, 
\eqref{eq:Gn}, \eqref{eq:bG comm} and \eqref{eq:tr met}, we compute 
\begin{align}
u \left(v_g(\beta) \right)
=& 
\frac{v_g(\beta)}{4} \cdot {\rm tr} \left(G_\beta^{-1} \circ D_u G_\beta \right) \\
=&
- \frac{v_g(\beta)}{2} \cdot {\rm tr} \left(G_\beta^{-1} \circ \bb \circ D_u \bb \right)
=
v_g(\beta) \left \la \left(G_\beta^{-1} \circ \bb \right)^\flat, D_u \beta \right \ra, 
\end{align}
which implies the first equation. 

Next, we prove the second equation. 
Since $D_i \beta = (D_i \beta) (e_j, e_k) e^j \wedge e^k/2$, we have 
\begin{align} \label{eq:Dbb}
D_i \bb = (D_i \beta) (e_j, e_k) e^j \otimes e_k. 
\end{align}
Then it follows that 
\begin{align}
\left((D_i \bb) \circ G_\beta^{-1} \circ \bb \right)(e_i) 
=&
(D_i \beta) (e_j, e_k) \cdot 
g \left(e_j, \left(G_\beta^{-1} \circ \bb \right)(e_i) \right) \cdot e_k\\
=&
\frac{1}{2} 
\left \{
(D_i \beta) (e_j, e_k) - (D_j \beta) (e_i, e_k) \right \}\cdot 
g \left(e_j, \left(G_\beta^{-1} \circ \bb \right)(e_i) \right) \cdot e_k, 
\end{align}
where we use the fact that $G_\beta^{-1} \circ \bb$ is skew-symmetric. 
Since 
$d \beta = e^\ell \wedge D_\ell \beta$, we have 
\begin{align}
d \beta (e_i, e_j, e_k) = 
(D_i \beta) (e_j, e_k) + (D_j \beta) (e_k, e_i) + (D_k \beta) (e_i, e_j). 
\end{align}
Then we obtain
\begin{align}
\left((D_i \bb) \circ G_\beta^{-1} \circ \bb \right)(e_i)
=&
\frac{1}{2} \left \{ d \beta (e_i, e_j, e_k) - (D_k \beta) (e_i, e_j) \right \} \cdot 
g \left(e_j, \left(G_\beta^{-1} \circ \bb \right)(e_i) \right) \cdot e_k\\
=&
\frac{1}{2} \left \{ i(e_k) d \beta -(D_k \beta) \right \} (e_i, e_j) \cdot 
\left(G_\beta^{-1} \circ \bb \right)^\flat (e_i, e_j) \cdot e_k\\
=& 
\left \la \left(G_\beta^{-1} \circ \bb \right)^\flat, i(e_k) d \beta - D_k \beta \right \ra e_k, 
\end{align}
and the proof of Lemma \ref{lem:cl diff} is completed. 
\end{proof}

\begin{proof}[Proof of Proposition \ref{prop:div set}]
By the definition of the stress-energy tensor $S_{g, \beta}$ in \eqref{eq:set} 
and Lemma \ref{lem:cl diff}, we have 
\begin{align} \label{eq:divS 1}
\begin{split}
{\rm div} S_{g, \beta}
=& 
e_i \left( v_g(\beta) \right) \cdot (G_\beta^{-1})^\flat (e_i, \cdot) 
+ v_g(\beta) \cdot \left(D_i (G_\beta^{-1})^\flat\right) (e_i, \cdot) \\
=& 
v_g (\beta) \left \{ \left \la \left(G_\beta^{-1} \circ \bb \right)^\flat, D_i \beta \right \ra
\cdot G_\beta^{-1} (e_i)
+ (D_i G_\beta^{-1}) (e_i) \right \}^\flat. 
\end{split}
\end{align}

Since 
$
0=D_i \left(G_\beta \circ G_\beta^{-1} \right) 
= \left(D_i G_\beta \right) \circ G_\beta^{-1} + G_\beta \circ \left(D_i G_\beta^{-1} \right), 
$
we have 
\begin{align} \label{eq:diff ginv}
\begin{split}
D_i G_\beta^{-1} 
=& G_\beta^{-1} \circ \left( \left(D_i \bb \right) \circ \bb
+ \bb \circ \left(D_i \bb \right) \right) \circ G_\beta^{-1} \\
=& 
G_\beta^{-1} \circ \left( \left(D_i \bb \right) \circ G_\beta^{-1} \circ \bb \right) 
+ \left( G_\beta^{-1} \circ \bb \right ) \circ \left(D_i \bb \right)  \circ G_\beta^{-1}. 
\end{split}
\end{align}
By the definition of $\delta_\beta$ in \eqref{eq:delb}, 
we have 
\begin{align}
\left( \left(D_i \bb \right) \circ G_\beta^{-1} \right) (e_i) 
= 
\left( i(G_\beta^{-1} (e_i)) D_i \beta \right)^\sharp  
= - (\delta_\beta \beta)^\sharp. 
\end{align}
Then by Lemma \ref{lem:cl diff}, it follows that 
\begin{align} \label{eq:divS 2}
(D_i G_\beta^{-1}) (e_i)
= 
\left \la \left(G_\beta^{-1}\circ \bb \right)^\flat, 
i(e_k) d \beta - D_k \beta \right \ra G_\beta^{-1} (e_k) 
- (G_\beta^{-1} \circ \bb) \left((\delta_\beta \beta)^\sharp \right). 
\end{align}
Then by \eqref{eq:divS 1} and \eqref{eq:divS 2}, the proof is completed. 
\end{proof}

The following useful property holds if $\beta$ satisfies a conservation law. 

\begin{proposition} \label{prop:div formula}
Suppose that $\beta$ satisfies a conservation law. 
Then for any compactly supported 1-form $\alpha$, we have 
$$
\int_X (\delta_\beta \alpha) \cdot v_g(\beta) \cdot \vol_g=0. 
$$
\end{proposition}
\begin{proof}
For any symmetric 2-tensor $S \in \Gamma(X,S^2 T^*X)$, we have 
\begin{align}
-d^* (S(\alpha^\sharp, \cdot)) 
= i(e_i) D_i (S(\alpha^\sharp, \cdot))
= (D_i S) (\alpha^\sharp, e_i) + S (D_i \alpha^\sharp, e_i)
= ({\rm div} S) (\alpha^\sharp) + S (D_i \alpha^\sharp, e_i). 
\end{align}
If $S=S_{g, \beta}$, we compute 
\begin{align}
S_{g, \beta} (D_i \alpha^\sharp, e_i) 
=& (-g + v_g (\beta) \cdot (G_\beta^{-1})^\flat) (D_i \alpha^\sharp, e_i) \\
=& d^* \alpha + v_g (\beta) \cdot g(G_\beta^{-1}(D_i \alpha^\sharp), e_i)
= d^* \alpha - v_g (\beta) \cdot \delta_\beta \alpha. 
\end{align}
Hence we obtain 
\begin{align} \label{eq:div general}
-d^* (S_{g, \beta} (\alpha^\sharp, \cdot)) 
= 
({\rm div} S_{g, \beta}) (\alpha^\sharp) 
+ d^* \alpha - v_g (\beta) \cdot \delta_\beta \alpha, 
\end{align}
which implies the statement. 
\end{proof}

Proposition \ref{prop:div formula} implies the following 
``integration by parts" formula. 

\begin{corollary} \label{cor:ibp}
Suppose that $\beta$ satisfies a conservation law. 
Then for any functions $f_1, f_2$, one of which is compactly supported, 
we have 
$$
\int_X (\Delta_\beta f_1) \cdot f_2 \cdot v_g (\beta) \vol_g 
= 
\int_X \langle df_1, (G_\beta^{-1})^* df_2 \rangle v_g (\beta) \vol_g 
= 
\int_X f_1 \cdot (\Delta_\beta f_2) \cdot v_g (\beta) \vol_g. 
$$
\end{corollary}

\begin{proof} 
Set $\alpha=f_1 df_2$ in Proposition \ref{prop:div formula}. 
Then 
\begin{align} \label{eq:ibp proof}
\begin{split}
\delta_\beta (f_1 d f_2)
=& -i(G_\beta^{-1}(e_i)) D_i (f_1 d f_2) \\
=& -i(G_\beta^{-1}(e_i)) \left(e_i(f_1) d f_2 + f_1 D_i d f_2 \right)
= - \la df_1, (G_\beta^{-1})^* df_2 \ra + f_1 \Delta_\beta f_2. 
\end{split}
\end{align}
Since $\la df_1, (G_\beta^{-1})^* df_2 \ra$ is symmetric 
with respect to $f_1$ and $f_2$, the proof is completed. 
\end{proof}

This ``integration by parts" formula holds for functions, 
but will not hold for any differential forms. 

More generally, we can show the following. 
We do not use this result in this paper, but this formula might be useful 
for a future study.

\begin{lemma} \label{lem:ivp gen}
For a 2-form $\beta \in \Omega^2$, define $\delta'_\beta:\Omega^k \to \Omega^{k-1}$ by 
$$
\delta'_\beta=\delta_\beta
- i \left( v_g (\beta)^{-1} ({\rm div} S_{g, \beta})^\sharp \right) 
$$
and define a second order elliptic operator $\Delta'_\beta:\Omega^k \to \Omega^k$ by
$\Delta'_\beta = d \delta'_\beta + \delta'_\beta d$. Then 

\begin{enumerate}
\item 
for any compactly supported 1-form $\alpha$, we have 
$$
\int_X (\delta'_\beta \alpha) \cdot v_g(\beta) \cdot \vol_g=0. 
$$

\item
For any functions $f_1, f_2$, one of which is compactly supported, we have 
$$
\int_X (\Delta'_\beta f_1) \cdot f_2 \cdot v_g (\beta) \vol_g 
= 
\int_X \langle df_1, (G_\beta^{-1})^* df_2 \rangle v_g (\beta) \vol_g 
= 
\int_X f_1 \cdot (\Delta'_\beta f_2) \cdot v_g (\beta) \vol_g. 
$$
\end{enumerate}
\end{lemma}

\begin{proof}
Note that \eqref{eq:div general} holds for any $\beta \in \Omega^2$ and 
compactly supported 1-form $\alpha$. Then (1) is immediate from \eqref{eq:div general}. 
For (2), set $\alpha=f_1 df_2$ in (1). 
Using \eqref{eq:ibp proof}, we compute 
\begin{align}
\delta'_\beta (f_1 d f_2)
=& - \la df_1, (G_\beta^{-1})^* df_2 \ra + f_1 \delta_\beta d f_2 
- i \left( v_g (\beta)^{-1} ({\rm div} S_{g, \beta})^\sharp \right) (f_1 df_2) \\
=& 
- \la df_1, (G_\beta^{-1})^* df_2 \ra + f_1 \delta'_\beta d f_2,  
\end{align}
which implies (2). 
\end{proof}

The map $\delta'_\beta:\Omega^k \to \Omega^{k-1}$ for $k=1$ 
has a mirror symmetric interpretation. 
See the paragraph below Proposition \ref{prop:FM codiff}. \\

It is natural to expect the Weitzenb\"ock-type formula for $\Delta_\beta$. We obtain the following formula. 
Note that there exists a term containing the first order derivative 
unlike the standard Weitzenb\"ock formula.

\begin{proposition} \label{prop:weit}
For any $k$-form $\alpha \in \Omega^k$, we have 
\begin{align}
\Delta_\beta \alpha = 
\delta_\beta D \alpha 
-e^i \wedge i \left( (D_i G_\beta^{-1}) (e_j) \right) D_j \alpha
- e^i \wedge i \left(G_\beta^{-1} (e_j) \right) R(e_i, e_j) \alpha,  
\end{align}
where 
$R$ is the curvature tensor of $g$, and 
$\delta_\beta D$ is the composition of 
the covariant derivative $D=D^{\Lambda^k}: \Omega^0(\Lambda^k) \to \Omega^1(\Lambda^k)$ 
and 
$\delta_\beta = -i(G_\beta^{-1} (e_i)) D_i = -i(G_\beta^{-1} (e_i)) D_i^{\Lambda^k}
:\Omega^1(\Lambda^k) \to \Omega^0(\Lambda^k)$. 
\end{proposition}

\begin{proof}
We follow the proof of the standard Weitzenb\"ock formula as in 
\cite[Theorem 5.3.1]{konno2013}. 
We compute 
\begin{align}
d \delta_\beta \alpha
=&
e^i \wedge D_i \left( -i \left( G_\beta^{-1} (e_j) \right) D_j \alpha \right)\\
=&
e^i \wedge \left( -i \left( (D_i G_\beta^{-1}) (e_j) \right) D_j \alpha
-i \left( G_\beta^{-1} (D_i e_j) \right) D_j \alpha
-i \left( G_\beta^{-1} (e_j) \right) D_i D_j \alpha
\right). 
\end{align}
Since 
$D_i e_j = g(D_i e_j, e_k) e_k = - g(e_j, D_i e_k) e_k$, we have 
\begin{align} \label{eq:weit 1}
d \delta_\beta \alpha
=
-e^i \wedge i \left( G_\beta^{-1} (e_j) \right) \left(D_i D_j \alpha - D_{D_i e_j} \alpha \right) 
- e^i \wedge i \left( (D_i G_\beta^{-1}) (e_j) \right) D_j \alpha. 
\end{align}
We also compute 
\begin{align} \label{eq:weit 2}
\begin{split}
\delta_\beta d \alpha
=&
-i \left( G_\beta^{-1} (e_i) \right) D_i \left(e^j \wedge D_j \alpha \right)\\
=&
-i \left( G_\beta^{-1} (e_i) \right) \left((D_i e^j) \wedge D_j \alpha + e^j \wedge D_i D_j \alpha \right)\\
=&
-i \left( G_\beta^{-1} (e_i) \right) \left(e^j \wedge 
\left( D_i D_j \alpha - D_{D_i e_j} \alpha \right) \right), 
\end{split}
\end{align}
where we use $D_i e^j = g(D_i e_j, e_k) e^k = - g(e_j, D_i e_k) e^k$. 
By the same computation as in \eqref{eq:weit 2}, we see that 
\begin{align} \label{eq:weit 3}
\delta_\beta D \alpha 
= 
-g \left( G_\beta^{-1} (e_i), e_j \right) \left( D_i D_j \alpha - D_{D_i e_j} \alpha \right). 
\end{align}
Hence, by \eqref{eq:weit 1}, \eqref{eq:weit 2} and \eqref{eq:weit 3}, we obtain 
\begin{align} \label{eq:weit 4}
\Delta_\beta \alpha - \delta_\beta D \alpha
=
- e^i \wedge i \left( (D_i G_\beta^{-1}) (e_j) \right) D_j \alpha
- \Theta \left( D_i D_j \alpha - D_{D_i e_j} \alpha \right), 
\end{align}
where 
\begin{align}
\Theta 
=& (e^i \wedge) \circ i \left( G_\beta^{-1} (e_j) \right) 
+ i \left( G_\beta^{-1} (e_i) \right) \circ (e^j \wedge)
- g \left( G_\beta^{-1} (e_i), e_j \right)\\
=& 
(e^i \wedge) \circ i \left( G_\beta^{-1} (e_j) \right) - (e^j \wedge) \circ i \left( G_\beta^{-1} (e_i) \right). 
\end{align}
Then 
\begin{align}
\Theta \left( D_i D_j \alpha - D_{D_i e_j} \alpha \right)
=&
e^i \wedge i \left( G_\beta^{-1} (e_j) \right) 
\left \{
\left( D_i D_j \alpha - D_{D_i e_j} \alpha \right)
- \left( D_j D_i \alpha - D_{D_j e_i} \alpha \right)
\right \} \\
=&
e^i \wedge i \left( G_\beta^{-1} (e_j) \right) R(e_i, e_j) \alpha. 
\end{align}
This together with \eqref{eq:weit 4} gives the desired formula. 
\end{proof}

%%%%%%%%%%%%%%%%%%%%%%%%%%%%%%%%%%%%%%%%%%%%%%%%
\section{The monotonicity formulas} \label{sec:monot formula}

In this section, we prove the monotonicity formulas 
for some versions of volume functionals. 
These formulas can be considered 
to be analogous to Price's monotonicity formula for Yang--Mills connections (\cite{price1983monotonicity}) 
and the monotonicity formula for minimal submanifolds (\cite[Propositions 1.12 and 1.15]{colding2011minimal}).

Section \ref{sec:mono notation} is a summary of the notation 
in Section \ref{sec:monot formula}. 
In Section \ref{sec:modified vol}, 
we compute the difference in 
(radius-normalized) ``modified volume" on geodesic balls of different radii 
in Theorem \ref{thm:mono}, 
where 
the ``integration by parts" formula in Corollary \ref{cor:ibp} plays a role. 
Then we deduce the monotonicity formula (Corollary \ref{cor:mono sh}). 
In Section \ref{sec:vol}, using this proof, we also compute 
the difference in (radius-normalized) volume in Theorem \ref{thm:mono vol}
and consider that under what conditions we can obtain the monotonicity 
in Corollary \ref{cor:mono vol sh}. 
In Section \ref{sec:normalized vol}, similar computations and considerations 
are given for (radius-normalized) normalized volume in Theorem \ref{thm:mono nvol 1} 
and Corollary \ref{cor:mono nvol sh}. 
In Section \ref{sec:odd dim}, we provide examples where these conditions 
for monotonicity are satisfied. 
That is, when a manifold is odd dimensional and satisfies certain curvature conditions, 
we obtain monotonicity formulas 
(Corollary \ref{cor:mono vol odd dim} and Theorem \ref{thm:mono nvol 3}). 
Section \ref{sec:2dim} deals with the case when a manifold is 2-dimensional. 
This case is rather special and we can obtain the monotonicity formulas 
(Theorem \ref{thm:monotnicity 2dim}) directly from Corollary \ref{cor:mono sh}.

\subsection{Notation} \label{sec:mono notation}
We first introduce the notation. 
Let $(X,g)$ be an oriented $n$-dimensional Riemannian manifold 
with the Levi-Civita connection $D$. 
Fix $p \in X$. Denote by ${\rm inj}_g (p)$ the injectivity radius of $(X,g)$ at $p$. 
Take $0 <r_p < {\rm inj}_g (p)$ satisfying the following. 
\begin{itemize}
\item 
There are normal coordinates $(x^1, \cdots, x^n)$ 
in the open geodesic ball $B_{r_p} (p)$ of radius $r_p$ centered at $p \in X$. 
\item
There is $c(p) \geq 0$ such that 
\begin{align} \label{eq:p est}
|g_{i j} - \delta_{i j}| \leq c(p) r^2, 
\qquad
\left| D \frac{\partial}{\partial x^i} \right| \leq c(p) r, 
\end{align}
for any $1 \leq i,j \leq n$, where 
$g_{i j} = g(\partial/\partial x^i, \partial/\partial x^j)$ and $r=\sqrt{\sum_{i=1}^n (x^i)^2}$. 
\end{itemize}

Note that since $g_{i j} (p)=\delta_{i j}$ 
and $\partial g_{i j}/ \partial x^k (p) =0$, 
the Taylor expansions of $g_{i j}$ show that the constants $r_p$ and $c(p)$ can be 
chosen depending only on ${\rm inj}_g (p)$ and the curvature of $g$. 

If $g$ is flat, we can take any $0 <r_p < {\rm inj}_g (p)$ and $c(p) =0$. 
If $X$ is compact, 
then it follows from \cite[p. 16, Theorem 1.3]{hebey2000nonlinear} that we can choose uniform constants
$0 < \delta_0 = \delta_0 (X, g) < {\rm inj}_g (X):= \inf_{p \in X} {\rm inj}_g (p)$ 
and 
$c_0=c_0(X,g) \geq 0$ such that the above condition holds with 
$r_p=\delta_0$ and $c(p)=c_0$ for any point of $p \in X$.

In what follows, we will denote by $O (1)$ a number or a function 
bounded by a constant depending only on $n = \dim X$. 
We use the following function, which is useful for the proof in this section.

For a 2-form $\beta \in \Omega^2$ and a unit vector $v \in TX$, define 
\begin{align} \label{eq:Xi}
\Xi (\beta, v) = \tr (G_\beta^{-1}) - g(G_\beta^{-1}(v), v),  
\end{align}
where $G_\beta$ is defined in \eqref{eq:Gn}. 

\begin{lemma} \label{lem:Xi nonneg}
We have $\Xi (\beta, v) \geq 0$ for any 2-form $\beta \in \Omega^2$ and a unit vector $v \in TX$. 
\end{lemma}
\begin{proof}
Suppose that $v \in T_p X$ for $p \in X$. 
Take an orthonormal frame $\{e_i \}_{i=1}^n$ of $T_p X$ such that $e_1=v$. Then 
$
\Xi (\beta, v) = \sum_{i=2}^n g(G_\beta^{-1}(e_i), e_i). 
$
Since $G_\beta^{-1}$ is positive definite, we see that $\Xi (\beta, v) \geq 0$. 
\end{proof}

\subsection{The modified volume case} \label{sec:modified vol}

We first give the monotonicity formula for the 
(radius-normalized) ``modified volume",  
where the ``modified volume" is defined by 
$$
\Omega^2 \to [0, \infty], \qquad
\beta \mapsto 
\int_X  \tr (G_\beta^{-1}) v_g (\beta) \vol_g. 
$$
In the following, we can prove the formula in a slightly more general form.

\begin{theorem} \label{thm:mono}
Let $(X,g)$ be an oriented $n$-dimensional Riemannian manifold. 
Fix $p \in X$ and use the notation in Section \ref{sec:mono notation}. 
Then there exists
a constant $a=a(n,p,g) \geq 0$ such that for 
a 2-form $\beta \in \Omega^2$ satisfying a conservation law, 
a nonnegative smooth function $f$ on $X$, 
$\kappa \in \rl$ and $0 < \sigma < \rho \leq r_p$, 
we have 
\begin{align}
&\frac{e^{a \rho^2}}{\rho^\kappa} \int_{B_\rho (p)} f \tr (G_\beta^{-1}) v_g (\beta) \vol_g 
-
\frac{e^{a \sigma^2}}{\sigma^\kappa} \int_{B_\sigma (p)} f \tr (G_\beta^{-1}) 
v_g (\beta) \vol_g \\
\geq & 
\int_{B_\rho (p) \backslash B_\sigma (p)} \frac{e^{a r^2}}{r^\kappa} f 
\Xi \left(\beta, \frac{\partial}{\partial r} \right) v_g (\beta) \vol_g \\
&+ 
\int_\sigma^\rho \frac{e^{a \tau^2}}{\tau^{\kappa +1}} 
\left(
\int_{B_\tau (p)} \left( (1-\kappa) f \tr (G_\beta^{-1}) + \frac{\tau^2 - r^2}{2} (-\Delta_\beta f) \right) v_g (\beta) \vol_g 
\right) 
d \tau. 
\end{align}
(The last term is the integration by $\tau \in [\sigma, \rho]$.)

Furthermore, the following holds. 
\begin{itemize}
\item If $(X,g)=(\rl^n, g_0)$, where $g_0$ is the standard flat metric, 
we can take
$a=0$ and the above inequality holds for any $0< \sigma < \rho < \infty$. 
\item
If $X$ is compact, we can choose uniform constants $a \geq 0$ and $\delta_0 >0$ such that the
above holds for any $0 < \sigma < \rho \leq r_p = \delta_0$. 
\end{itemize}
\end{theorem}

The following is immediate from Theorem \ref{thm:mono}. 
Note that $(1-\kappa) f \tr (G_\beta^{-1})=0$ when $\kappa=1$ 
and $\tau^2 - r^2 \geq 0$ on $B_\tau (p)$. 
Recall also Lemma \ref{lem:Xi nonneg}. 

\begin{corollary} \label{cor:mono sh}
In addition to assumptions in Theorem \ref{thm:mono}, 
suppose further that a function $f$ satisfies $-\Delta_\beta f \geq 0$. 
Then 
$$
(0,r_p] \to \rl, \qquad 
\rho \mapsto \frac{e^{a \rho^2}}{\rho} \int_{B_\rho (p)} 
f \tr (G_\beta^{-1}) v_g (\beta)  \vol_g
$$
is non-decreasing. 
In particular, setting $f=1$, we see that 
$$
(0,r_p] \to \rl, \qquad 
\rho \mapsto \frac{e^{a \rho^2}}{\rho} \int_{B_\rho (p)} 
\tr (G_\beta^{-1}) v_g(\beta)  \vol_g
$$
is non-decreasing. 
\end{corollary}

\begin{proof}
Without loss of generality we can suppose $\rho < r_p$. 
The case $\rho = r_p$ follows by the obvious approximation argument. 

Since $(x^1, \cdots, x^n)$ are normal coordinates and $r=\sqrt{\sum_{i=1}^n (x^i)^2}$, 
we have
\begin{align} \label{eq:rad isom}
g \left(\frac{\partial}{\partial r}, \cdot \right) = dr
\end{align}
by the Gauss lemma (cf. \cite[p. 133, Lemma 12]{petersen2006riem}). 
In particular, we have $g(\partial/\partial r, \partial/\partial r)=1$. 
Let $\{ e_i \}_{i=1}^n$ be an orthonormal local frame of $TX$ around $p$ 
such that $e_1=\partial/\partial r$. 
Since the unit radial vector field $\partial/\partial r$ 
is the velocity of a (radial) geodesic, it follows that
\begin{align} \label{eq:rad para}
D_{\frac{\partial}{\partial r}} \frac{\partial}{\partial r} =0.     
\end{align}
Recall the notation \eqref{eq:def deri}. 
For $\mu, \nu \geq 2$, set 
$$
D_\mu \left( r \frac{\partial}{\partial r} \right)
= \sum_{\nu=2}^n b_{\mu \nu} e_\nu 
\qquad 
\Longleftrightarrow
\qquad 
b_{\mu \nu} = g \left(D_\mu \left( r \frac{\partial}{\partial r} \right), e_\nu \right). 
$$

\begin{lemma} \label{lem:est b}
For any $(w_2, \cdots, w_n) \in \rl^{n-1}$, we have 
$$
\sum_{\mu, \nu=2}^n b_{\mu \nu} w^\mu w^\nu = \left(1+O(1)c(p)r^2 \right) \sum_{\mu=2}^n (w^\mu)^2. 
$$
\end{lemma}
\begin{proof}
Since $r \partial/ \partial r = \sum_{i=1}^n x^i \partial/ \partial x^i$, we have 
$$
b_{\mu \nu} 
= g \left( e_\mu + \sum_{i=1}^n x^i D_\mu \frac{\partial}{\partial x^i}, e_\nu \right)
= \delta_{\mu \nu} + \sum_{i=1}^n x^i g \left( D_\mu \frac{\partial}{\partial x^i}, e_\nu \right). 
$$
Setting $w=\sum_{\mu=2}^n w_\mu e_\mu$, we have 
$$
\sum_{\mu, \nu=2}^n b_{\mu \nu} w^\mu w^\nu = |w|^2 
+ \sum_{i=1}^n x^i g \left( D_w \frac{\partial}{\partial x^i}, w \right). 
$$
Since 
$$
|x^i| \leq r, \qquad 
g \left( D_w \frac{\partial}{\partial x^i}, w \right) 
\leq \left| D_w \frac{\partial}{\partial x^i} \right| |w| 
\leq \sum_{\mu=2}^n |w_\mu| \left| D_\mu \frac{\partial}{\partial x^i} \right| |w| 
\leq (n-1) c(p) r |w|^2, 
$$
where we use \eqref{eq:p est}, we obtain 
$$
\sum_{\mu, \nu=2}^n b_{\mu \nu} w^\mu w^\nu 
\leq \left(1 + n (n-1) c(p) r^2 \right) |w|^2
$$
and the proof is completed. 
\end{proof}

Next, we use the ``integration by parts" formula in Corollary \ref{cor:ibp}. 
Let $\xi: \rl \to \rl$ be a smooth function with support in $(-\infty, r_p)$. 
Then $\xi \circ r$ is a function $B_{r_p} (p) \to \rl$ which is smooth on 
$B_{r_p} (p)- \{ 0 \}$. 
Suppose that it extends smoothly on $B_{r_p} (p)$. 
Further, we consider $\xi = \xi \circ r$ a function on $X$ by the zero extension. 
Then by Corollary  \ref{cor:ibp}, we have 
\begin{align} \label{eq:mon ibp}
\int_X \xi \cdot (\Delta_\beta f) \cdot v_g (\beta) \vol_g 
= 
\int_X f \cdot (\Delta_\beta \xi) \cdot v_g (\beta) \vol_g 
\end{align}
for any nonnegative smooth function $f$. 
We first compute $\Delta_\beta \xi$. 

\begin{lemma} \label{lem:lap xi}
We have 
\begin{align}
\Delta_\beta \xi 
= 
- \xi'' g \left( G_\beta^{-1} \left(\frac{\partial}{\partial r} \right), 
\frac{\partial}{\partial r} \right)
- \frac{\xi'}{r} 
\left( 1+  O(1) c(p) r^2 \right) \Xi \left(\beta, \frac{\partial}{\partial r} \right), 
\end{align}
where $\xi'= d \xi/dr$ and $\xi''=d^2 \xi/dr^2$. 
\end{lemma}

\begin{proof}
Since $d \xi = \xi' dr$, we have 
$$
\Delta_\beta \xi
=
\delta_\beta (\xi' dr) 
= 
- \sum_{i=1}^n i(G_\beta^{-1}(e_i)) D_i (\xi' dr). 
$$
By \eqref{eq:rad isom} and \eqref{eq:rad para}, we see that 
$D_{\frac{\partial}{\partial r}} dr =0$. 
Thus we have 
$$
\Delta_\beta \xi
=
- \xi'' g \left( G_\beta^{-1} \left(\frac{\partial}{\partial r} \right), 
\frac{\partial}{\partial r} \right) 
- \xi' \sum_{\mu=2}^n (D_\mu dr) (G_\beta^{-1} (e_\mu)). 
$$
For $\mu \geq 2$, we have 
$$
D_\mu dr 
= \frac{1}{r} g \left( D_\mu \left(r \frac{\partial}{\partial r} \right), \cdot \right)
= \frac{1}{r} \sum_{\nu=2}^n b_{\mu \nu} g(e_\nu, \cdot) 
$$
Then by Lemma \ref{lem:est b}, we compute 
\begin{align}
\sum_{\mu, \nu=2}^n b_{\mu \nu} g(e_\nu, G_\beta^{-1} (e_\mu))
&= 
\sum_{\mu, \nu=2}^n b_{\mu \nu} g \left( (\id_{TX} + \bb)^{-1} e_\mu, (\id_{TX} + \bb)^{-1} (e_\nu) \right) \\
&=
\sum_{i=1}^n 
\sum_{\mu, \nu=2}^n 
b_{\mu \nu} g \left( (\id_{TX} + \bb)^{-1} e_\mu, e_i \right) g \left((\id_{TX} + \bb)^{-1} (e_\nu), e_i \right) \\
&=
\left(1+O(1)c(p)r^2 \right) \sum_{i=1}^n \sum_{\mu=2}^n g \left( (\id_{TX} + \bb)^{-1} e_\mu, e_i \right)^2 \\
&= 
\left(1+O(1)c(p)r^2 \right) \sum_{\mu=2}^n g(G_\beta^{-1} (e_\mu), e_\mu)\\
&=
\left(1+O(1)c(p)r^2 \right) \Xi \left(\beta, \frac{\partial}{\partial r} \right) 
\end{align}
and the proof of Lemma \ref{lem:lap xi} is completed. 
\end{proof}

Now we choose $\xi$ more specifically. 
Fix $0<\eps<1$ such that 
$(1+\eps) \rho < r_p$. (Recall that $\rho < r_p$.) 
Take a smooth function $\phi_\eps \in C^\infty([0, \infty))$ satisfying 
$\phi_\eps (t)=1$ for $t \in [0,1]$, $\phi_\eps (t)=0$ for $t \geq 1+ \eps$. 
We may also assume that 
\begin{align} \label{eq:cut off der}
\sup_{t \in (0, \infty)}|\phi_\eps'(t)| \leq \frac{2}{\eps}     
\end{align}
as stated in the proof of \cite[Theorem 8.8]{gilbarg2001elliptic}. 
Introduce a new parameter $\tau \in [\sigma, \rho]$ and 
set 
\begin{align}
\phi_{\eps, \tau} (r) 
&= \phi_\eps \left(\frac{r}{\tau} \right)\\
\xi (r)
&= \xi_{\eps,\tau} (r) 
= \int_r^\infty s \phi_{\eps, \tau} (s) ds.     
\end{align}
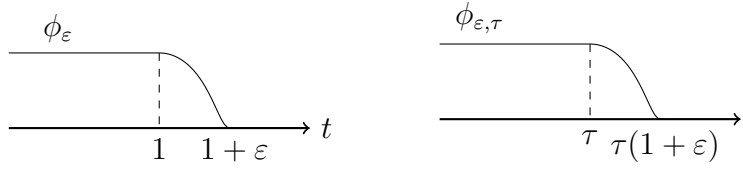
\begin{center}
\begin{tikzpicture}[scale=1]%倍率
\draw[thick, ->] (-2,0)--(2,0) node[right] {$t$};
\draw (-1,1) node[above left]{$\phi_\eps$};
\draw (1,0) node[below]{$1+\eps$};
\draw[dashed](0,1)--(0,0)node[below]{$1$};
\draw[domain=-2:0] plot(\x, 1); 
\draw[domain=0:0.99] plot(\x, {exp (1-1/(1-\x*\x)});
\draw[domain=1:2] plot(\x, 0); %node[below] {$y=-x^2$}
\end{tikzpicture}
\hspace{1cm}
\begin{tikzpicture}[scale=1]%倍率
\draw[thick, ->] (-2,0)--(2,0) node[right] {};
\draw (-1,1) node[above left]{$\phi_{\eps, \tau}$};
\draw (1,0) node[below]{$\tau (1+\eps)$};
\draw[dashed](0,1)--(0,0)node[below]{$\tau$};
\draw[domain=-2:0] plot(\x, 1); 
\draw[domain=0:0.99] plot(\x, {exp (1-1/(1-\x*\x)});
\draw[domain=1:2] plot(\x, 0); %node[below] {$y=-x^2$}
\end{tikzpicture}
\end{center}
This choice of $\xi$ is the same as in 
\cite[Section 16.1]{gilbarg2001elliptic} 
so that 
$- d \xi=  \phi_{\eps, \tau} (r) r dr$, 
which is the metric dual of the cut-off radial vector field. 
Since $\phi_{\eps, \tau} (s)=1$ when $s$ is sufficiently small, 
we see that 
$\xi (r) = C-r^2/2 = C -\sum_{i=1}^n (x^i)^2/2$ for some constant $C$
when $r$ is sufficiently small. 
Thus $\xi= \xi \circ r$ extends smoothly over $0$.  
Then we have 
$$
\xi_{\eps, \tau}' = -r \phi_{\eps, \tau}, \qquad 
\xi_{\eps, \tau}'' = - \phi_{\eps, \tau} - \frac{r}{\tau} \phi_\eps' \left( \frac{r}{\tau} \right)
= 
- \phi_{\eps, \tau} + \tau \frac{\partial}{\partial \tau} \phi_{\eps, \tau}. 
$$
By Lemma \ref{lem:lap xi}, it follows that 
\[
\begin{aligned}
\label{eq:mon lapxi}
\Delta_\beta \xi_{\eps, \tau} 
=& 
\left(\phi_{\eps, \tau} - \tau \frac{\partial}{\partial \tau} \phi_{\eps, \tau} \right) 
g \left( G_\beta^{-1} \left(\frac{\partial}{\partial r} \right), 
\frac{\partial}{\partial r} \right)
+ 
\phi_{\eps, \tau} \left( 1+  O(1) c(p) r^2 \right) \Xi \left(\beta, \frac{\partial}{\partial r} \right) \\
=& 
\phi_{\eps, \tau} \tr (G_\beta^{-1}) 
- \left( \tau \frac{\partial}{\partial \tau} \phi_{\eps, \tau} \right)
g \left( G_\beta^{-1} \left(\frac{\partial}{\partial r} \right), \frac{\partial}{\partial r} \right)
+ O(1) c(p) r^2 \phi_{\eps, \tau} \Xi \left(\beta, \frac{\partial}{\partial r} \right). 
\end{aligned}
\]

We also note the following. 
\begin{lemma}
\begin{align}
\int_X 
O(1) c(p) f r^2 \phi_{\eps, \tau} \Xi \left(\beta, \frac{\partial}{\partial r} \right) v_g(\beta) \vol_g
=
O(1) c(p) \tau^2 \int_X f \phi_{\eps, \tau} \Xi \left(\beta, \frac{\partial}{\partial r} \right) v_g(\beta) \vol_g, 
\label{eq:int O(1)} 
\\
\tau \frac{\partial}{\partial \tau} \int_X f \phi_{\eps, \tau} 
g \left( G_\beta^{-1} \left(\frac{\partial}{\partial r} \right), \frac{\partial}{\partial r} \right) v_g(\beta) \vol_g 
= 
\int_X f \left( \tau \frac{\partial}{\partial \tau} \phi_{\eps, \tau} \right) 
g \left( G_\beta^{-1} \left(\frac{\partial}{\partial r} \right), \frac{\partial}{\partial r} \right) v_g(\beta) \vol_g. 
\label{eq:change diff int}
\end{align}
\end{lemma}

Recall that $O (1)$ is a number or a function 
bounded by a constant depending only on $n = \dim X$.
The two $O(1)$s in \eqref{eq:int O(1)} are different functions 
or numbers, but they both have this property, 
so we use the same symbols by an abuse of notation.

\begin{proof}
Since $\phi_{\eps, \tau} (r) \neq 0$ precisely when $r \leq (1+\eps) \tau (< 2\tau)$, 
we have $r^2 \phi_{\eps, \tau} \leq 4 \tau^2 \phi_{\eps, \tau}$. 
Then we see \eqref{eq:int O(1)}. 

Also, by the fact that $\phi_{\eps, \tau}' (r) \neq 0$ precisely when 
$r \leq (1+\eps) \tau < (1+\eps) \rho < r_p$ 
and \eqref{eq:cut off der}, we see that 
$$
\left| \tau \frac{\partial}{\partial \tau} \phi_{\eps, \tau} \right|
=
\left|  r \phi_\eps' \left( \frac{r}{\tau} \right) \right|
\leq \frac{2 r_p}{\eps}. 
$$
Thus 
$\tau \frac{\partial}{\partial \tau} \phi_{\eps, \tau}$ is uniformly bounded with respect to $\tau$ 
and has support in $B_{(1+\eps) \rho} (p)$. 
By Lemma \ref{lem:Gn-1 eq}, we see that 
$g \left(G_\beta^{-1} \left(\frac{\partial}{\partial r} \right), \frac{\partial}{\partial r} \right) \leq 1$. 
Then we obtain \eqref{eq:change diff int} by Lebesgue's dominated convergence theorem. 
\end{proof}

Now set 
\begin{align} \label{eq:def IJQ}
\begin{split}
I_{\eps, \tau} &=\int_X f \phi_{\eps, \tau} \tr \left(G_\beta^{-1} \right) v_g(\beta) \vol_g, 
\qquad
J_{\eps, \tau}=\int_X f \phi_{\eps, \tau} 
\Xi \left(\beta, \frac{\partial}{\partial r} \right) v_g(\beta) \vol_g, \\
Q_{\eps, \tau} &=
\int_X \xi_{\eps, \tau} (\Delta_\beta f)  v_g(\beta) \vol_g. 
\end{split}
\end{align}
Substituting \eqref{eq:mon lapxi} into \eqref{eq:mon ibp} 
and using \eqref{eq:int O(1)}, \eqref{eq:change diff int}, 
$
g \left(G_\beta^{-1} \left(\frac{\partial}{\partial r} \right), \frac{\partial}{\partial r} \right)
=
\tr \left(G_\beta^{-1} \right) 
- \Xi \left(\beta, \frac{\partial}{\partial r} \right)$, 
we obtain 
\begin{align} \label{eq:int A}
\tau \frac{\partial}{\partial \tau}  
\left( I_{\eps, \tau} - J_{\eps, \tau} \right) 
= I_{\eps, \tau} 
+ O(1) c(p) \tau^2 J_{\eps, \tau}
- 
Q_{\eps, \tau}. 
\end{align}
Then \eqref{eq:int A} implies that  
\begin{equation} \label{eq:diff eI}
\begin{split}
\frac{\partial}{\partial \tau} 
\left( \frac{e^{a \tau^2}}{\tau^\kappa} I_{\eps, \tau} \right) 
=&
\frac{e^{a \tau^2}}{\tau^{\kappa -1}} 
\left( 2a I_{\eps, \tau} 
-\frac{\kappa}{\tau^2} I_{\eps, \tau} 
+ \frac{1}{\tau} \frac{\partial}{\partial \tau} 
I_{\eps, \tau} \right) \\
=&
\frac{e^{a \tau^2}}{\tau^{\kappa -1}} 
\left \{ 2a I_{\eps, \tau} 
-\frac{\kappa}{\tau^2} I_{\eps, \tau} 
+ \frac{1}{\tau^2} 
\left( 
\tau \frac{\partial}{\partial \tau} J_{\eps, \tau} 
+ I_{\eps, \tau} 
+ O(1) c(p) \tau^2 J_{\eps, \tau} 
-  
Q_{\eps, \tau} 
\right)
\right \}
\\
=&
\frac{e^{a \tau^2}}{\tau^{\kappa -1}} 
\left \{ 2a I_{\eps, \tau} +O(1)c(p) J_{\eps, \tau} 
+ \frac{1-\kappa}{\tau^2} I_{\eps, \tau}
+ \frac{1}{\tau} \frac{\partial}{\partial \tau} J_{\eps, \tau} 
- \frac{1}{\tau^2} Q_{\eps, \tau} 
\right \}. 
\end{split}
\end{equation}
Since $\tr \left(G_\beta^{-1} \right) \geq \Xi \left(\beta, \frac{\partial}{\partial r} \right)$, 
we have $I_{\eps, \tau} \geq J_{\eps, \tau}$.  
Choosing a non-negative number $a$ such that $2a \geq - O(1) c(p)$, we obtain 
\begin{align} \label{eq:mon ineq}
\frac{\partial}{\partial \tau} 
\left( \frac{e^{a \tau^2}}{\tau^\kappa} I_{\eps, \tau} \right)
\geq
\frac{e^{a \tau^2}}{\tau^{\kappa -1}} 
\left \{ \frac{1-\kappa}{\tau^2} I_{\eps, \tau}
+ \frac{1}{\tau} \frac{\partial}{\partial \tau} J_{\eps, \tau} 
- \frac{1}{\tau^2} Q_{\eps, \tau} 
\right \}.  
\end{align}
Then the result follows by integrating over $[\sigma, \rho]$ and 
letting $\eps$ tend to zero. 
Indeed, since 
$$
\lim_{\eps \to 0} \phi_{\eps, \tau} (r) 
= 
\lim_{\eps \to 0} \phi_{\eps} \left( \frac{r}{\tau} \right)
= 
\begin{cases}
1 & (r \leq \tau)\\
0 & (r > \tau) 
\end{cases}, 
\qquad 
\lim_{\eps \to 0} \xi_{\eps, \tau} (r) 
= 
\begin{cases}
\int_r^\tau s ds = \frac{\tau^2-r^2}{2} & (r \leq \tau)\\
0 & (r > \tau) 
\end{cases},  
$$
we have 
$$
\lim_{\eps \to 0} I_{\eps, \tau} = I_{\tau}, \qquad
\lim_{\eps \to 0} J_{\eps, \tau} = J_{\tau}, \qquad
\lim_{\eps \to 0} Q_{\eps, \tau} = Q_{\tau}, 
$$
where 
\begin{align}
I_{\tau}&=\int_{B_\tau (p)} f \tr \left(G_\beta^{-1} \right) v_g(\beta) \vol_g, 
\qquad
J_{\tau}=\int_{B_\tau (p)} f 
\Xi \left(\beta, \frac{\partial}{\partial r} \right) v_g(\beta) \vol_g, \\
Q_\tau &= 
\int_{B_\tau (p)} \frac{\tau^2-r^2}{2} (\Delta_\beta f)  v_g(\beta) \vol_g. 
\end{align}
Hence 
\begin{align}
\lim_{\eps \to 0} \int _{\sigma}^{\rho} 
\frac{\partial}{\partial \tau} 
\left( \frac{e^{a \tau^2}}{\tau^\kappa} I_{\eps, \tau} \right) d \tau 
&= 
\frac{e^{a \rho^2}}{\rho^\kappa} I_{\rho}
- 
\frac{e^{a \sigma^2}}{\sigma^\kappa} I_\sigma,  
\\
\lim_{\eps \to 0} \int _{\sigma}^{\rho} 
\frac{e^{a \tau^2}}{\tau^{\kappa +1}} 
\left \{ (1-\kappa) I_{\eps, \tau} 
- Q_{\eps, \tau} 
\right \} d \tau 
&=
\int_\sigma^\rho \frac{e^{a \tau^2}}{\tau^{\kappa +1}} 
\left \{
(1-\kappa) I_\tau - Q_\tau 
\right \}
d \tau.  
\end{align}
We also have 
\begin{align}
\int_{\sigma}^{\rho} 
\left( \frac{e^{a \tau^2}}{\tau^\kappa} \frac{\partial}{\partial \tau} J_{\eps, \tau} \right)
d \tau
&=
\left [ \frac{e^{a \tau^2}}{\tau^\kappa} J_{\eps, \tau} 
\right]_{\sigma}^{\rho}
- 
\int _{\sigma}^{\rho} 
\frac{\partial}{\partial \tau} \left(\frac{e^{a \tau^2}}{\tau^\kappa} \right) J_{\eps, \tau} 
d \tau \\
&\xrightarrow[\varepsilon \rightarrow 0]{}
\left [ \frac{e^{a \tau^2}}{\tau^\kappa} J_{\tau} 
\right]_{\sigma}^{\rho}
- 
\int _{\sigma}^{\rho} 
\frac{\partial}{\partial \tau} \left(\frac{e^{a \tau^2}}{\tau^\kappa} \right) J_{\tau} 
d \tau
= 
\int_{\sigma}^{\rho} 
\left( \frac{e^{a \tau^2}}{\tau^\kappa} \frac{\partial}{\partial \tau} J_{\tau} \right)
d \tau. 
\end{align}
By the coarea formula (cf.  \cite[Exercise III.12 (c) or (d)]{chavel2006}) 
for $r:B_\tau (p) - \{ 0 \} \to (0,\tau)$, 
we have 
$$
J_\tau 
= 
\int_0^\tau 
\left( \int_{\partial B_\upsilon (p)} f 
\Xi \left(\beta, \frac{\partial}{\partial r} \right) v_g(\beta) \vol_{g_\upsilon} \right) 
d \upsilon, 
$$
where $\vol_{g_\upsilon}$ is the induced volume form on the boundary $\partial B_\upsilon (p)$. Hence 
$$
\frac{\partial}{\partial \tau} J_{\tau} 
= 
\int_{\partial B_\tau (p)} f 
\Xi \left(\beta, \frac{\partial}{\partial r} \right) v_g(\beta) \vol_{g_\tau}. 
$$
Then by the coarea formula again, we obtain 
\begin{align}
\int_\sigma^\rho \left( \frac{e^{a \tau^2}}{\tau^\kappa} 
\frac{\partial}{\partial \tau} J_{\tau} \right) d \tau 
=
\int_{B_\rho (p) \backslash B_\sigma (p)} \frac{e^{a r^2}}{r^\kappa} f  
\Xi \left(\beta, \frac{\partial}{\partial r} \right) v_g(\beta) \vol_g. 
\end{align}
The last two statements follow from the discussion in Section \ref{sec:mono notation}. 
\end{proof}

\subsection{The volume case} \label{sec:vol}
Next, we compute the difference in 
(radius-normalized) volume defined in \eqref{eq:def vol func} 
on geodesic balls of different radii, 
and consider that under what conditions we can obtain the monotonicity.

\begin{theorem} \label{thm:mono vol}
Let $(X,g)$ be an oriented $n$-dimensional Riemannian manifold. 
Fix $p \in X$ and use the notation in Section \ref{sec:mono notation}. 
Then there exists a constant $a=a(n,p,g) \geq 0$ such that for 
a 2-form $\beta \in \Omega^2$ satisfying a conservation law, 
a nonnegative smooth function $f$ on $X$, 
$\kappa \in \rl$ and $0 < \sigma < \rho \leq r_p$,  
we have 
\begin{align}
&\frac{e^{a \rho^2}}{\rho^\kappa} \int_{B_\rho (p)} f v_g (\beta) \vol_g 
-
\frac{e^{a \sigma^2}}{\sigma^\kappa} \int_{B_\sigma (p)} f v_g (\beta) \vol_g \\
\geq & 
\int_{B_\rho (p) \backslash B_\sigma (p)} \frac{e^{a r^2}}{r^\kappa} f 
\left(1- g \left( G_\beta^{-1} \left(\frac{\partial}{\partial r} \right), 
\frac{\partial}{\partial r} \right) \right) v_g (\beta) \vol_g \\
&+ 
\int_\sigma^\rho \frac{e^{a \tau^2}}{\tau^{\kappa +1}} 
\left(
\int_{B_\tau (p)} \left( \left(\tr (G_\beta^{-1})-\kappa \right) f 
+ \frac{\tau^2 - r^2}{2} (-\Delta_\beta f) \right) v_g (\beta) \vol_g 
\right) 
d \tau. 
\end{align}
Furthermore, the last two statements of Theorem \ref{thm:mono} 
about $a$ and $\delta_0$ also hold. 
\end{theorem}

Since $1- g \left( G_\beta^{-1} \left(\frac{\partial}{\partial r} \right), 
\frac{\partial}{\partial r} \right) \geq 0$, 
the following holds from Theorem \ref{thm:mono vol}.

\begin{corollary} \label{cor:mono vol sh}
In addition to assumptions in Theorem \ref{thm:mono vol}, 
suppose further that $\tr (G_\beta^{-1}) \geq \kappa$ 
and a function $f$ satisfies $-\Delta_\beta f \geq 0$. 
Then 
$$
(0,r_p] \to \rl, \qquad 
\rho \mapsto \frac{e^{a \rho^2}}{\rho^\kappa} \int_{B_\rho (p)} 
f v_g (\beta)  \vol_g
$$
is non-decreasing. 
\end{corollary}

\begin{proof}[Proof of Theorem \ref{thm:mono vol}]
Use the notation in the proof of Theorem \ref{thm:mono}. 
Set 
$$
A_{\eps, \tau} = \int_X f \phi_{\eps, \tau} v_g (\beta) \vol_g, 
\qquad
K_{\eps, \tau} = \int_X f \phi_{\eps, \tau} 
\left(1- g \left( G_\beta^{-1} \left(\frac{\partial}{\partial r} \right), 
\frac{\partial}{\partial r} \right) \right) v_g (\beta) \vol_g. 
$$
Recall 
$I_{\eps, \tau}, J_{\eps, \tau}$ and $Q_{\eps, \tau}$ defined in \eqref{eq:def IJQ}. 
Then 
$$
A_{\eps, \tau} = I_{\eps, \tau} - J_{\eps, \tau} + K_{\eps, \tau}. 
$$
Then by \eqref{eq:int A}, we have 
\begin{align}
&\frac{\partial}{\partial \tau} 
\left( \frac{e^{a \tau^2}}{\tau^\kappa} A_{\eps, \tau} \right) \\
=&
\frac{e^{a \tau^2}}{\tau^{\kappa -1}} 
\left \{ 2a A_{\eps, \tau} 
-\frac{\kappa}{\tau^2} A_{\eps, \tau} 
+ 
\frac{1}{\tau^2} 
\left( 
I_{\eps, \tau} 
+ O(1) c(p) \tau^2 J_{\eps, \tau}
- 
Q_{\eps, \tau}
\right)
+ \frac{1}{\tau} \frac{\partial}{\partial \tau} K_{\eps, \tau} 
\right \} \\
=& 
\frac{e^{a \tau^2}}{\tau^{\kappa -1}} 
\left \{ 2a A_{\eps, \tau} +O(1)c(p) J_{\eps, \tau} 
+ \frac{1}{\tau} \frac{\partial}{\partial \tau} K_{\eps, \tau} 
+ \frac{1-\kappa}{\tau^2} I_{\eps, \tau}
+\frac{\kappa}{\tau^2} (J_{\eps, \tau} - K_{\eps, \tau})
- \frac{1}{\tau^2} Q_{\eps, \tau} 
\right \} \\
=&
\frac{e^{a \tau^2}}{\tau^{\kappa -1}} 
\left \{ 2a A_{\eps, \tau} +O(1)c(p) J_{\eps, \tau} 
+ \frac{1}{\tau} \frac{\partial}{\partial \tau} K_{\eps, \tau} 
+ \frac{1}{\tau^2} 
\left(
\int_X f \phi_{\eps, \tau} \left(\tr \left(G_\beta^{-1} \right)-\kappa \right) v_g (\beta) \vol_g
- Q_{\eps, \tau} \right) 
\right \}. 
\end{align}
Since  
$\Xi \left(\beta, \frac{\partial}{\partial r} \right) \leq \tr \left(G_\beta^{-1} \right) \leq n$, 
we have 
$J_{\eps, \tau} \leq n A_{\eps, \tau}. $
Thus we may write 
$$
O(1)c(p) J_{\eps, \tau} = O(1)c(p) A_{\eps, \tau}. 
$$
Hence if we choose $a \geq 0$ such that $2a+O(1)c(p) \geq 0$, we obtain 
\begin{align} \label{eq:diff eA}
\frac{\partial}{\partial \tau} 
\left( \frac{e^{a \tau^2}}{\tau^\kappa} A_{\eps, \tau} \right)
\geq 
\frac{e^{a \tau^2}}{\tau^{\kappa -1}} 
\left \{  
\frac{1}{\tau} \frac{\partial}{\partial \tau} K_{\eps, \tau} 
+ \frac{1}{\tau^2} 
\left(
\int_X f \phi_{\eps, \tau} \left(\tr \left(G_\beta^{-1} \right)-\kappa \right) v_g (\beta) \vol_g
- Q_{\eps, \tau} \right) 
\right \}. 
\end{align}
Then the result follows by integrating over $[\sigma, \rho]$ and 
letting $\eps$ tend to zero 
as in the proof of Theorem \ref{thm:mono}. 
\end{proof}

%%%%%%%%%%%%%%%%%%%%%%%%%%%%%%%%%%%%%%%%%%%%%%%%%%%

\subsection{The normalized volume case} \label{sec:normalized vol}

We compute the difference in 
(radius-normalized) normalized volume defined in \eqref{eq:def nor vol func} 
on geodesic balls of different radii, 
and consider that under what conditions we can obtain the monotonicity. 
This case will be important by the property \eqref{eq:norvol 00}. 
As a corollary, we obtain the vanishing theorems under some conditions in 
Corollary \ref{cor:vanish} and Remark \ref{rem:2dim cons law}.

\begin{theorem} \label{thm:mono nvol 1}
Let $(X,g)$ be an oriented $n$-dimensional Riemannian manifold. 
Fix $p \in X$ and use the notation in Section \ref{sec:mono notation}. 
Take a constant $a=a(n,p,g) \geq 0$ as in Theorem \ref{thm:mono vol}. 
For a smooth function $\theta :[0, r_p] \to \rl$, define $\Theta :[0, r_p] \to \rl$ by 
$$
\Theta (\tau)=\int_0^\tau \frac{e^{a \zeta^2} \theta (\zeta)}{\zeta^{\kappa -1}} d \zeta. 
$$
Then for a 2-form $\beta \in \Omega^2$ satisfying a conservation law, 
a nonnegative smooth function $f$ on $X$, 
$\kappa \in \rl$ and $0 < \sigma < \rho \leq r_p$, 
we have 
\begin{align}
&\left( 
\frac{e^{a \rho^2}}{\rho} \int_{B_\rho (p)} f (v_g (\beta)-1) \vol_g 
+ 2a \Theta (\rho) 
\right)
-
\left( 
\frac{e^{a \sigma^2}}{\sigma} \int_{B_\sigma (p)} f (v_g (\beta)-1) \vol_g 
+ 2a \Theta (\sigma)
\right)  \\
\geq & 
2a \int_\sigma^\rho \frac{e^{a \tau^2}}{\tau^{\kappa -1}} 
\left( \theta (\tau) 
- \int_{B_\tau (p)} f \vol_g 
\right) d \tau 
+
\int_{B_\rho (p) \backslash B_\sigma (p)} \frac{e^{a r^2}}{r^\kappa} f 
\left(1- g \left( G_\beta^{-1} \left(\frac{\partial}{\partial r} \right), 
\frac{\partial}{\partial r} \right) \right) v_g (\beta) \vol_g \\
&+ 
\int_\sigma^\rho \frac{e^{a \tau^2}}{\tau^{\kappa +1}} 
\left(
n \int_{B_\tau (p)} f \vol_g - \tau \frac{\partial}{\partial \tau} 
\int_{B_\tau (p)} f \vol_g
\right) 
d \tau \\
&+ 
\int_\sigma^\rho \frac{e^{a \tau^2}}{\tau^{\kappa +1}} 
\left( 
\int_{B_\tau (p)} 
f \left \{
- (v_g(\beta)-1) \kappa + \left({\rm tr} (G_\beta^{-1}) v_g (\beta) -n \right)
\right \} \vol_g 
\right) 
d \tau \\
&+ 
\int_\sigma^\rho \frac{e^{a \tau^2}}{\tau^{\kappa +1}} 
\left(
\int_{B_\tau (p)} \left( 
\frac{\tau^2 - r^2}{2} (-\Delta_\beta f) \right) v_g (\beta) \vol_g 
\right) 
d \tau. 
\end{align}
Furthermore, the last two statements of Theorem \ref{thm:mono} 
about $a$ and $\delta_0$ also hold. 
\end{theorem}

\begin{proof}
Use the notation in the proof of Theorems \ref{thm:mono} and \ref{thm:mono vol}. 
Set 
$$
L_{\eps, \tau}
= 
\int_X f \phi_{\eps, \tau} \vol_g. 
$$

By \eqref{eq:diff eA}, we have 
\begin{align}
& \frac{\partial}{\partial \tau} 
\left( \frac{e^{a \tau^2}}{\tau^\kappa} (A_{\eps, \tau} - L_{\eps, \tau}) + 2a \Theta (\tau) \right) \\
\geq &
\frac{e^{a \tau^2}}{\tau^{\kappa -1}}
\left \{  
\frac{1}{\tau} \frac{\partial}{\partial \tau} K_{\eps, \tau} 
+ \frac{1}{\tau^2} 
\left(
\int_X f \phi_{\eps, \tau} \left(\tr \left(G_\beta^{-1} \right)-\kappa \right) v_g (\beta) \vol_g
- Q_{\eps, \tau} \right) \right. \\
&\left. 
+2a \left(
\theta (\tau) - L_{\eps, \tau}
\right) 
+ \frac{\kappa}{\tau^2} L_{\eps, \tau}
- \frac{1}{\tau} \frac{\partial}{\partial \tau} L_{\eps, \tau}
\right \} \\
=&  
\frac{e^{a \tau^2}}{\tau^{\kappa -1}}
\left \{  
2a \left(
\theta (\tau) - L_{\eps, \tau}
\right) 
+ \frac{1}{\tau} \frac{\partial}{\partial \tau} K_{\eps, \tau} 
+ \frac{1}{\tau^2} \left( n L_{\eps, \tau} 
- \tau \frac{\partial}{\partial \tau} L_{\eps, \tau} \right) \right. \\
&\left. 
+ 
\frac{1}{\tau^2} 
\int_X f \phi_{\eps, \tau} 
\left(
- (v_g(\beta)-1) \kappa + 
\left({\rm tr}(G_\beta^{-1}) v_g (\beta) -n \right)
\right) \vol_g
- \frac{1}{\tau^2} Q_{\eps, \tau} \right \}. \\
\end{align}
Then the result follows by integrating over $[\sigma, \rho]$ and 
let $\eps$ tend to zero 
as in the proof of Theorem \ref{thm:mono}. 
\end{proof}

By Theorem \ref{thm:mono nvol 1}, we immediately see the following. 

\begin{corollary} \label{cor:mono nvol sh}
In addition to the assumptions in Theorem \ref{thm:mono nvol 1}, 
suppose further the following. 
\begin{itemize}
\item
There exists $0 < r'_p \leq r_p$ such that for any $\tau \in [0, r'_p]$ 
$$
n \int_{B_\tau (p)} f \vol_g 
\geq 
\tau \frac{\partial}{\partial \tau} \int_{B_\tau (p)} f \vol_g. 
$$

\item
For any $\tau \in [0, r'_p]$ 
$$
\theta (\tau) \geq \int_{B_\tau (p)} f \vol_g. 
$$

\item
$- (v_g(\beta)-1) \kappa + 
\left({\rm tr}(G_\beta^{-1}) v_g (\beta) -n \right) \geq 0$. 

\item 
$-\Delta_\beta f \geq 0$.  
\end{itemize}

Then 
$$
(0,r'_p] \to \rl, \qquad 
\rho \mapsto \frac{e^{a \rho^2}}{\rho^\kappa} \int_{B_\rho (p)} f (v_g (\beta)-1) \vol_g 
+ 2a \Theta (\rho) 
$$
is non-decreasing. 
\end{corollary}

We will give an example that 
the first two assumptions of Corollary \ref{cor:mono nvol sh} 
are satisfied for $f=1$ and $\theta (\tau)=\omega_n \tau^n$,  
where $\omega_n = \frac{2 \pi^{n/2}}{n \Gamma(n/2)}$, 
which is the volume of the unit ball in $\rl^n$. 
In this case, these two assumptions are geometrical 
and can be satisfied under certain curvature conditions.

\begin{lemma} \label{lem:cond assump}
Let $(X,g)$ be an oriented $n$-dimensional Riemannian manifold. 
Denote by $R$, ${\rm Ric}$, ${\rm Scal}$ 
the curvature tensor, Ricci curvature, scalar curvature, respectively.

Fix $p \in X$. 
Denote by $T_p$ the value of a tensor $T$ at $p$.  
Suppose that one of the following conditions is satisfied. 
\begin{enumerate}
    \item ${\rm Scal}_p >0$. 
    \item ${\rm Scal}_p = 0$ and $3|R_p|^2-8|{\rm Ric}_p|^2 - 18 (\Delta {\rm Scal})_p >0$, 
    where $\Delta {\rm Scal} = d^* d {\rm Scal}$. 
    \item ${\rm Ric} \geq 0$ on $X$.   
\end{enumerate}
Then there exists $0 < r'_p \leq r_p$ such that 
for any $\tau \in [0, r'_p]$ 
$$
n \int_{B_\tau (p)} \vol_g 
\geq 
\tau \frac{\partial}{\partial \tau} \int_{B_\tau (p)} \vol_g, 
\qquad 
\mbox{and} \qquad 
\omega_n \tau^n \geq \int_{B_\tau (p)} \vol_g. 
$$
\end{lemma}

The assumptions (1), (2) and (3) are used to control the volume of geodesic balls 
centered at $p$.

\begin{proof}
For (1) and (2), the statement follows from the Taylor expansion of 
the volume of the geodesic ball with respect to the radius. 
By \cite[Corollaries 3.2 and 3.3]{gray1974geodesicball}, 
the volume of the geodesic ball of radius $\tau$ centered at $p$ is given by 
\begin{align}
\int_{B_\tau (p)} \vol_g
=
\omega_n \tau^n \left \{ 
1 - \frac{{\rm Scal}_p}{6(n+2)} \tau^2 
+ \frac{-3|R_p|^2+8|{\rm Ric}_p|^2 + 5 {\rm Scal}_p^2+ 18 (\Delta {\rm Scal})_p}{360(n+2)(n+4)} \tau^4 
+ O (\tau^6) 
\right \}
\end{align}
and higher order terms are described by derivatives of the curvature tensor. 
Then we have 
\begin{align}
&n \int_{B_\tau (p)} \vol_g 
-
\tau \frac{\partial}{\partial \tau} \int_{B_\tau (p)} \vol_g\\ 
=&
\omega_n \tau^n \left \{ 
\frac{{\rm Scal}_p}{3(n+2)} \tau^2 
+ \frac{3|R_p|^2-8|{\rm Ric}_p|^2 - 5 {\rm Scal}_p^2- 18 (\Delta {\rm Scal})_p}{90(n+2)(n+4)} \tau^4 
+ O (\tau^6) 
\right \}. 
\end{align}
Thus the statement holds when (1) or (2) is satisfied.

For (3), by the volume estimation and the relative volume comparison theorem 
(cf. \cite[p. 269, Lemmas 35 and 36]{petersen2006riem} 
or \cite[Theorems I\hspace{-.1em}I\hspace{-.1em}I.4.4 and 
I\hspace{-.1em}I\hspace{-.1em}I.4.5]{chavel2006}
), we have 
\begin{align}
\omega_n \tau^n \geq \int_{B_\tau (p)} \vol_g,  \quad 
0 \geq 
\frac{\partial}{\partial \tau} \left( \frac{1}{\omega_n \tau^n} \int_{B_\tau (p)} \vol_g \right)
= 
\frac{1}{\omega_n \tau^{n+1}} 
\left(
\tau \frac{\partial}{\partial \tau} \int_{B_\tau (p)} \vol_g 
-n \int_{B_\tau (p)} \vol_g 
\right). 
\end{align}
\end{proof}

%%%%%%%%%%%%%%%%%%%%%%%%%%%%%%%%%%%%%%%%%%%%%%%%%%%%%%%%%%%%%%%%%%%%%%
\subsection{Odd dimensional case} \label{sec:odd dim}
In this subsection, we provide examples where the conditions for monotonicity 
(Corollaries \ref{cor:mono vol sh} and \ref{cor:mono nvol sh}) are satisfied. 
We assume that a manifold is odd dimensional. 
Then we can show the following algebraic estimate. 
This is a key technical statement in the odd dimensional case.

\begin{lemma} \label{lem:ineq tr}
Set $\dim X=n=2m+1$. 
For any $\beta \in \Omega^2$, we have 
$$
{\rm tr}(G_\beta^{-1}) \geq 1+\frac{2m}{v_g(\beta)}. 
$$
\end{lemma}

\begin{proof}
We prove this by the pointwise computation. 
Fix $p \in X$. 
Since $\bb$ is skew-symmetric, there exists an orthonormal basis $\{f^j \}_{j=1}^{2m+1}$ of $T^*_p X$ 
and $\lambda_j \in \rl$ such that 
$$
\beta=\sum_{j=1}^{m} \lambda_j f^{2j-1} \wedge f^{2j} 
$$
at $p$. Setting $\mu_j = 1+\lambda_j^2$, we have 
$$
G_\beta=
\left(
\begin{array}{cc}
\mu_1 & 0 \\
0 & \mu_1 
\end{array}
\right) 
\oplus \cdots \oplus 
\left(
\begin{array}{cc}
\mu_m & 0 \\
0 & \mu_m 
\end{array}
\right)
\oplus 1. 
$$
Since $X$ is odd dimensional, $G_\beta$ always has eigenvalue 1. 
This is important for the estimate. 
Then 
$$
v_g(\beta)^2 = \sqrt{\det G_\beta} = \prod_{j=1}^m \mu_j, 
\qquad 
{\rm tr}(G_\beta^{-1}) = 1 + \sum_{j=1}^m \frac{2}{\mu_j}, 
$$
and 
$$
v_g(\beta)^2 \left \{ {\rm tr}(G_\beta^{-1}) 
- \left( 1+\frac{2m}{v_g(\beta)} \right) \right \}
=
\prod_{j=1}^m \mu_j \sum_{k=1}^m \frac{2}{\mu_k}
- 2m \prod_{j=1}^m \sqrt{\mu_j}
=
2 \sum_{j=1}^m \nu_j - 2m \prod_{j=1}^m \sqrt{\mu_j}, 
$$
where we set $\nu_j=\prod_{k=1}^m \mu_k/\mu_j$. Since 
$$
\sum_{j=1}^m (\sqrt{\nu_j}-\sqrt{\mu_j})^2= 
\sum_{j=1}^m (\nu_j+\mu_j) -2m \prod_{j=1}^m \sqrt{\mu_j}, 
$$
it follows that 
\begin{align}
v_g(\beta)^2 \left \{ {\rm tr}(G_\beta^{-1}) 
- \left( 1+\frac{2m}{v_g(\beta)} \right) \right \}
=
\sum_{j=1}^m (\sqrt{\nu_j}-\sqrt{\mu_j})^2 
+ \sum_{j=1}^m (\nu_j-\mu_j).
\end{align}
Since $\mu_j \geq 1$, we have $\nu_j \geq \mu_k$ for $j \neq k$. Then 
\begin{align}
\sum_{j=1}^m (\nu_j-\mu_j)   
= \sum_{j=1}^{m-1} \nu_{j+1} + \nu_1 - \sum_{j=1}^{m-1} \mu_j - \mu_m
\geq
\sum_{j=1}^{m-1} \mu_j + \mu_m - \sum_{j=1}^{m-1} \mu_j - \mu_m
=0
\end{align}
and the proof of Lemma \ref{lem:ineq tr} is completed. 
\end{proof}

\begin{remark}
When $\dim X=n=2m$, we have 
${\rm tr}(G_\beta^{-1}) v_g(\beta) \geq 2m$ 
for any $\beta \in \Omega^2$ by the proof of Lemma \ref{lem:ineq tr}. 
Since $v_g (\beta)\geq 1$, it follows that 
${\rm tr}(G_\beta^{-1}) v_g (\beta) \geq n$ for any $n$ and $\beta \in \Omega^2$. 
\end{remark}

By Lemma \ref{lem:ineq tr}, 
we have ${\rm tr}(G_\beta^{-1}) \geq 1$ in the odd dimensional case. 
By Corollary \ref{cor:mono vol sh}, we have the following 
for the volume functional. 

\begin{corollary} \label{cor:mono vol odd dim}
Let $(X,g)$ be an oriented odd dimensional Riemannian manifold. 
Fix $p \in X$ and use the notation in Section \ref{sec:mono notation}. 
Then there exists 
a constant $a=a(n,p,g) \geq 0$ such that 
for a 2-form $\beta \in \Omega^2$ satisfying a conservation law and 
a nonnegative smooth function $f$ on $X$ with $-\Delta_\beta f \geq 0$, 
$$
(0,r_p] \to \rl, \qquad 
\rho \mapsto \frac{e^{a \rho^2}}{\rho} \int_{B_\rho (p)} 
f v_g (\beta)  \vol_g
$$
is non-decreasing. 
In particular, setting $f=1$, we see that 
$$
(0,r_p] \to \rl, \qquad 
\rho \mapsto \frac{e^{a \rho^2}}{\rho} \int_{B_\rho (p)} v_g(\beta)  \vol_g
$$
is non-decreasing. 
\end{corollary}

By Lemma \ref{lem:ineq tr}, we have 
\begin{align}
\left. - (v_g(\beta)-1) \kappa + 
\left({\rm tr}(G_\beta^{-1}) v_g (\beta) -n \right)
\right|_{\kappa=1}
\geq 
- (v_g(\beta)-1) + v_g (\beta)-1 
= 0.     
\end{align}
So the third assumption of Corollary \ref{cor:mono nvol sh} is satisfied 
in the odd dimensional case. 
This together with Lemma \ref{lem:cond assump} 
implies the following result for the normalized volume functional.

\begin{theorem} \label{thm:mono nvol 3} 
Let $(X,g)$ be an oriented Riemannian manifold of dimension $n=2m+1$. 
Fix $p \in X$ and use the notation in Section \ref{sec:mono notation}. 
Suppose further that one of the following conditions is satisfied. 
\begin{enumerate}
    \item ${\rm Scal}_p >0$. 
    \item ${\rm Scal}_p = 0$ and $3|R_p|^2-8|{\rm Ric}_p|^2 - 18 (\Delta {\rm Scal})_p >0$, 
    where $\Delta {\rm Scal} = d^* d {\rm Scal}$.  
    \item ${\rm Ric} \geq 0$ on $X$.   
\end{enumerate}
Take a constant $a=a(n,p,g) \geq 0$ 
as in Theorem \ref{thm:mono vol} and 
define $\Theta: [0, \infty) \to [0, \infty)$ by 
\begin{align} \label{eq:Theta monot nvol}
\Theta (\tau)=\omega_n \int_0^\tau e^{a \zeta^2} \zeta^n d \zeta, 
\end{align}
where $\omega_n = \frac{2 \pi^{n/2}}{n \Gamma(n/2)}$, which is the volume of 
the unit ball in $\rl^n$. 
Then there exists $0 < r'_p \leq r_p$ such that 
for a 2-form $\beta \in \Omega^2$ satisfying a conservation law, 
$$
(0,r'_p] \to \rl, \qquad 
\rho \mapsto \frac{e^{a \rho^2}}{\rho} \int_{B_\rho (p)} (v_g (\beta)-1) \vol_g 
+ 2a \Theta (\rho) 
$$
is non-decreasing. 
\end{theorem}

\begin{remark}
When $a>0$, the function $\Theta$ is explicitly written as 
$$
\Theta (\tau)
=
\frac{\omega_{2m+1} \cdot m!}{2 (-a)^{m+1}}
\left(
1-e^{a \tau^2} \sum_{k=0}^m \frac{(-a \tau^2)^k}{k!}
\right). 
$$
In particular, we have 
\begin{align}
\Theta (\tau)= 
\left\{
\begin{array}{ll}
\frac{1}{a} (e^{a \tau^2}-1) & \mbox{for} \quad n=2m+1=1,\\
\frac{2 \pi}{3 a^2} \{(a \tau^2-1) e^{a \tau^2}+1 \} &  \mbox{for} \quad n=2m+1=3,\\
\frac{4 \pi^2}{15 a^3} \{(a^2 \tau^4-2a \tau^2+2) e^{a \tau^2}-2 \} &  \mbox{for} \quad n=2m+1=5,\\
\frac{8 \pi^3}{105 a^4} \{(a^3 \tau^6-3a^2 \tau^4+6a \tau^2-6) e^{a \tau^2}+6 \} &  \mbox{for} \quad n=2m+1=7.\\
\end{array}
\right.
\end{align}

\end{remark}

When $(X,g)$ is $(\rl^{2m+1}, g_0)$, where $g_0$ is the standard flat metric, 
we can take $a=0$ and $r_p'=\infty$. 
Hence Theorem \ref{thm:mono nvol 3} gives 
the monotonicity formula for the normalized volume in the case of $(\rl^{2m+1}, g_0)$. 
As an application, we can prove the following vanishing theorem.

\begin{corollary} \label{cor:vanish}
Let $\beta \in \Omega^2(\rl^{2m+1})$ be a 2-form 
satisfying a conservation law. 
Suppose that for some $p \in \rl^{2m+1}$, we have 
$$
\int_{B_r(p)} (v_{g_0}(\beta)-1) \vol_{g_0} = o(r) 
\qquad \mbox{as} \qquad r \to \infty. 
$$
Then $\beta=0$. 

In particular, if $\beta \in \Omega^2(\rl^{2m+1})$ is a 2-form 
satisfying a conservation law and 
$
\int_{\rl^{2m+1}} (v_{g_0}(\beta)-1) \vol_{g_0} < \infty,  
$
then $\beta=0$. 
\end{corollary}

\begin{proof}
Suppose for contradiction that $\beta \neq 0$. 
Recalling the property \eqref{eq:norvol 00}, we see that there exists $R_0>0$ such that 
$$
\frac{1}{R_0} \int_{B_{R_0} (p)} (v_{g_0}(\beta)-1) \vol_{g_0} >0. 
$$
On the other hand, Theorem \ref{thm:mono nvol 3} implies that 
$$
\frac{1}{R} \int_{B_{R} (p)} (v_{g_0}(\beta)-1) \vol_{g_0} 
\geq \frac{1}{R_0} \int_{B_{R_0} (p)} (v_{g_0}(\beta)-1) \vol_{g_0} >0 
$$
for any $R \geq R_0$. 
By assumption, the left hand side goes to zero as $R \to \infty$, 
which is a contradiction. 
\end{proof}

\subsection{2-dimensional case} \label{sec:2dim}
In this subsection, we assume that 
$(X,g)$ is an oriented Riemannian manifold of dimension 2. 
This case is rather special and 
we can show the monotonicity formulas for the (radius-normalized) 
volume and normalized volume functional under some conditions 
directly from Corollary \ref{cor:mono sh}. 
The following is a special property in the 2-dimensional case. 

\begin{lemma} \label{lem:2dim cons}
For a 2-form $\beta \in \Omega^2$, we have the following. 
\begin{enumerate}
    \item 
    $\delta_\beta \beta =0$ if and only if $D \beta =0$. 
    \item
When $(X,g)$ and $\beta$ are real analytic, 
$\beta$ satisfies a conservation law if and only if $D \beta =0$. 
\end{enumerate}
\end{lemma}

\begin{proof}
Take a local orthonormal frame $\{ e_i \}$ with its dual $\{ e^i \}$. 
Set 
$$
\beta = \lambda e^1 \wedge e^2, \qquad 
D_i \beta = a_i e^1 \wedge e^2
$$
for some local functions $\lambda, a_i$. Then we have 
\begin{align}
\bb = \lambda (e^1 \otimes e_2 - e^2 \otimes e_1), \qquad
\gi = \frac{1}{1+\lambda^2} \id_{TX}. 
\end{align}
We compute 
\begin{align}
\delta_\beta \beta 
= -i(G_\beta^{-1} (e_i)) D_i \beta
= \frac{1}{1+\lambda^2} (a_2 e_1-a_1 e_2),  
\end{align}
which implies (1). 

Next, we prove (2). Since $X$ is 2-dimensional, we have $d \beta =0$. 
Then by Proposition \ref{prop:div set}, $\beta$ satisfies a conservation law if and only if 
\begin{align} \label{eq:2dim cons law}
\begin{split}
\bb
\left((\delta_\beta \beta)^\sharp \right) =0
&\quad
\Longleftrightarrow
\quad
\lambda a_1 = \lambda a_2 =0 \\
&\quad
\Longleftrightarrow
\quad
\mbox{for any $p \in X$, } \beta_p=0 \quad \mbox{or} \quad (D \beta)_p=0. 
\end{split}
\end{align}
Set $Z=\{ p \in X \mid \beta_p=0 \}$. 
Since $\beta$ is real analytic, 
$\beta$ is constant zero ($X=Z$) or 
$X-Z$ is open and dense by the identity theorem. 
In the first case, we obviously have $D \beta =0$. 
In the latter case, we have $D \beta =0$ on $X-Z$ by \eqref{eq:2dim cons law}. 
Since $D \beta$ is continuous, we have $D \beta =0$ on $X$, and the proof is completed. 
\end{proof}

\begin{theorem} \label{thm:monotnicity 2dim}
Let $(X,g)$ be an oriented $2$-dimensional Riemannian manifold. 
Fix $p \in X$ and use the notation in Section \ref{sec:mono notation}. 
Then there exists
a constant $a=a(n,p,g) \geq 0$ such that for 
\begin{itemize}
\item 
a 2-form $\beta \in \Omega^2$ satisfying $\delta_\beta \beta =0$, or 
satisfying a conservation law if 
$(X,g)$ and $\beta$ are real analytic, 

\item 
a nonnegative smooth function $f$ on $X$ with $-\Delta_\beta f \geq 0$, 
\end{itemize}
functions 
$$
(0,r_p] \to \rl, \qquad 
\rho \mapsto \frac{e^{a \rho^2}}{\rho} \int_{B_\rho (p)} 
f v_g (\beta) \vol_g, 
\qquad
\mbox{and}
\qquad
\rho \mapsto \frac{e^{a \rho^2}}{\rho} \int_{B_\rho (p)} 
f (v_g (\beta)-1) \vol_g, 
$$
are non-decreasing. 
Furthermore, the last two statements of Theorem \ref{thm:mono} 
about $a$ and $\delta_0$ also hold. 
\end{theorem}

\begin{proof}
Set 
$$
\chi = \frac{1}{{\rm tr} (G_{\beta }^{-1})} 
\qquad \mbox{or} \qquad \frac{1}{{\rm tr} (G_{\beta }^{-1})} \left( 1 - \frac{1}{v_g(\beta)}\right). 
$$
By Lemma \ref{lem:2dim cons}, we have $D \beta=0$. 
Then by \eqref{eq:diff ginv} and Lemma \ref{lem:cl diff}, 
we have $D \gi =0$ and $D v_g(\beta)=0$, 
which implies that $\chi$ is constant. 
By the definition of $\Delta_\beta$ in \eqref{eq:lapb}, we compute 
\begin{align}
\Delta_\beta (f \chi) 
=& \delta_\beta (\chi df + f d \chi) \\
=& 
-i(G_\beta^{-1} (e_i)) 
\left( e_i (\chi) df +\chi D_i df + e_i (f) d \chi + f D_i d \chi \right) \\
=& 
f \Delta_\beta \chi + \chi \Delta_\beta f - 2 \left \la (\gi)^* d f, d \chi \right \ra. 
\end{align}
Since $\chi$ is constant, we have 
$$
- \Delta_\beta (f \chi) = - \chi \Delta_\beta f \geq 0.  
$$
Replacing $f$ with $f \chi$ in Corollary \ref{cor:mono sh}, the proof is completed. 
\end{proof}

\begin{remark} \label{rem:2dim cons law}
Under the assumptions of Theorem \ref{thm:monotnicity 2dim}, 
we can also obtain the vanishing theorem as in Corollary \ref{cor:vanish}. 

When $\dim X=2$, the normalized volume functional agrees with a simpler functional 
defined in \cite[(2.5)]{Dong2011vanishing} 
for $F(t)=\sqrt{1+2t}-1$ by Remark \ref{rem:norvol 00}. 
Then functional for this $F$ is called 
the generalized Yang--Mills--Born--Infeld energy functional with the
plus sign 
when $\beta$ is replaced with the curvature of a connection of a vector bundle 
(\cite[(8.1)]{Dong2011vanishing}).
We can directly check that 
$\beta \in \Omega^2$ satisfies a conservation law in our sense if and only if 
$\beta \in \Omega^2$ satisfies an $F$-conservation law 
in the sense of \cite[Definition 2.1]{Dong2011vanishing} 
by \eqref{eq:2dim cons law} and \cite[Lemma 2.3]{Dong2011vanishing}.

We only use the fact that $D \beta =0$ in the proof of Theorem \ref{thm:monotnicity 2dim}. 
We obtain similar results in any dimensions when $\beta$ satisfies $D \beta =0$. 

As in the proof of Theorem \ref{thm:monotnicity 2dim}, 
we can consider $\Delta_\beta (f \chi)$ or $\Delta_\beta \chi$ in any dimensions. 
We tried to compute using the Weitzenb\"ock-type formula (Proposition \ref{prop:weit}), 
but we could not obtain nice estimates like  $- \Delta_\beta (f \chi) \geq 0$. 
If we can obtain such estimates, we will obtain the monotonicity formulas 
for the volume and the normalized volume without additional assumptions 
as in the 2-dimensional case.
\end{remark}

%%%%%%%%%%%%%%%%%%%%%%%%%%%%%%%%%%%%%%%%%%%%%%%%%%%%%%%%%%%%%%%%%%%%%%
\section{Minimal connections} \label{sec:min conn}
Use the notation in Section \ref{sec:cons not}. 
Let 
$(X,g)$ be an oriented $n$-dimensional Riemannian manifold 
and $L \to X$ be a smooth complex line bundle with a Hermitian metric $h$. 
Set 
\[
\begin{aligned}
\mathcal{A}_{0}=\{\, \mbox{Hermitian connections of }(L,h) \,\}
= \nabla_0 + \i \Omega^1 \cdot \id_L, 
\end{aligned}
 \]
where $\nabla_0 \in \cA_{0}$ is any fixed connection. 
We regard the curvature 2-form $F_\n$ of $\n$ as a $\i \rl$-valued closed 2-form on $X$. 
For simplicity, set 
\begin{align} \label{eq:nabla Gnabla}
\FF := - \i F_\n \in \Omega^2, \qquad 
G_\n:=G_{\FF}= \id_{TX}-\FFs \circ \FFs. 
\end{align}
As in Section \ref{sec:def vol}, 
define the {\bf volume functional} $V: \cA_0 \rightarrow [0, \infty]$ and the {\bf normalized volume functional} 
$V^0: \cA_0 \rightarrow [0, \infty]$
by 
\[
V(\n) = V(\FF)
=
\int_X v(\n) \vol_g, 
\qquad 
V^0(\n) = V^0(\FF) = \int_X (v(\n)-1) \vol_g, 
\]
where 
$$
v(\n)=
v_g (\FF) 
= (\det G_\n)^{1/4}. 
$$

\subsection{The first variation}
In this subsection, we compute the first variation of $V$ and define the mean curvature. 
This was computed in \cite[Proposition 3.2]{kawai2021mirror}, 
but we give a simpler description here. 
We say that a connection is minimal if its mean curvature vanishes. 
Then we show that $\FF$ satisfies a conservation law for each minimal connection $\n$. 
After that, we see that 
the formal ``large radius limit" of the defining equation of minimal connections 
is that of Yang--Mills connections 
and prove the existence theorem.

\begin{proposition} \label{prop:fistvar}
Use the notation in Section \ref{sec:cons not}. 
Let $\{ \n_t \}_{t \in (-\eps, \eps)} \subset \cA_0$ be a compactly supported variation 
of $\n=\n_0 \in \cA_0$ with $V^0(\n) < \infty$. 
Set 
$$
a= \left. \frac{1}{\i} \frac{d}{dt} \n_t \right|_{t=0} \in \Omega^1_c. 
$$
Then we have 
\[
\left. \frac{d}{dt} V^0(\n_t) \right|_{t=0} 
= - \la a, H(\n) \ra_{L^2}. 
\]
Here, $\la \,\cdot\,, \,\cdot\, \ra_{L^2}$ 
is the $L^2$ inner product with respect to the metric $g$, 
$\Omega^1_c \subset \Omega^1$ is the space of compactly supported 1-forms on $X$, 
and $H(\n) \in \Omega^1$ is given by 
\begin{align} \label{eq:MC}
H(\n) 
= v(\n) \cdot (G_\n^{-1})^* \left((D_i \FFs)(G_\n^{-1} (e_i)) \right)^\flat
= v(\n) \cdot (G_\n^{-1})^* \left( i(G_\n^{-1} (e_i)) D_i \FF \right) 
\end{align}
for a local orthonormal frame $\{ e_i \}$. 
\end{proposition}

\begin{proof}
Differentiating $\left(\det G_{\n_t} \right)^{1/4}$, we compute 
\[
\begin{aligned}
\left. \frac{d}{dt} V^0(\n_t) \right|_{t=0} 
=& - \frac{1}{4}\int_X {\rm tr} 
\left(G_\n^{-1}\circ  \left( (da)^\sharp \circ \FFs + \FFs \circ (da)^\sharp \right) \right) \left(\det G_\n\right)^{1/4} \  \vol_g, \\
=&
- \frac{1}{2}\int_X {\rm tr} \left(G_\n^{-1}\circ \FFs \circ (da)^\sharp\right) v(\n)  \vol_g, 
\end{aligned}
\]
where we use \eqref{eq:bG comm}. By \eqref{eq:tr met}, we have 
$$
{\rm tr} \left(G_\n^{-1}\circ \FFs \circ (da)^\sharp\right)=-2\left\la da,\left(G_\n^{-1}\circ \FFs \right)^{\flat} \right\ra. 
$$
Hence it follows that 
\begin{align} \label{eq:mc dstar}
H(\n) =- d^{\ast}\left(v(\n) \left(G_\n^{-1}\circ \FFs \right)^{\flat}\right). 
\end{align}
We compute 
\begin{equation} \label{eq:mc 1}
\begin{split}
H(\n) 
=&
i(e_i) D_i \left(v(\n) \left(G_\n^{-1}\circ \FFs \right)^{\flat}\right)\\
=&
e_i \left( v(\n) \right) \cdot i(e_i) \left(G_\n^{-1}\circ \FFs \right)^{\flat}
+ 
v(\n) i(e_i) D_i \left(G_\n^{-1}\circ \FFs \right)^{\flat} \\
=&
v(\n) 
\left \la \left(G_\n^{-1} \circ \FFs \right)^\flat, D_i \FF \right \ra 
\left( \left(G_\n^{-1}\circ \FFs \right) (e_i) \right)^{\flat} 
+
v(\n) 
\left(
\left( D_i (G_\n^{-1} \circ \FFs) \right) (e_i) 
\right)^\flat, 
\end{split}
\end{equation}
where we use Lemma \ref{lem:cl diff}. 
We compute 
$\left( D_i (G_\n^{-1} \circ \FFs) \right) (e_i)$. 

Since 
$
0=D_i \left(G_\n \circ G_\n^{-1} \right) 
= \left(D_i G_\n \right) \circ G_\n^{-1} + G_\n \circ \left(D_i G_\n^{-1} \right), 
$
we have 
\begin{align} \label{eq:DG-1}
D_i G_\n^{-1} = G_\n^{-1} \circ \left( \left(D_i \FFs \right) \circ \FFs
+ \FFs \circ \left(D_i \FFs \right)
\right) \circ G_\n^{-1}. 
\end{align}
Then we compute 
\begin{align}
D_i \left(G_\n^{-1} \circ \FFs \right)
=&
G_\n^{-1} \circ \left( (D_i \FFs) \circ \FFs
+ \FFs \circ (D_i \FFs) \right) \circ G_\n^{-1} \circ \FFs 
+ G_\n^{-1} \circ (D_i \FFs)\\
=&
G_\n^{-1} \circ (D_i \FFs) \circ G_\n^{-1} 
+ G_\n^{-1} \circ \FFs \circ (D_i \FFs) \circ G_\n^{-1} \circ \FFs, 
\end{align}
where we use \eqref{eq:bGb}. 
Then by Lemma \ref{lem:cl diff} and the fact that $d \FF =0$, we obtain 
$$
\left( D_i (G_\n^{-1} \circ \FFs) \right) (e_i)
= 
\left( G_\n^{-1} \circ (D_i \FFs) \circ G_\n^{-1} \right) (e_i) 
- 
\left \la \left(G_\n^{-1} \circ \FFs \right)^\flat, D_i \FF \right \ra 
\left(G_\n^{-1}\circ \FFs \right) (e_i). 
$$
This together with \eqref{eq:mc 1} implies 
Proposition \ref{prop:fistvar}. 
\end{proof}

\begin{definition}
We call $H(\n)$ the {\bf mean curvature} of $\n \in \cA_0$. 
(Note that $H(\n)$ is also defined when $V^0(\n)=\infty$.) 
We call $\n \in \cA_0$ {\bf minimal} if $H (\n)=0$. 
\end{definition}

By the description of the mean curvature in Proposition \ref{prop:fistvar}, 
we immediately obtain \cite[Proposition 3.3]{kawai2021mirror}, where we computed the symbol 
of the linearization of the mean curvature. 
This was important to show 
the short-time existence and uniqueness of the line bundle mean curvature flow 
in \cite[Theorem 3.7]{kawai2021mirror}. 

As in \eqref{eq:delb}, 
define a differential operator $\delta_\n: \Omega^k \to \Omega^{k-1}$ depending on $\n \in \cA_0$ by 
\begin{align} \label{eq:del nabla}
\delta_\n \alpha 
:= \delta_{E_\n} \alpha
= -i(G_\n^{-1} (e_i)) D_i \alpha. 
\end{align}
(Note that $\delta_\n$ agrees with the standard codifferential $d^*$ if $\n$ is flat.) 
Then $\n$ is minimal if and only if $\delta_\n \FF=0$ by Proposition \ref{prop:fistvar}. 
Note that this characterization is similar to 
that of a Yang--Mills connection. 
See also Remark \ref{rem:large radius limit of min conn}. 

\begin{remark} \label{rem:int by parts mirror interp}
When a manifold has the torus fibration structure, minimal submanifolds 
(more precisely, minimal graphs) correspond to minimal connections 
via the real Fourier--Mukai transform (Lemma \ref{lem:FM min}). 
Similarly, 
the codifferential for 1-forms on a minimal submanifold 
corresponds to $\delta_\n$ for 1-forms (Proposition \ref{prop:FM codiff}). 
Thus minimal connections and $\delta_\n$ for 1-forms 
can be considered as the ``mirrors" of 
minimal submanifolds and the codifferential for 1-forms on a minimal submanifold, respectively.  
\end{remark}

As in \eqref{eq:lapb}, we can define a second order elliptic operator 
$\Delta_\n:\Omega^k \to \Omega^k$ by
$$
\Delta_\n :=\Delta_{E_\n} = d \delta_\n + \delta_\n d. 
$$
Since $d \FF=0$ by the Bianchi identity, we obtain the following 
by Propositions \ref{prop:fistvar} and \ref{prop:div set}. 

\begin{corollary} \label{cor:min cons}
If a Hermitian connection $\n \in \cA_0$ satisfies $H (\n)=0$, then $\Delta_\n \FF =0$. 
\end{corollary}
The converse might hold, but we are not sure if it does.

\begin{theorem} \label{thm:min imply cons}
For each minimal connection $\n$, 
$E_\n$ satisfies a conservation law. 
\end{theorem}

Hence the results in Sections \ref{sec:SET cons law} and \ref{sec:monot formula} 
hold if we replace $\beta$ with $\FF$ for a minimal connection $\n$. 
In particular, we have the following vanishing theorem 
from Corollary \ref{cor:vanish} and Remark \ref{rem:2dim cons law}.

\begin{corollary} \label{cor:vanish min}
Suppose that $n=2m+1$ or $2$. 
Let $(L, h) \to \rl^n$ be a (necessarily trivial) smooth complex Hermitian line bundle 
over $(\rl^n, g_0)$, where $g_0$ is the standard flat metric. 
If $\n$ is a minimal connection satisfying 
$$
\int_{B_r(p)} (v (\n)-1) \vol_{g_0} = o(r) 
\qquad \mbox{as} \qquad r \to \infty 
$$
for some $p \in \rl^n$, 
then $\n$ is flat. 

In particular, if $\n$ is minimal and 
$
\int_{\rl^n} (v(\n)-1) \vol_{g_0} < \infty,  
$
then $\n$ is flat. \\
\end{corollary}

\begin{remark} \label{rem:variational minimal}
We explain why Theorem \ref{thm:min imply cons} holds from the variational point of view 
as in Remark \ref{rem:variational SET}. 
The difference is that $\FF$ cannot take any value in $\Omega^2$. 
But the pullback of diffeomorphisms commutes with the exterior derivative, so a similar result holds. 

Indeed, let $\n$ be a minimal connection with $V^0_g(\n)< \infty$. 
Take any compactly supported vector field $u$ and denote by $\{ \eta_t \}$ the flow of $u$. 
By the diffeomorphism invariance \eqref{eq:nor vol diff inv}, it follows that 
\begin{align}
0= \left. \frac{d}{dt} V^0_{\eta_t^* g} (\eta_t^* \FF) \right|_{t=0}
= 
\left. \frac{d}{dt} V^0_{\eta_t^* g} (\FF) \right|_{t=0}
+ 
\left. \frac{d}{dt} V^0_g (\eta_t^* \FF) \right|_{t=0}
= 
\la L_u g, S_{g, \FF} \ra_{L^2} 
+ 
\left. \frac{d}{dt} V^0_g (\eta_t^* \FF) \right|_{t=0}. 
\end{align}
Since 
$$
L_u \FF = (i(u) d + d i(u)) \FF = d (i (u) \FF)
$$
and $\n$ is minimal, 
it follows that 
$$
\left. \frac{d}{dt} V^0_g (\eta_t^* \FF) \right|_{t=0} 
= 
\left. \frac{d}{dt} V^0_g (\nabla + t \i i (u) \FF \cdot \id_L) \right|_{t=0} 
= 0. 
$$
Then we obtain ${\rm div} S_{g, \FF}=0$ by \cite[Lemma 1.60]{besse2008einstein}. 
\end{remark}

\begin{remark} \label{rem:large radius limit of min conn}
We can show that the formal ``large radius limit" of the defining equation of minimal connections 
is that of Yang--Mills connections. 
Consider the family of metrics 
$$
\{ g_r:=r^2 g \}_{r>0}.
$$
Denote by $\sharp_r$ the $\sharp$ operator for $g_r$ as defined in Section \ref{sec:cons not}. 
In particular, we have 
$$
\FF 
= g_r \left(\FFsr (\cdot), \cdot \right)
= r^2 g \left(\FFsr (\cdot), \cdot \right)
\qquad
\Longleftrightarrow
\qquad
\FFsr = \frac{1}{r^2} \FFs. 
$$
Set 
$$
G^r_\n = \id_{TX} - \FFsr \circ \FFsr = \id_{TX} - \frac{1}{r^4} \FFs \circ \FFs. 
$$
Since the Levi-Civita connection of $g_r$ agrees with that of $g$, 
the defining equation of minimal connections with respect to $g_r$ is given by 
$$
\delta_\n^r \FF 
:= - \frac{1}{r^2} i \left(\left(G^r_\n\right)^{-1} (e_i) \right) D_i \FF 
=- \frac{1}{r^2} i \left(\left( \id_{TX} - \frac{1}{r^4} \FFs \circ \FFs \right)^{-1} (e_i) \right) D_i \FF 
=0. 
$$
Note that $\{ e_i \}$ is a local orthonormal frame for $g$. 
Thus, formally taking the "large radius limit", 
which means the leading behaviour of $\cF_r(\n)$ as $r \to \infty$, we obtain 
$$
d^* \FF=0. 
$$
This is exactly the defining equation of Yang--Mills connections. 
Thus it is natural to expect that minimal connections for a sufficiently large metric  
will behave like Yang--Mills connections. 
\end{remark}

Using Remark \ref{rem:large radius limit of min conn}, 
we can show the following existence theorem (Theorem \ref{thm:exist minimal conn}). 
We first note the following. 

\begin{lemma} \label{lem:ex YM conn}
Let $(X, g)$ be a compact oriented Riemannian manifold 
and $L \to X$ be a smooth complex line bundle with a Hermitian metric $h$. 
Then there is a unique Yang--Mills connection up to the addition of closed 1-forms. 
\end{lemma}

\begin{proof}
For any $\n \in \cA_0$, we have $d \FF=0$. 
So it defines a cohomology class 
$[\FF] \in H^2(X, \rl)$, which is known to be equal to $-2 \pi c_1(L)$. 
Then there exists a 1-form $b \in \Omega^1$ such that $\FF + db$ is harmonic 
by Hodge theory. Since 
$$
E_{\n + \i b \cdot \id_L} = \FF + d b, 
$$
$\n + \i b \cdot \id_L$ is a Yang--Mills connection. 

If $\n'=\n +  \i (b + b') \cdot \id_L$ for $b' \in \Omega^1$ 
is also a Yang--Mills connection, we have 
$0= d^* E_{\n'} = d^* d b'$, 
which is equivalent to $d b'=0$. 
\end{proof}

Using this, we can show the following.

\begin{theorem} \label{thm:exist minimal conn}
Let $(X, g)$ be a compact oriented Riemannian manifold 
and $L \to X$ be a smooth complex line bundle with a Hermitian metric $h$. 
Then there exists a minimal connection with respect to $g_r$ for sufficiently large $r>0$. 
\end{theorem}

\begin{proof}
Define a map $\cF :[0,1] \times \cA_0 \to d^* \Omega^2$ by 
\begin{align}
\cF (s,\n) = - 
\left(\det \widetilde G^s_\n \right)^{1/4} 
\left(\left( \widetilde G^s_\n \right)^{-1} \right)^* 
i \left(\left( \widetilde G^s_\n \right)^{-1} (e_i) \right) D_i \FF, 
\end{align}
where 
$$
\widetilde G^s_\n:= \id_{TX} - s^4 \FFs \circ \FFs. 
$$
Note that $\cF (0, \cdot)^{-1} (0)$ 
is the set of Yang--Mills connections with respect to $g$ 
and 
$\cF (s, \cdot)^{-1} (0)$ for $s \neq 0$  
is the set of minimal connections with respect to $g_{1/s}$. 

The image of $\cF$ is contained in $d^* \Omega^2$ by the following reason. 
If $s=0$, we have $\cF (0,\n)= d^* E_\n \in d^* \Omega^2$. 
For $s \neq 0$, we have 
$$
\cF (s,\n) = - \frac{1}{s^2} H^{1/s} (\n)  
$$
where $H^{1/s} (\n)$ is the mean curvature for $g_{1/s}$ as defined by \eqref{eq:MC}. 
By \eqref{eq:mc dstar}, we have 
$$
\cF (s,\n) = - \frac{1}{s^2} H^{1/s} (\n) \in d^{*_{1/s}} \Omega^2 = d^* \Omega^2, 
$$
where $d^{*_{1/s}}$ is the codifferential with respect to $g_{1/s}$.

We want to apply the implicit function theorem to show the statement. 
Fix a Yang--Mills connection $\n_0 \in \cF (0, \cdot)^{-1} (0)$, 
whose existence is guaranteed by Lemma \ref{lem:ex YM conn}. 
Denote by the linearization $(d \cF)_{(0,\n_0)}: \rl \oplus \i \Omega^1 \to d^* \Omega^2$ 
of $\cF$ at $(0,\n_0)$. Then we have  
$$
(d \cF)_{(0,\n_0)} (0, \i b) = - d^* db. 
$$
By the Hodge decomposition, we have 
$\rl \oplus \i \Omega^1 = \rl \oplus \i (Z^1 \oplus d^* \Omega^2)$, 
where $Z^1$ is the space of closed 1-forms. 
Then we see that 
$(d \cF)_{(0,\n_0)}|_{\i d^* \Omega^2}: \i d^* \Omega^2 \to d^* \Omega^2$ is an isomorphism 
and $\ker (d \cF)_{(0,\n_0)} = \rl \oplus \i Z^1$. 
Hence, we can apply the implicit function theorem (after the Banach completion)
and we see that $\cF (s, \cdot)^{-1} (0) \neq \emptyset$ for sufficiently small $s$.

Finally, we explain how to recover the regularity of elements in
$\cF (s, \cdot)^{-1} (0)$ after the Banach completion. 
Since the curvature is invariant under the addition of closed 1-forms, 
there exists $a_s \in \Omega^1$ such that 
\begin{align} 
\cF (s, \n_0 + \i a_s \cdot \id_L)=0, \qquad d^* a_s=0 \tag{$*_s$}
\end{align}
for sufficiently small $s$. 
In particular, 
$(*_0)$ is given by 
$d^* d a_0 = d^* a_0=0$, which is an overdetermined elliptic equation. 
To be  overdetermined elliptic is an open condition, so we see 
that $(*_s)$ is also overdetermined elliptic for sufficiently small $s$. 
Hence we can find a smooth element in $\cF (s, \cdot)^{-1} (0)$ around $(0,\n_0)$ 
and the proof is completed. 
\end{proof}

In \cites{kawai2020deformation, kawai2021deformationSpin(7)}, 
we considered deformation theory of $G_2, \Spin (7)$-dDT connections, 
whose definitions are given in the next subsection. 
Similarly, we will be able to consider the deformation theory 
of minimal connections. 
As in the case of minimal submanifolds, we expect that 
the space of infinitesimal deformations of minimal connections 
is finite dimensional but infinitesimal deformations are 
obstructed in general.

%%%%%%%%%%%%%%%%%%%%%%%%%%%%%%%%%%%%%%%%%%%%%%%%%%%%%%%%%%%%%%%%%%%%%%%%%%%%%%
\subsection{DDT and dHYM connections are minimal} \label{sec:ddt dhym min}

Deformed Donaldson--Thomas connections for a $G_2$-manifold 
($G_2$-dDT connections), 
deformed Donaldson--Thomas connections for a Spin(7)-manifold 
(${\rm Spin}(7)$-dDT connections), 
deformed Hermitian Yang--Mills (dHYM) connections 
are introduced in \cites{LYZ2000, lee2009geometric, kawai2021FM} 
as the ``mirrors" of calibrated submanifolds 
in $G_2$-, ${\rm Spin}(7)$-, and Calabi--Yau manifolds, respectively. 

In this subsection, we show that each of them is minimal, 
just as calibrated submanifolds are minimal submanifolds. 
First, recall the definition of these connections. 
For the definition of $G_2$- or ${\rm Spin} (7)$-structures, 
see for example \cite[Sections 2.2 and 2.3]{kawai2021mirror}.

\begin{definition} \label{def:G2dDT}
Let $X$ be a 7-manifold with a $G_2$-structure $\varphi \in \Omega^3$ 
and $L \to X$ be a smooth complex line bundle with a Hermitian metric $h$. 
A Hermitian connection $\nabla$ of $(L,h)$ satisfying 
\begin{align}\label{eq:def G2dDT}
-\frac{1}{6} \FFth + \FF \wedge * \varphi=0 
\end{align}
is called a deformed Donaldson--Thomas (dDT) connection 
for a manifold with a $G_2$-structure (a {\bf $G_2$-dDT connection}). 
\end{definition}

\begin{definition} \label{def:Spin7dDT}
Let $X$ be an 8-manifold with a ${\rm Spin}(7)$-structure $\Phi \in \Omega^4$ 
and $L \to X$ be a smooth complex line bundle with a Hermitian metric $h$.
Denote by $\Omega^k_\ell \subset \Omega^k$ 
the subspace of the space of $k$-forms corresponding to 
the $\ell$-dimensional irreducible representation of ${\rm Spin}(7)$. 
Let $\pi^k_\ell : \Omega^k \rightarrow \Omega^k_\ell$ be the projection. 
A Hermitian connection $\nabla$ of $(L,h)$ satisfying 
\begin{align}\label{eq:def Spin(7)dDT}
\pi^2_7 \left( \FF -* \left( \frac{\FFth}{6} \right)\right)=0, \qquad
\pi^4_7 (\FFtw)=0 
\end{align}
is called a deformed Donaldson--Thomas connection 
for a manifold with a ${\rm Spin}(7)$-structure (a {\bf ${\rm Spin}(7)$-dDT connection}). 
\end{definition}

\begin{definition} \label{def:dHYM}
Let $X$ be a K\"ahler manifold of (complex) dimension $n$ with a K\"ahler form $\omega$ 
and $L \to X$ be a smooth complex line bundle with a Hermitian metric $h$. 
Denote by $\pi^{2,0}:\Omega^2 \to \Omega^{2,0}$ the projection 
onto the space of $(0,2)$-forms. Fix $\theta \in \rl$. 
A Hermitian connection $\nabla$ of $(L,h)$ satisfying 
\begin{align}\label{eq:def dHYM}
\pi^{2,0} (\FF)=0, \qquad 
\ \mathop{\mathrm{Im}}\left(e^{-\i \theta} (\omega + \i \FF)^n\right)
=0 
\end{align}
is called a {\bf deformed Hermitian Yang--Mills (dHYM) connection} 
with phase $e^{\sqrt{-1}\theta}$. 
\end{definition}

When a (torsion-free) $G_2$- or ${\rm Spin}(7)$-manifold 
is compact without boundary, 
a dDT connection is a global minimizer of $V$ 
by the ``mirror" of associator or Cayley equality, 
and hence it is minimal. 
(More generally, this statement also holds when the $G_2$-structure is only closed.) 
See \cite[Theorems 4.2, 5.1, Corollaries 4.5, 5.4]{kawai2021mirror}. 
Similarly, 
a dHYM connection is a global minimizer of $V$ 
by the ``mirror" of special Lagrangian equality 
(\cite[Theorems 6.1, 6.2, Corollaries 6.5]{kawai2021mirror}) 
and hence it is minimal when a K\"ahler manifold is compact without boundary of 
(complex) dimension 3 or 4. 
We show that these connections are also minimal  
when a manifold is noncompact (and of any dimension in the dHYM case).

\begin{proposition} \label{prop:dDTismin}
$G_2$-dDT connections on a manifold with a closed $G_2$-structure $\varphi$ ($d \varphi=0$), 
${\rm Spin}(7)$-dDT connections on a (torsion-free) ${\rm Spin}(7)$-manifold 
and dHYM connections on a K\"ahler manifold 
are minimal. 
\end{proposition}

\begin{proof}
The $G_2$ and ${\rm Spin}(7)$ cases are proved by a slight modification of 
the proof in the compact case, just as in the case of calibrated submanifolds. 

Denote by $\varphi$ a $G_2$-structure with $d \varphi=0$. 
Let $g$ be the induced metric from $\varphi$. 
Use the notation in Section \ref{sec:cons not}. 
Let $\n \in \cA_0$ be a $G_2$-dDT connection. 
Since the minimality is a local property, we show that $\n$ is minimal on an open neighborhood $U$ of any point. 
We may assume that $V^0_U(\n):=\int_U (v(\n)-1) \vol_g < \infty$. 
Let $\widetilde \nabla$ be a compactly supported variation of $\n|_U$. That is, 
$\widetilde \nabla = \nabla|_U + \i a \cdot \id_L$ for a 1-form $a \in \Omega^1(U)$ 
such that $a =0$ outside a compact subset $K \subset U$. (So we have $a|_{\partial K} =0$.) 
We show that 
$V^0_U (\widetilde \n) \geq V^0_U (\n)$ if $|a|_{C^1}$ is sufficiently small, 
which implies that $\n$ is minimal. 

Note that 
\begin{align}
V^0_U (\widetilde \n) = \int_{U-K} (v(\n)-1) \vol_g + \int_K (v(\widetilde\n)-1) \vol_g. 
\end{align}
Since 
$|1-\la \FFtw, * \varphi \ra/2| = v(\n) \geq 1$ by \cite[Theorem 5.1]{kawai2021mirror}, 
we may assume that 
$1-\la \FFtw, * \varphi \ra/2 >0$. 
Then we may also assume that $1-\la E^2_{\widetilde \n}, * \varphi \ra/2 >0$
if $|a|_{C^1}$ is sufficiently small.  
By \cite[Theorem 5.1]{kawai2021mirror}, we have 
\begin{align}
\int_K (v(\widetilde\n)-1) \vol_g
\geq 
\int_K \left( \left|1- \frac{1}{2} \la E^2_{\widetilde \n}, * \varphi \ra \right| -1 \right) \vol_g 
= 
- \frac{1}{2} \int_K \la E^2_{\widetilde \n}, * \varphi \ra \vol_g
= 
- \frac{1}{2} \int_K E^2_{\widetilde \n} \wedge \varphi. 
\end{align}
Since $\widetilde \nabla = \nabla + \i a \cdot \id_L$, we have 
$E_{\widetilde \nabla}=\FF + da$. Then 
$$
E^2_{\widetilde \nabla} \wedge \varphi - \FFtw \wedge \varphi
= 
d \left( \left(2a \wedge \FF +a \wedge da \right) \wedge \varphi \right),  
$$
where we use $d \FF =0$ and $d \varphi=0$. 
This implies that 
$$
\int_K E^2_{\widetilde \n} \wedge \varphi 
=
\int_K \FFtw \wedge \varphi
$$
by the Stokes' theorem and $a|_{\partial K} =0$. 
Hence we obtain 
\begin{align}
V^0_U (\widetilde \n) \geq \int_{U-K} (v(\n)-1) \vol_g - \frac{1}{2} \int_K \FFtw \wedge \varphi 
= 
V^0_U (\n), 
\end{align}
which implies that $\n$ is minimal.

The same argument works for ${\rm Spin}(7)$-dDT connections and 
dHYM connections on 3- or 4-dimensional K\"ahler manifolds. \\

To show the K\"ahler case of any dimension, we give a direct computation. 
Denote by $\omega, g, J$ a K\"ahler form, a metric and a complex structure on a K\"ahler manifold. 
Let $\n \in \cA_0$ be a dHYM connection with phase $e^{\sqrt{-1}\theta}$. 
Define a metric $g_\n$ by 
$g_\n=(\id_{TX} +E_\n^\sharp)^* g$. 
Its associated 2-form $\omega_\n$ is given by $\omega_\n=(\id_{TX} +E_\n^\sharp)^* \omega$. 
Then $(\omega_\n, g_\n, J)$ defines a Hermitian structure on $X$. 
Denote by $*_\n$ the Hodge star operator induced by $g_\n$. 

Take a local orthonormal frame $\{ e_i \}$ with its dual $\{ e^i \}$ and 
we use the notation \eqref{eq:def deri}. 
Differentiating \eqref{eq:def dHYM} by $D_i$ and using $D \omega=0$, we have 
$$
\pi^{2,0} (D_i \FF)=0, \qquad 
\ \mathop{\mathrm{Im}}\left(e^{-\i \theta} (\omega + \i \FF)^{n-1} \wedge \sqrt{-1} D_i \FF \right)
=0.  
$$
By \cite[(4.3)]{kawai2020deformation}, this is equivalent to 
$$
\pi^{2,0} (D_i \FF)=0, \qquad 
\omega_\n^{n-1} \wedge D_i \FF =0.  
$$
Thus $D_i \FF$ is a primitive (1,1)-form with respect to the Hermitian structure $(\omega_\n, g_\n, J)$. 
Then by \cite[Proposition 1.2.31]{Huybrechts2005}, we obtain 
$$
*_\n (D_i \FF \wedge \omega_\nabla^{n-2})=-(n-2)! D_i \FF. 
$$
Then we have 
\begin{align}
(n-2)! \delta_\n E_\n 
=& -(n-2)! i(G_\n^{-1}(e_i)) D_i E_\n \\
=& i(G_\n^{-1}(e_i)) *_\n (D_i E_\n \wedge \omega_\nabla^{n-2})
=
*_\n \left( g_\n(G_\n^{-1}(e_i), \cdot) \wedge D_i E_\n \wedge \omega_\nabla^{n-2} \right). 
\end{align}
Since 
$g_\n
=g((\id_{TX} +E_\n^\sharp)(\cdot), (\id_{TX} +E_\n^\sharp)(\cdot))
=g(G_\n(\cdot), \cdot)
$, we have 
$g_\n(G_\n^{-1}(e_i), \cdot) =e^i$. 
Since $e^i \wedge D_i E_\n = d E_\n =0$, we see that $\delta_\n E_\n=0$. 
\end{proof}

\begin{remark}
We can also prove Proposition \ref{prop:dDTismin} by a direct computation 
in the $G_2$ and ${\rm Spin}(7)$ cases as follows. 
However, we need to assume that the $G_2$-structure is torsion-free. 

Denote by $\varphi$ a torsion-free $G_2$-structure on a $G_2$-manifold. 
Let $g$ be the induced metric from $\varphi$. 
Use the notation in Section \ref{sec:cons not}. 
Let $\n \in \cA_0$ be a $G_2$-dDT connection. 
Define a $G_2$-structure $\varphi_\n$ by 
$\varphi_\n=(\id_{TX} +E_\n^\sharp)^* \varphi$. 
Denote by $g_\n$ and $*_\n$ 
the Riemannian metric and the Hodge star operator induced by $\varphi_\n$. 
Differentiating \eqref{eq:def G2dDT} by $D_i$, we have 
$$
\left( - \frac{1}{2} E_\n^2 + * \varphi \right) \wedge D_i E_\n=0, 
$$
where we use $D *\varphi=0$. 
Then by \cite[Theorem C.1]{kawai2020deformation}, we see that 
$*_\n \varphi_\n \wedge D_i E_\n =0$, which is equivalent to 
$$
*_\n (D_i E_\n \wedge \varphi_\n) = - D_i E_\n. 
$$
Then we have 
\begin{align}
\delta_\n E_\n 
= -i(G_\n^{-1}(e_i)) D_i E_\n
= i(G_\n^{-1}(e_i)) *_\n (D_i E_\n \wedge \varphi_\n)
=
- *_\n \left( e^i \wedge D_i E_\n 
\wedge \varphi_\n \right)
=0. \\
\end{align}

The ${\rm Spin}(7)$ case is proved as follows. 
Denote by $\Phi$ a torsion-free ${\rm Spin}(7)$-structure on a ${\rm Spin}(7)$-manifold. 
Let $\n \in \cA_0$ be a ${\rm Spin}(7)$-dDT connection. 
Similarly, define a ${\rm Spin}(7)$-structure $\Phi_\n$ by 
$\Phi_\n=(\id_{TX} +E_\n^\sharp)^* \Phi$. 
Denote by $g_\n$ and $*_\n$ 
the Riemannian metric and the Hodge star operator induced by $\Phi_\n$. 
The ${\rm Spin}(7)$-structure $\Phi_\n$ induces 
the decomposition $\Lambda^k = \oplus \Lambda^k_{\ell, \n}$ 
of the space of $k$-forms. 
Note that 
$\Lambda^k_{\ell, \n} = (\id_{TX} +E_\n^\sharp)^* \Lambda^k_\ell$ 
by \cite[Section 2.4]{kawai2021deformationSpin(7)}. 

Since $D \Phi=0$, the Levi-Civita connection $D$ preserves 
the decomposition $\Lambda^k = \oplus \Lambda^k_\ell$. 
Then differentiating \eqref{eq:def Spin(7)dDT} by $D_i$, we have 
$$
\pi^2_7 \left( D_i E_\n -
* \left( \frac{E_\n^2}{2} \wedge D_i E_\n \right)\right)=0, \qquad
\pi^4_7 (E_\n \wedge D_i E_\n)=0. 
$$
By \cite[Theorem A.8 (2)]{kawai2021deformationSpin(7)}, we see that 
$D_i E_\n \in (\id_{TX} +E_\n^\sharp)^* \Lambda^2_{21} = \Lambda^2_{21, \n}$, 
which is equivalent to 
$$
*_\n (D_i E_\n \wedge \Phi_\n) = - D_i E_\n. 
$$
Then we have 
\begin{align}
\delta_\n E_\n 
= -i(G_\n^{-1}(e_i)) D_i E_\n
= i(G_\n^{-1}(e_i)) *_\n (D_i E_\n \wedge \Phi_\n)
=
*_\n \left(e^i \wedge D_i E_\n 
\wedge \Phi_\n \right)
=0. \\
\end{align}
\end{remark}

\subsection{$G_2$-dDT connections}
By Proposition \ref{prop:dDTismin} and Theorem \ref{thm:min imply cons}, 
results in Section \ref{sec:odd dim} also hold for $G_2$-dDT connections. 
In this subsection, we show that 
better estimates hold for $G_2$-dDT connections.

Set $V=\rl^7$. 
Let $\{ e_i \}_{i=1}^7$ be the standard basis of $V$ 
with its dual $\{ e^i \}_{i=1}^7$. 
Define a 3-form $\varphi \in \Lambda^3 V^*$ by  
\begin{equation} \label{varphi}
\varphi = e^{123} + e^{145} + e^{167} + e^{246} - e^{257} - e^{347} - e^{356},
\end{equation}
where $e^{i_1 \dots i_k}$ is short for $e^{i_1} \wedge \cdots \wedge e^{i_k}$. 
The stabilizer of the standard $\GL (V)$-action on $\Lambda^3 V^*$ at $\varphi$ is known 
to be the exceptional $14$-dimensional simple Lie group $G_2 \subset {\rm GL}(V)$. 
The group $G_2$ acts on $\Lambda^k V^*$ and it has the irreducible decomposition 
$$
\Lambda^k V^* = \bigoplus_\ell \Lambda^k_\ell V^*, 
$$
where $\Lambda^k_\ell V^* \subset \Lambda^k V^*$ 
is the $\ell$-dimensional irreducible subrepresentation of $G_2$. 
For more details, see for example \cite[Section 2.2]{kawai2021mirror}.

\begin{lemma}\label{lem: simplified_dDT}
For a $2$-form $\beta \in \Lambda^2 V^*$, the followings are equivalent. 

\begin{enumerate}
\item 
$\beta$ satisfies $-\beta^3/6+\beta\wedge\ast\varphi = 0$. 

\item 
$\beta$ is, up to the $G_2$-action, of the form
\begin{align}
\beta 
= a i(e_1)\varphi + \left( b_1 e^{23} + b_2 e^{45} + b_3 e^{67} \right) 
= (a+b_1)e^{23} + (a+b_2)e^{45} + (a+b_3)e^{67} 
\end{align} 
for $a, b_1, b_2, b_3 \in \rl$ satisfying $b_1+b_2+b_3=0$ and 
\begin{equation}\label{eq:ptwise}
\left(3-a^2+\frac{1}{2}\sum\limits_{j=1}^3 b_j^2\right)a = b_1b_2b_3.
\end{equation}

\item
$\beta$ is, up to the $G_2$-action, of the form 
\begin{align}
\beta = c_1 e^{23} + c_2 e^{45} + c_3 e^{67} 
\end{align} 
for $c_1, c_2, c_3 \in \rl$ satisfying 
\begin{equation}\label{eq: dDT_3}
c_1+c_2+c_3 = c_1c_2c_3. 
\end{equation}
\end{enumerate}

\end{lemma} 

\begin{proof}
We first show (1)$\Rightarrow$(2). 
Suppose that a $2$-form $\beta \in\Lambda^2 V^*$ 
satisfies $-\beta^3/6+\beta\wedge\ast\varphi = 0$. 
Decompose $\beta$ as 
$$
\beta=\beta_7 + \beta_{14} = i(u) \varphi + \beta_{14} \in \Lambda^2_7 V^* \oplus \Lambda^2_{14} V^* 
$$
for $u \in V$. 
Recall that every element in $\fg_2 \cong \Lambda^2_{14} V^*$ 
is ${\rm Ad}(G_2)$-conjugate to an element of a Cartan subalgebra. 
Then we may assume that 
$$
\beta_{14}=b_1 e^{23} + b_2 e^{45} + b_3 e^{67}
$$
with $b_j \in \rl$ and $b_1 +b_2 + b_3=0$. 
Then by \cite[Proposition D.1]{kawai2020deformation}, we have 
\begin{align} \label{eq:ptwise0}
\left( 3-|u|^2 + \frac{1}{2} \sum_{j=1}^3 b_j^2 \right) u^\flat = b_1 b_2 b_3 e^1. 
\end{align}

Then we discuss as in the proof of \cite[Lemma B.3]{kawai2020deformation}.
If $3-|u|^2 + \sum_{j=1}^3 b_j^2/2 \neq 0$, 
\eqref{eq:ptwise0} implies that $u$ is a constant multiple of $e_1$ 
and \eqref{eq:ptwise} must be satisfied. 

If $3-|u|^2 + \sum_{j=1}^3 b_j^2/2 = 0$, we have $b_1 b_2 b_3=0$ by \eqref{eq:ptwise0}. 
We may assume that 
$b_1=0$ and set $b = b_2 = - b_3$. 
Then $\beta_{14}=b (e^{45}-e^{67})$.

Suppose that $b \neq 0$. 
By \cite[Lemma C.2]{kawai2020deformation}, 
we have $i(u) \beta_{14}=0$, which implies that 
$u = u^1 e_1 + u^2 e_2 + u^3 e_3$ for $u^j \in \rl$. 
Now, define $j:V=\rl^7 \to \rl^7=V$ by 
$j (x^1,x^2,x^3,x^4,x^5,x^6,x^7) = (x^1,x^2,x^3,x^4,x^5,x^6, -x^7)$ 
and the ${\rm SU}(2)$-action $\rho: {\rm SU}(2) 
\to {\rm GL}(\rl^7) = {\rm GL}(\rl^3 \oplus \cx^2)$ by 
\[
\rho (g) = 
\left(
\begin{array}{cc}
\rho_- (g)  & 0 \\
0 & g \\
\end{array}
\right), 
\]
where $\rho_-:{\rm SU}(2) \to {\rm SO}(3)$ is the double cover. 
It is known that this ${\rm SU}(2)$-action preserves $j^* \varphi$. 
Thus, the ${\rm SU}(2)$-action $j \circ \rho \circ j$ 
preserves $\varphi$. 
Since $\beta_{14}= b (e^{45}-e^{67})$ is invariant 
under this ${\rm SU}(2)$-action, 
we may further assume that $u^2=u^3=0$. 
Then, the proof in this case is done. 

If $b=0$, we have $\beta_{14}=0$. 
Then we may assume that $u=|u| e_1$ and the proof in this case is done. 
These arguments yield (2). \\

Next, we show (2)$\Rightarrow$(3). Set $c_j=a+b_j$. Since 
$$
c_1c_2c_3 - (c_1+c_2+c_3) 
= (a+b_1)(a+b_2)(a+b_3)-3a
= a^3 + (b_1 b_2 + b_1 b_3 + b_2 b_3 -3) a + b_1 b_2 b_3 
$$
and $0=(b_1 + b_2 + b_3)^2 = \sum_{j=1}^3 b_j^2 + 2 (b_1 b_2 + b_1 b_3 + b_2 b_3)$, 
we see that \eqref{eq:ptwise} is equivalent to \eqref{eq: dDT_3}. 

Finally, we show (3)$\Rightarrow$(1). 
Since the equation $-\beta^3/6+\beta\wedge\ast\varphi = 0$ is $G_2$-invariant, 
we only have to show that $-\beta^3/6+\beta\wedge\ast\varphi = 0$ for 
$\beta = c_1 e^{23} + c_2 e^{45} + c_3 e^{67}$ satisfying \eqref{eq: dDT_3}. 
Since 
$\ast \varphi = e^{4567} + e^{2367} + e^{2345} + e^{1357} - e^{1346} - e^{1256} - e^{1247}$, 
we compute 
$$
-\frac{\beta^3}{6}+\beta\wedge\ast\varphi
= 
-c_1 c_2 c_3 e^{234567} + (c_1+c_2+c_3) e^{234567} =0 
$$
and the proof is completed. 
\end{proof}

By Lemma \ref{lem: simplified_dDT}, we obtain the following.

\begin{lemma} \label{lem:estimate dDT}
For any $2$-form $\beta \in \Lambda^2 V^*$ satisfying 
$-\beta^3/6+\beta \wedge\ast\varphi = 0$, we have 
\begin{align}
{\rm tr} (G_\beta^{-1}) \geq \frac{5}{2}
\qquad \mbox{and} \qquad
\frac{{\rm tr} (G_\beta^{-1}) v(\beta)-7}{v(\beta)-1} \geq \frac{13}{7}, 
\end{align}
where $G_\beta$ and $v(\beta)$ are defined as in \eqref{eq:def vol func} and \eqref{eq:Gn}, 
respectively. 
\end{lemma}

\begin{proof}
We may assume that $\beta = c_1 e^{23} + c_2 e^{45} + c_3 e^{67}$ satisfying \eqref{eq: dDT_3}. 
By \eqref{eq: dDT_3}, we have 
$(c_1 c_2 -1) c_3 = c_1 + c_2$. If $c_1 c_2 -1=0$, we have $c_1+c_2=0$. 
Then $0=c_1 c_2 -1=-c_1^2 -1<0$, which is impossible. 
So we have $c_1 c_2 -1 \neq 0$ and 
$$
c_3= \frac{c_1+c_2}{c_1 c_2 -1}. 
$$
Then 
$$
c_3^2+1 
= \frac{(c_1 + c_2)^2 + (c_1 c_2 -1)^2}{(c_1 c_2 -1)^2}
= \frac{c_1^2 c_2^2 + c_1^2 + c_2^2 +1}{(c_1 c_2 -1)^2}
= \frac{(c_1^2+1) (c_2^2+1)}{(c_1 c_2 -1)^2}
$$
Hence 
$$
{\rm tr} (G_\beta^{-1}) 
= 1+ \sum_{j=1}^3 \frac{2}{c_j^2+1} 
= 1 + \frac{2 \{ c_1^2+c_2^2+2 +(c_1 c_2-1)^2 \}}{(c_1^2+1) (c_2^2+1)}
= 3+\frac{4(-c_1 c_2 +1)}{(c_1^2+1) (c_2^2+1)}. 
$$
Now we can show that 
\begin{align} \label{eq:estimate dDT 1}
4(-c_1 c_2 +1)
= \frac{1}{2} 
\left\{  
- (c_1^2+1) (c_2^2+1) + (c_1-c_2)^2+(c_1 c_2 -3)^2
\right\}. 
\end{align}
Then it follows that 
$$
{\rm tr} (G_\beta^{-1}) 
= \frac{5}{2} + \frac{(c_1-c_2)^2+(c_1 c_2 -3)^2}{2 (c_1^2+1) (c_2^2+1)} \geq \frac{5}{2}. \\
$$

Next, we show the second inequality. 
We can compute 
\begin{align}
v(\beta) &= \sqrt{(c_1^2+1) (c_2^2+1) (c_3^2+1)} = \frac{(c_1^2+1) (c_2^2+1)}{|c_1 c_2-1|}, \\
{\rm tr} (G_\beta^{-1}) v(\beta)
&= 
\left( 3+\frac{4(-c_1 c_2 +1)}{(c_1^2+1) (c_2^2+1)} \right) 
\cdot 
\frac{(c_1^2+1) (c_2^2+1)}{|c_1 c_2-1|} 
= 
\frac{3 (c_1^2+1) (c_2^2+1)}{|c_1 c_2-1|} + \frac{4 (- c_1 c_2+1)}{|c_1 c_2-1|}.  
\end{align}
Thus when $c_1 c_2-1<0$, we have 
$({\rm tr} (G_\beta^{-1}) v(\beta)-7)/(v(\beta)-1) = 3$. 

Suppose that $c_1 c_2-1>0$. Then 
$$
\frac{{\rm tr} (G_\beta^{-1}) v(\beta)-7}{v(\beta)-1}
=  (j \circ f) (c_1,c_2)
$$
where 
$$
j (x) = \frac{3x-11}{x-1} = 3 - \frac{8}{x-1}, \qquad 
f (c_1,c_2)= \frac{(c_1^2+1) (c_2^2+1)}{c_1 c_2-1}. 
$$
By \eqref{eq:estimate dDT 1}, it follows that 
$$
\frac{1}{f (c_1,c_2)} 
=
\frac{(c_1^2+1) (c_2^2+1) - \{ (c_1-c_2)^2+(c_1 c_2 -3)^2 \} }{8 (c_1^2+1) (c_2^2+1)} 
\leq \frac{1}{8}. 
$$
Thus $f (c_1,c_2) \geq 8$ and 
the minimum is attained at $c_1=c_2=\pm \sqrt{3}$, which satisfies $c_1 c_2 -1 >0$. 
Since $j$ is monotonically increasing, we see that 
$$
\frac{{\rm tr} (G_\beta^{-1}) v(\beta)-7}{v(\beta)-1}
=  (j \circ f) (c_1,c_2)
\geq 
j (8) = \frac{13}{7}. 
$$
\end{proof}

Lemma \ref{lem:estimate dDT} gives pointwise estimates of 
${\rm tr} (G_\nabla^{-1})$ 
and $\frac{{\rm tr} (G_\nabla^{-1}) v(\nabla)-7}{v(\nabla)-1}$ 
for a $G_2$-dDT connection $\nabla$ for a manifold with a $G_2$-structure. 
These are pointwise estimates, so there is no need for conditions 
on a $G_2$-structure like being torsion-free. 
By Proposition \ref{prop:dDTismin}, Theorem \ref{thm:min imply cons}, 
Lemma \ref{lem:estimate dDT} and Corollary \ref{cor:mono vol sh}, 
we have the following for the volume functional. 

\begin{corollary} \label{cor:mono vol G2dDT} 
Let $X$ be a 7-manifold with a closed $G_2$-structure $\varphi$ ($d \varphi =0$) 
and $L \to X$ be a smooth complex line bundle with a Hermitian metric $h$. 
Fix $p \in X$ and use the notation in Section \ref{sec:mono notation}. 
Then there exists 
a constant $a=a(n,p,g) \geq 0$ such that 
for a $G_2$-dDT connection $\n$ of $(L,h)$ 
and a nonnegative smooth function $f$ on $X$ satisfying 
$
-\Delta_\n f \geq 0, 
$
$$
(0,r_p] \to \rl, \qquad 
\rho \mapsto \frac{e^{a \rho^2}}{\rho^{5/2}} \int_{B_\rho (p)} 
f v (\n)  \vol_g
$$
is non-decreasing. 
In particular, setting $f=1$, we see that 
$$
(0,r_p] \to \rl, \qquad 
\rho \mapsto \frac{e^{a \rho^2}}{\rho^{5/2}} \int_{B_\rho (p)} v (\n)  \vol_g
$$
is non-decreasing. 
\end{corollary}

By Lemmas \ref{lem:estimate dDT}, \ref{lem:cond assump} and Corollary \ref{cor:mono nvol sh}, 
we have the following result for 
the normalized volume functional. 
Recall that (torsion-free) $G_2$-manifolds are Ricci-flat. 

\begin{proposition} \label{prop:mono nvol G2dDT}
Let $X$ be a (torsion-free) $G_2$-manifold 
and $L \to X$ be a smooth complex line bundle with a Hermitian metric $h$. 
Fix $p \in X$ and use the notation in Section \ref{sec:mono notation}. 
Take 
a constant $a=a(n,p,g) \geq 0$ as in Corollary \ref{cor:mono vol G2dDT}. 
Define $\Theta: [0, \infty) \to [0, \infty)$ by 
$$
\Theta (\tau)
=
\omega_7 \int_0^\tau \frac{e^{a \zeta^2} \zeta^7}{\zeta^{\frac{13}{7}-1}} d \zeta
=
\frac{16 \pi^3}{105} \int_0^\tau e^{a \zeta^2} \zeta^{\frac{43}{7}} d \zeta. 
$$
Then there exists $0 < r'_p \leq r_p$ such that 
for any $G_2$-dDT connection $\n$ of $(L,h)$ 
$$
(0,r'_p] \to \rl, \qquad 
\rho \mapsto \frac{e^{a \rho^2}}{\rho^{13/7}} \int_{B_\rho (p)} (v (\n)-1) \vol_g 
+ 2a \Theta (\rho) 
$$
is non-decreasing. 
\end{proposition}

When $X$ is $\rl^7$ with the standard $G_2$-structure inducing the standard flat metric, 
we can take $a=0$ and $r_p'=\infty$. 
Hence Proposition \ref{prop:mono nvol G2dDT} gives 
the monotonicity formula for the normalized volume 
for $G_2$-dDT connections on $\rl^7$. 
As an application, we can prove the following vanishing theorem 
stronger than Corollary \ref{cor:vanish min}. 
The proof is the same as in Corollary \ref{cor:vanish}.

\begin{corollary} \label{cor:vanish G2dDT}
Let $(L, h) \to \rl^7$ be a (necessarily trivial) smooth complex Hermitian line bundle 
over $\rl^7$ with the standard $G_2$-structure inducing the standard flat metric $g_0$. 
If $\n$ is a $G_2$-dDT connection satisfying 
$$
\int_{B_r(p)} (v(\n)-1) \vol_{g_0} = o(r^{13/7}) 
\qquad \mbox{as} \qquad r \to \infty 
$$
for some $p \in \rl^7$, 
then $\n$ is flat. 

In particular, if $\n$ is a $G_2$-dDT connection and 
$
\int_{\rl^7} (v(\n)-1) \vol_{g_0} < \infty,  
$
then $\n$ is flat. 
\end{corollary}

\begin{remark}
A similar statement as in Lemma \ref{lem: simplified_dDT} also holds 
in the ${\rm Spin}(7)$ case. 
See \cite[Corollary A.6 and (A.26)]{kawai2021deformationSpin(7)}.
However, the same computation as in Lemma \ref{lem:estimate dDT} 
only yields that 
${\rm tr} (G_\n^{-1}) \geq 0$ and
$({\rm tr} (G_\n^{-1}) v(\n)-8)/(v(\n)-1) \geq 0$
for a ${\rm Spin}(7)$-dDT connection $\n$. 
Thus we do not obtain better estimates in the ${\rm Spin}(7)$ case 
by the same computation. 
\end{remark}

\appendix
\section{Correspondence by the real Fourier--Mukai transform}
\label{sec:FM}

In this appendix, we give the ``mirror" correspondence between geometric objects 
by the real Fourier--Mukai transform. 

Let $B\subset\mathbb{R}^p$ be an open set with coordinates $(x^{1},\cdots,x^{p})$ and 
$f=(f^{p+1},\cdots,f^{p+q}):B\to T^{q}$ be a smooth map with values in $T^{q}$. 
We use coordinates $(y^{p+1},\cdots,y^{p+q})$ for $T^{q}$. 
Throughout this appendix, we use indices $i, j, k, \ell$ to run over $1, \cdots, p$, 
whereas indices $a, b, c$ range over $p+1, \cdots, p+q$. 
The graph of $f$ is denoted by 
\[S:=\{\,(x,f(x))\mid x\in B\,\}, \]
which is a $p$-dimensional submanifold in $X:=B\times T^{q}$. 

Then we can consider the real Fourier--Mukai transform of $S$, 
which was first introduced by \cite{LYZ2000}. See also \cite[Section 2]{kawai2021FM}. 
The idea is as follows. 
Set 
$$
T^q=\rl^q/2 \pi \Z^q, \qquad 
(T^q)^* =(\rl^q)^*/2 \pi (\Z^q)^*, \qquad 
X^*=B \times (T^q)^*, 
$$
where $(\rl^q)^*$ is the dual space of $\rl^q$ 
and 
$(\Z^q)^* = \{ \alpha \in (\rl^q)^* \mid \la \alpha, v\ra \in \Z 
\mbox{ for any } v \in \Z^q \}$. 
Then we identify  
$T^q$ with $H^1((T^q)^*, \rl)/ 2 \pi H^1((T^q)^*, \Z)$.  
Furthermore, this space can also be identified with the space 
$\cM_{\rm flat}$ 
of 
flat Hermitian connections of the trivial complex line bundle 
$(T^q)^* \times \cx \to (T^q)^*$ modulo unitary gauge transformations. 
See for example the proof of \cite[Corollary 4.6 (2)]{kawai2021mirror}. 
Explicitly, this identification is given by 
\begin{align}
T^q \to \cM_{\rm flat}, \qquad
z:=[z^{p+1}, \cdots, z^{p+q}] 
\mapsto \n^z:=\left[ d + \i z^a dy^a \right], 
\end{align}
where $d$ is the standard flat connection of $(T^q)^* \times \cx \to (T^q)^*$ 
and we also use $(y^{p+1}, \cdots, y^{p+q})$ for coordinates on $(T^q)^*$. 

For a map $f:B \to T^q$ and $x \in B$, 
$f(x) \in T^q$ so it defines $\n^{f(x)} \in \cM_{\rm flat}$. 
Then we see that 
the family of $\{ \nabla^{f(x)} \}_{x \in B}$ 
gives a Hermitian connection 
\[\nabla:=d+\sqrt{-1} f^a dy^a \]
of the trivial complex line bundle over $X^{*} \cong X$. 
This $\nabla$ is called the real Fourier--Mukai transform of $S$. 
Its curvature 2-form $F_\nabla$ is given by 
\begin{equation} \label{eq:FM curv}
F_\n = \i \frac{\partial f^{a}}{\partial x^{i}} dx^i \wedge dy^a. 
\end{equation}
Note that $\n$ is determined up to the action of the group of gauge transformations, 
but $F_\n$ is independent of it. 
Use the notation in Section \ref{sec:cons not} and set 
\begin{align}
\FF:=& -\i F_\n, \qquad
\partial_{i}:=\frac{\partial}{\partial x^i}, \qquad 
\partial_a:=\frac{\partial}{\partial y^a}, \qquad 
f^a_i:=\frac{\partial f^a}{\partial x^i}, \qquad
f^a_{ij}:=\frac{\partial^2 f^a}{\partial x^i \partial x^j}, \\
v_i:=& (\id_{TX} + \FFs)(\partial_i) 
= \partial_i + f^a_i \partial_a, \qquad
\eta_a:=
(\id_{TX} + \FFs)(\partial_a) 
= \partial_a - f^a_j \partial_j 
\end{align}
for $1 \leq i,j \leq p$ and $p+1 \leq a \leq p+q$. 
Note that 
$\{ v_i \}_{i=1}^p$ spans $TS$ and 
$\{ \eta_a \}_{a=p+1}^{p+q}$ spans the orthogonal complement $T^\perp S$, 
where $X$ is endowed with the standard flat metric $\la \cdot, \cdot \ra$. 
The induced metric $g=(g_{ij})$ on $S$ is given by 
\begin{align} \label{eq:FM metric}
g_{i j} = \la v_i, v_j \ra = \delta_{ij} + f^a_i f^a_j. 
\end{align}
Set $g^{-1}=(g^{ij})$. 
Define $G_\n$ as in \eqref{eq:nabla Gnabla}. 
As in \eqref{eq:del nabla}, 
we also define a differential operator $\delta_\n: \Omega^k \to \Omega^{k-1}$ by 
$$
\delta_\n \alpha 
:= \delta_{E_\n} \alpha
= 
- i(G_\n^{-1} (\partial_i)) D_{\partial_i} \alpha 
- i(G_\n^{-1} (\partial_a)) D_{\partial_a} \alpha. 
$$

\begin{lemma}\label{lem:FM min}
The graph $f$ is minimal if and only if $\delta_\n \FF=0$. 
\end{lemma}
Thus minimal (graphical) submanifolds correspond to minimal connections via 
the real Fourier--Mukai transform. 
In this sense, minimal connections are considered as 
``mirrors" of minimal submanifolds.

\begin{proof}
Define $\iota:B \to X= B \times T^q$ by 
$$
\iota (x) = (x, f(x)). 
$$
Then the graph $f$ is minimal if and only if 
the mean curvature 
$$
H:= g^{ij} \left( D_{\partial_i}^{\iota^*TX} (\iota_* (\partial_j)) \right)^\perp 
$$
vanishes, where $D^{\iota^*TX}$ is the induced connection on 
the pullback $\iota^*TX$ from the Levi-Civita connection $D$ 
of $\la \cdot, \cdot \ra$ by $\iota$, 
$T^\perp S$ is the normal bundle of $S$ 
and $\perp:\iota^*TX=TS \oplus T^\perp S \to T^\perp S$ is the projection. 
Since 
$$
D_{\partial_i}^{\iota^*TX} (\iota_* (\partial_j))
= 
D_{\partial_i}^{\iota^*TX} \left(\partial_j + f^a_j \partial_a \right)
= 
f^a_{i j} \partial_a, 
$$
we have
$$
\left\la D_{\partial_i}^{\iota^*TX} (\iota_* (\partial_j)), \eta_a \right \ra 
= f^a_{i j}
$$
for $p+1 \leq a \leq p+q$. Since 
$\{ \eta_a \}_{a=p+1}^{p+q}$ spans $T^\perp S$, it follows that 
$H=0$ if and only if 
\begin{align} \label{eq:FM H=0}
g^{ij} f^a_{i j} = 0 \qquad \mbox{for any } p+1 \leq a \leq p+q. 
\end{align}

Next, we compute $\delta_\n \FF$. 
Since 
$$
\FFs= f^a_i (dx^i \otimes \partial_a -dy^a \otimes \partial_i), 
$$
we have 
$$
G_\n=\id_{TX}-(\FFs)^2
=
(\delta_{i j} + f^a_i f^a_j) dx^i \otimes \partial_j
+
(\delta_{a b} + f^a_i f^b_i) dy^a \otimes \partial_b. 
$$
Then it follows that 
$G_\n^{-1}$ is a linear combination of $dx^i \otimes \partial_j$ and 
$dy^a \otimes \partial_b$. 
Using this, we compute 
\begin{align}
\delta_\n \FF
= 
- i(G_\n^{-1}(\partial_i)) D_{\partial_i} (f^a_j dx^j \wedge dy^a)
=
- f^a_{ij} dx^j(G_\n^{-1} (\partial_i)) dy^a, 
\end{align}
where we use the fact that $f^a$ is a function of $(x^1,\cdots, x^p)$, 
so $D_{\partial_b} (f^a_j dx^j \wedge dy^a) =0$. 
Since $g_{ij}=\la v_i, v_j \ra= \la G_\n (\partial_i), \partial_j \ra$, 
we have 
\begin{align} \label{eq:FM Gn-1}
dx^j(G_\n^{-1} (\partial_i)) = \la G_\n^{-1} (\partial_i), \partial_j \ra = g^{ij}, 
\end{align}
and hence 
\begin{align} \label{eq:FM delE}
\delta_\n \FF = - g^{ij} f^a_{ij} dy^a. 
\end{align}
Then by \eqref{eq:FM H=0} and \eqref{eq:FM delE}, the proof is completed. 
\end{proof}

\begin{lemma}
The Christoffel symbol $\Gamma^k_{i j}$ of $g$ is given by 
\begin{align} \label{eq:FM chris}
\Gamma^k_{i j} = g^{k \ell}f^a_{ij} f^a_\ell. 
\end{align}

Let $D^g$ the Levi-Civita connection of the induced metric $g$. 
For any 1-form $\alpha \in \Omega^1(B)$ on $B$, we have 
\begin{align} \label{eq:FM conn}
D^g_{\partial_j} \alpha = D_{\partial_j} \alpha 
- \left\la (G_\n^{-1})^* \alpha, df^a \right \ra df^a_j. 
\end{align}
\end{lemma}

\begin{proof}
Recall that 
$$
\Gamma^k_{i j}=
\frac{1}{2} g^{k \ell} 
\left( \frac{\partial g_{j \ell}}{\partial x^i} + \frac{\partial g_{i \ell}}{\partial x^j} 
- \frac{\partial g_{i j}}{\partial x^\ell} \right). 
$$
By \eqref{eq:FM metric}, we have 
$$
\frac{\partial g_{i j}}{\partial x^k}=
\frac{\partial}{\partial x^k} (f^a_i f^a_j)
=
f^a_{i k} f^a_j + f^a_i f^a_{j k}. 
$$
Thus 
\begin{align}
\frac{\partial g_{j \ell}}{\partial x^i} + \frac{\partial g_{i \ell}}{\partial x^j} 
- \frac{\partial g_{i j}}{\partial x^\ell}
=
f^a_{i j} f^a_\ell + f^a_j f^a_{i \ell}
+ f^a_{i j} f^a_\ell + f^a_i f^a_{j \ell}
- f^a_{i \ell} f^a_j - f^a_i f^a_{j \ell}
= 
2 f^a_{i j} f^a_\ell. 
\end{align}
Hence we obtain \eqref{eq:FM chris}.

Set $\alpha=\alpha_k dx^k$. 
Then by \eqref{eq:FM chris}, we compute 
\begin{align}
D^g_{\partial_j} \alpha
= D^g_{\partial_j} (\alpha_k dx^k) 
=
\frac{\partial \alpha_k}{\partial x^j} dx^k 
- \alpha_k \Gamma^k_{ji} dx^i 
=
D_{\partial_j} \alpha
- 
\alpha_k g^{k \ell} f_{ij}^a f_\ell^a dx^i. 
\end{align}
By \eqref{eq:FM Gn-1}, we have $g^{k \ell} = \la (G_\n^{-1})^* dx^k, dx^\ell \ra$. 
Then it follows that 
$$
\alpha_k g^{k \ell} f_{ij}^a f_\ell^a dx^i
= 
\left\la (G_\n^{-1})^* \alpha, df^a \right \ra df^a_j
$$
and the proof is completed. 
\end{proof}

Let $d^{*_g}$ be the codifferential with respect to $g$. Then we have the following. 
\begin{proposition} \label{prop:FM codiff}
For any 1-form $\alpha \in \Omega^1(B)$ on $B$, we have 
\begin{align}
d^{*_g} \alpha 
= \delta_\nabla \alpha 
+ i \left((G_\n^{-1} \circ \FFs) \left((\delta_\n E_\n)^\sharp \right) \right) \alpha. 
\end{align}
In particular, if the graph $f$ is minimal, we have 
$$
d^{*_g} \alpha = \delta_\nabla \alpha. 
$$
\end{proposition}

In other words, we have $d^{*_g}=\delta'_{\FF}:\Omega^1(B) \to \Omega^0(B)$, 
where $\delta'_\beta$ for a 2-form $\beta$ is defined in Lemma \ref{lem:ivp gen}. 
See also Proposition \ref{prop:div set}. 
In this sense, $\delta'_\beta: \Omega^1 \to \Omega^0$ can be considered as the 
``mirror" of the codifferential $\Omega^1 \to \Omega^0$.

\begin{proof}
By \eqref{eq:FM conn}, \eqref{eq:FM Gn-1} and \eqref{eq:FM delE}, we have 
\begin{align}
d^{*_g} \alpha
=
- g^{ij} i(\partial_i) D_{\partial_j}^g \alpha 
=
\delta_\n \alpha 
+ 
g^{ij} \left\la (G_\n^{-1})^* \alpha, df^a \right \ra f^a_{ij} 
=
\delta_\n \alpha 
- 
\la \delta_\n \FF, dy^a \ra 
\left\la (G_\n^{-1})^* \alpha, df^a \right \ra. 
\end{align}
We also compute 
\begin{align}
\la \delta_\n \FF, dy^a \ra df^a
=
\la \delta_\n \FF, f^a_i dy^a \ra dx^i 
=& 
\la \delta_\n \FF, i(\partial_i) \FF \ra dx^i \\
=& 
\FF (\partial_i, (\delta_\n \FF)^\sharp) dx^i
=
- \left(\FFs \left((\delta_\n E_\n)^\sharp \right) \right)^\flat. 
\end{align}
These equations imply the statement. 
\end{proof}

More generally, we can also consider the real Fourier--Mukai transform of cycles, 
i.e., pairs of submanifolds and connections on them as in \cite{kawai2021FM},  
but we could not find nice geometric objects on the cycle side 
corresponding to minimal connections. 
In this case, there is a contribution from the connection side 
to the value of $G_\n$, 
and the computation becomes much more complicated unlike the proof of Lemma \ref{lem:FM min}. 
Even if the calculation goes well somehow, it is another matter whether we can find a geometric structure independent of the torus fibration structure.

%\end{CJK}

%===============================================================================

\bibliography{references}

%===============================================================================

\end{document}